\documentclass{article}

\usepackage{lineno}
\usepackage{verbatim,amsmath,amsfonts,epsfig,graphics,setspace,amssymb, amsbsy, float,stmaryrd, multirow}
\usepackage{caption}
\usepackage{epsfig}
\usepackage{enumerate,color}
\usepackage{url}
\usepackage{algorithm}
\usepackage{algorithmic}
\usepackage{amsthm}
\usepackage{tikz}

\newcounter{rmrk}[section]
\newenvironment{rem}{\stepcounter{rmrk}
\noindent {\newline \bf Remark
\arabic{section}.\arabic{rmrk}.\ }}{\newline}

\newtheorem{lemma}{Lemma}[section]
\newtheorem{corollary}{Corollary}[section]
\newtheorem{theorem}{Theorem}[section]

\newcommand{\R}{\mathbb{R}}

\newcommand{\p}{\mathbb{P}}

\newcommand{\q}{\mathbb{Q}}

\newcommand{\B}{\mathcal{B}}

\newcommand{\DG}{\text{DG}}

\newcommand{\pspace}{\Lambda}
\newcommand{\dspace}{\mathcal{D}}

\newcommand{\pborel}{\mathcal{B}_{\pspace}}

\newcommand{\dborel}{\mathcal{B}_{\dspace}}

\newcommand{\initmeas}{P_{\text{init}}}
\newcommand{\initmeasomega}{P_{\text{init}, \omega}}
\newcommand{\updatemeas}{P_{\text{up}}}
\newcommand{\initdens}{\pi_{\text{init}}}
\newcommand{\updatedens}{\pi_{\text{up}}}
\newcommand{\predictmeas}{P_\text{pred}}
\newcommand{\predictdens}{\pi_\text{pred}}
\newcommand{\obsmeas}{P_{\text{obs}}}
\newcommand{\obsdens}{\pi_{\text{obs}}}
\newcommand{\obsmeasi}{P_{\text{obs}, i}}

\newcommand{\expnumber}[2]{{#1}\mathrm{E}{#2}}

\DeclareMathOperator*{\argmin}{arg\,min}

\title{Iterative Data-Consistent Inversion with Multiple Push-forward Constraints}
\author{Tianyi Jiang\thanks{Department of Statistics, Colorado State University, Fort Collins, CO 80523} \and
Troy Butler\thanks{Department of Mathematical and Statistical Sciences, University of Colorado Denver, Denver, CO 80202 ({\tt Troy.Butler@ucdenver.edu})}
  \and
    Timothy Wildey\thanks{Optimization and Uncertainty Quantification Department, Center for Computing Research, Sandia National Labs, Albuquerque, NM 87185.}
    \and
    Tim Kutta\thanks{Department of Mathematics, Aarhus University, Denmark}
    \and 
    Haonan Wang\thanks{Department of Statistics, Colorado State University, Fort Collins, CO 80523}}
\date{{\normalsize This Draft: \today}}

\begin{document}

\maketitle

\begin{abstract}
A foundational challenge in uncertainty quantification involves estimating a probability measure on the space of uncertain parameters such that its push-forward through a computational model matches an observed probability measure on the output data associated with quantities of interest (QoI). When multiple, distinct sets of observational data are available, the desired parameter measure should simultaneously satisfy multiple push-forward constraints associated with various subsets of the QoI. In this work, we present a convergent measure-theoretic framework for solving this problem based on an iterative application of Data-Consistent Inversion (DCI). We first rigorously establish the theoretical optimality of the DCI solution to the standard problem, proving that it minimizes the $f$-divergence over the space of all possible pullback measures that satisfy the push-forward constraint. This optimality property provides the foundation for our iterative DCI scheme, which is shown to converge to a solution of the multiple push-forward constraint problem. This iterative solution minimizes the cumulative $f$-divergence across all constraints and, under uniform initializations, represents the maximal entropy solution (the I-projection) onto the intersection of the solution sets. We provide a rigorous convergence analysis for the proposed method and demonstrate its practical utility through numerical examples, including a high-dimensional parameter space governed by partial differential equations, where the iterative approach robustly avoids the complexities associated with approximating high-dimensional joint observed measures.
\end{abstract}

%Label to \label{...}
%Refer to \ref{...}
%Refer to equation \eqref{...}

\noindent\textbf{Author's Note (November 2026).}
Following publication of this work in the \emph{International Journal for Uncertainty Quantification} (Vol.~16, No.~3, pp.~27--60, 2026), a gap was identified in Step 2 of the proof of Theorem~\ref{thm:main}.
We have revised the theorem and its proof to establish the stated convergence conclusion under an additional asymptotic entropy-consistency condition.
The corrected argument uses the KL Pythagorean identity and general properties of information projections, without requiring the joint push-forward measure to admit a density with respect to the product of its marginals or imposing the two-sided predictability assumption used in the original proof.
These corrections do not affect the iterative DCI algorithm or the numerical results presented in this work.

%%%%%%%%%%%%
\section{Introduction}
%%%%%%%%%%%%
% {\bf To-Do: Intro should both motivate and describe the stochastic inverse problem (SIP) and the generalized stochastic inverse problem (GSIP) considered in this work. This will lead to a high-level summary of the DCI literature that has developed solution methods for the SIP based on the disintegration theorem. This leads to the contributions of this work, which should then be clearly delineated including the theory of f-divergence optimality of the DCI solution to the SIP and the development and analysis of an iterative DCI approach for solving the GSIP. Numerical results both illustrate the theoretical results while also demonstrating the differences between the GSIP solution and the SIP solution.}

A key objective of modern computational science and engineering is the rigorous quantification of uncertainty in physics-based models, driven by the desire to produce data-informed predictions. 
Central to this effort is the formulation and solution of stochastic inverse problems (SIPs), which seek to constrain input uncertainties using observational data. 
The SIP is broadly defined in the context of this work as constructing a probability measure on an input space ($\Lambda$) of a measurable quantities of interest (QoI) map ($\phi$) that is subject to a push-forward constraint associated with some specified observed probability measure ($P_{\text{obs}}$) on the output space ($\dspace$). 
More formally, the solution is a probability measure $P$ on $\Lambda$ such that $P \circ \phi^{-1} = P_{\text{obs}}$. 
In other words, the solution is a pullback measure of the observed distribution, assigning the correct probability to events in the parameter space based on observed frequencies in the data space.
In that sense, the solution is naturally interpreted as being a data-generating distribution that is consistent with the observed output statistics and the assumed forward model. 

While the standard SIP considers a single (possibly vector-valued) QoI map, many practical applications involve conditioning parameter uncertainties on multiple, distinct observational datasets associated with separate QoI maps.
Such situations arise when different experimental campaigns lead to asynchronous data for which no joint structures are observed or known simultaneously across every QoI component from each QoI map. 
Additionally, in many realistic applications, even if joint structures were known across every QoI component, methods for estimating densities on such subspaces from finite data samples have practical dimension limitations that motivate a reformulation of the SIP on multiple, low-dimensional QoI subspaces.
This leads to the Generalized Stochastic Inverse Problem (GSIP), which requires finding a single probability measure $P$ that simultaneously satisfies $k$ distinct push-forward constraints, $P \circ \phi_i^{-1} = P_{\text{obs},i}$, for $i=1, \dots, k$, where $\phi_i$ denotes the $i$th (possibly vector-valued) QoI map and $P_{\text{obs},i}$ is the associated observed probability measure.
The GSIP presents unique challenges compared to the SIP, particularly because the full joint observed distribution across all output spaces ($\dspace_1 \times \cdots \times \dspace_k$ where $\dspace_i$ is the output space associated with $\phi_i$) is generally not assumed to be available or necessary for solution. 
Subsequently, solving the GSIP requires an approach that resolves these constraints individually and iteratively.

It is worth noting that the SIP and GSIP referred to in this work are distinct from the Bayesian formulation of parameter inference problems found in the uncertainty quantification literature. 
In a typical Bayesian framework, a common assumption is that of an additive noise model on the data that follows a given distribution, usually assumed to be Gaussian, which models the uncertainty associated with measurement errors.
Subsequently, the posterior distribution solving the Bayesian formulation of the problem is defined by a conditional density given by the product of a prior density on parameters and a data-likelihood function. 
The posterior is not a pullback measure. 
Instead, it is interpreted as defining the relative likelihoods that a fixed estimate for the parameters of interest could have produced all of the observed (noisy) data.
For more information on the Bayesian formulation and its solution, we direct the interested reader to~\cite{CKS14, Gelman2013, CDS10, KO2001, BMP+1994, Fitzpatrick1991, Burger_2014, PMS+14, APS+16}.
It is possible to construct hierarchical Bayesian methods that provide an alternative to the regular Bayesian framework to solve a similar problem to the SIP, e.g., see~\cite{Wikle1998}.
Hierarchical Bayesian methods commonly specify prior distributions from a parametric family of distributions.
Hyper-parameters are introduced as random variables into the inference process to describe the uncertainty in the parametric family of distributions used for the prior. 
Implementation of these approaches consequently require an increase in the number of samples that must be computed to obtain accurate inferences because samples must be drawn from a ``hierarchy'' of distributions \cite{pyMC3}.

To solve the SIP, we utilize the Data-Consistent Inversion (DCI) framework, which is a non-parametric measure-theoretic methodology rooted in the disintegration theorem \cite{BET+14, BJW18a, BBE24}. 
We utilize the density-based DCI terminology and notation in this work where one begins by specifying an initial density ($\initdens$) that quantifies uncertainties on the parameter space prior to the observation of QoI data.
A predicted density ($\predictdens$) is constructed as the density associated with the push-forward of the initial probability measure through the QoI map $\phi$. 
Given the observed density ($\obsdens$), the DCI solution, referred to as the updated density ($\updatedens$), is then given by the product of $\initdens(\lambda)$ and the ratio $\obsdens(\phi(\lambda)) / \predictdens(\phi(\lambda))$. 
The density-based DCI methodology has seen substantial theoretical development, establishing existence, uniqueness (up to choose of $\initdens$), stability, and convergence properties, e.g., see~\cite{BJW18a, BWY20, BSW+25}.

Despite the robustness of DCI for solving the standard SIP, solving the GSIP remains complex. 
Motivated alternating projection methods in convex optimization, we propose an iterative DCI approach. 
The foundation of this approach relies on demonstrating the optimality of the DCI solution, which naturally leads to the iterative scheme that is the focus of this work.

The main contributions of this work are summarized as follows:
\begin{enumerate}
    \item \textbf{Theoretical Optimality of DCI for SIP:} We prove that the DCI solution to the standard SIP, as defined by the disintegration theorem, is optimal in the sense that it minimizes the $f$-divergence (forward minimization) over the space of all pullback measures satisfying the push-forward constraint. Furthermore, we demonstrate that for the Kullback-Leibler (KL) divergence, the DCI solution satisfies both forward and backward minimization problems simultaneously under certain practical assumptions.
    \item \textbf{Development and Analysis of Iterative DCI for GSIP:} We introduce a novel iterative DCI procedure for solving the GSIP. This approach sequentially updates the parameter measure by cycling through the $k$ constraints, utilizing the DCI solution at each step as the new initial measure for the next iteration. We prove that this iterative sequence converges in total variation to a solution in the intersection of all constraint sets, assuming finite-KL feasibility and an asymptotic entropy-consistency condition on the sequence of joint push-forward measures.
    \item \textbf{Maximal Entropy Solution and I-Projection:} We show that the converged solution of the iterative DCI scheme is the information projection (I-projection) of the initial distribution onto the GSIP solution set. 
    An immediate consequence is that if the procedure is initialized with a uniform distribution, the solution corresponds to the maximal entropy measure consistent with the $k$ marginal constraints.
    \item \textbf{Numerical Validation:} We provide numerical examples that illustrate the theoretical convergence properties of the iterative DCI algorithm. Critically, we demonstrate the computational utility of the GSIP solution, including in scenarios where traditional joint SIP approaches fail due to numerical challenges in approximating high-dimensional observational densities or using misspecified joint densities that may violate predictability assumptions.
\end{enumerate}

To help situate the main contributions of this work within the broader landscape of inverse problems, especially those based on entropy-regularized methods, we note some conceptual connections to classical methods.
Specifically, the iterative updates resemble adaptive/iterative importance sampling schemes, and the cycling over multiple marginal constraints is motivated by classical iterative proportional fitting. 
We also note that the convergence result identifying the limit as the I-projection also has a clear information-geometric interpretation: among all parameter measures that satisfy the push-forward constraints, the algorithm selects the one closest to the chosen reference measure $P_\text{init}$. 
When $P_\text{init}$ is uniform on the feasible set, this provides a maximum-entropy characterization under the imposed constraints.

The rest of this manuscript is organized as follows.
Section~\ref{sec:DCI_summary} provides background on DCI and summarizes the necessary notation and terminology while also providing citations to relevant literature including to comparisons of DCI and Bayesian methodologies. 
Section~\ref{sec:f-divergence-solns} relates DCI solutions and $f$-divergences including the theoretical proofs of the optimality of DCI solutions.
The GSIP is formally defined in Section~\ref{sec:GSIP}.
The iterative approach for solving the GSIP is developed and analyzed in Section~\ref{sec:iterative_DCI}.
Numerical results follow in Section~\ref{sec:numerics}.
Conclusions and future research directions are discussed in Section~\ref{sec:conclusions}.
Instructions on how to obtain the supplementary material containing the data and scripts used to generate the results in this work are found in Section~\ref{sec:supplementary}. 
Section~\ref{sec:acknowledgments} contains acknowledgment of funding sources for this research.

\section{Data-Consistent Inversion (DCI): Background, Notation, and Terminology}\label{sec:DCI_summary}

% {\bf To-Do: \sout{Define the SIP and add high-level summary of DCI that introduces the initial, predicted, observed, and updated notation/terminology in the context of a pullback for a single (possibly vector-valued) map $\phi$. This is in contrast to the next sections that define the generalized SIP and iterative framework for solving it.} Need to update the bibliography as well.}

The {\em stochastic inverse problem (SIP)} of this work refers, broadly, to the problem of constructing a probability measure on the input space of a map that is subject to a push-forward constraint associated with some specified observed probability measure on the output space. 
More precisely, given separable, complete metric spaces $\pspace\subset\mathbb{R}^p$ and $\dspace\in\mathbb{R}^d$ that are equipped with Borel $\sigma$-algebras $\pborel$ and $\dborel$ (resp.), a measurable map $\phi:\pspace\to\dspace$, and an observed probability measure $\obsmeas$ on $(\dspace,\dborel)$, the SIP is defined as finding a probability measure $P$ on $(\pspace,\pborel)$ subject to the push-forward constraint
\begin{equation}\label{eq:data-consistency}
    P\circ\phi^{-1}(A) = \obsmeas(A), \ \forall \, A\in\dborel.
\end{equation}
In other words, the solution to the SIP is a pullback measure of $\obsmeas$.
The push-forward constraint given by~\eqref{eq:data-consistency} is sometimes referred to as the data-consistency constraint since, for a given $A\in\dborel$, the solution must assign the same probabilities to events $\phi^{-1}(A)\in\pborel$ as observed in the data space. 

To solve the SIP, we utilize the data-consistent inversion (DCI) framework~\cite{BET+14, BJW18a, BBE24}.
In recent years, the density-based approximation of the DCI solution, as derived in \cite{BJW18a} via the Disintegration Theorem \cite{changandpollard1997}, has seen the most development, analysis, and application, e.g., see~\cite{ZM2023, RPK+23, MSB+22, tran2021solving, BGW2020, FBB25, BH20}.
A similar form of the density-based DCI solution was separately derived in \cite{PR2000} through heuristic arguments based on logarithmic pooling and referred to as ``Bayesian melding.'' 
To help distinguish this methodology from classical Bayesian inference, the DCI terminology, which is utilized in this work, was introduced in \cite{BJW18b}.
In that work, an initial and predicted density (denoted by $\initdens$ and $\predictdens$, resp.) are used to describe the initial quantification of uncertainties on $(\pspace,\pborel)$ and $(\dspace,\dborel)$ (resp.) independent of any observed data.
An observed density, denoted by $\obsdens$, then describes the quantification of uncertainty for the observed output data.
The DCI solution is then obtained via the product of $\initdens$ with the ratio of $\obsdens$ to $\predictdens$ evaluated at $\phi(\lambda)$ for each $\lambda\in\pspace$. 
We refer to this solution as the updated density and denote it by $\updatedens$.
We therefore focus the rest of this section on summarizing, briefly, the precise mathematical content of \cite{BJW18a} with the terminology and notation described above to set the stage for the contributions of this current work.
For a thorough comparison of DCI and Bayesian frameworks and methods, we direct the interested reader to Sections 2 and 4 of \cite{PCT+23}, Section 7 in \cite{BJW18a}, Sections 1 and 2 of \cite{BWY20}, and to a recent review paper \cite{BBE24}.

As previously mentioned, the DCI solution to the SIP relies upon the disintegration theorem \cite{changandpollard1997}. 
This is a critical theoretical result that provides a rigorous framework to decompose a probability measure through a measurable function, and helps construct solutions at each step in the iterative algorithm that we develop in this work. 
Given a probability $P^*$ on $(\Lambda,\B_{\Lambda})$ and a measurable mapping $\phi: \Lambda \rightarrow \mathcal{D}$, the disintegration theorem allows us to disintegrate $P^*$ as
\begin{equation} \label{disintegration}
    P^*(A) = \int_{\mathcal{D}}P^*_{\omega}(A) dQ^{*}(\omega), \ \forall \, A\in B_\Lambda,
\end{equation}
where $Q^{*}:= P^{*}\circ\phi^{-1}$ is the push-forward measure of $P^*$.
Here, the $\{P^*_\omega\}_{\omega \in \mathcal{D}}$ represent the {\em disintegration} of $P^*$ into a $Q^*$-a.e.~uniquely defined family of probability measures representing conditional probability measures defined on $\phi^{-1}(\omega)$ for a.e.~$\omega\in\mathcal{D}$, i.e., $P^*_\omega(A) = P^*_\omega(A\cap \phi^{-1}(\omega))$.

Let $\initmeas$ denote a probability measure on $(\Lambda,\B_\Lambda)$ and $\predictmeas:=\initmeas\circ \phi^{-1}$ denote its push-forward measure on $(\mathcal{D}, \B_\dspace)$, then the disintegration of $\initmeas$ yields
\begin{equation}\label{eq:disintegrate_initial}
    \initmeas(A) = \int_\dspace \initmeasomega(A)\, d\predictmeas(\omega), \ \forall \, A\in B_\pspace.
\end{equation}
It is shown in \cite{BET+14} that for a given $P_\text{obs}$ on $(\mathcal{D},\B_\mathcal{D})$ and $P_\text{init}$ on $(\Lambda,\B_\Lambda)$ (referred to as an ansatz in that work), the DCI solution is given by 
\begin{equation} \label{integration}
    P_\text{up}(\cdot) := \int_{\mathcal{D}} P_{\text{init}, \omega}(\cdot) dP_\text{obs}(\omega).   
\end{equation}

We further assume that $\initmeas, \predictmeas,$ and $\obsmeas$ are all absolutely continuous with respect to their associated dominating measures (denoted by $\mu_\Lambda$ on $(\Lambda,\B_\Lambda)$ and $\mu_\mathcal{D}$ on $(\mathcal{D}, \B_\mathcal{D})$) so that they admit Radon-Nikodym derivatives denoted by $\initdens$, $\predictdens$, and $\obsdens$, respectively. 
Following \cite{BJW18a}, if there exists a constant $C>0$ such that $\obsdens(\omega)\leq C\predictdens(\omega)$ for a.e.~$\omega$ (referred to as the predictability assumption), then the Radon-Nikodym derivative of the updated measure is shown to exist in the form given by
\begin{equation}\label{eq:updated_density}
    \updatedens(\lambda) = \initdens(\lambda)\frac{\obsdens(\phi(\lambda))}{\predictdens(\phi(\lambda))}.
\end{equation}
Subsequently, the evaluation of $\updatemeas(A)$ for $A\in\B_\Lambda$ can be rewritten either as
\begin{equation}\label{eq:density-based-DCI}
    \updatemeas(A) = \int_\mathcal{D} \int_{A\cap \phi^{-1}(\omega)} \initdens(\lambda)\frac{\obsdens(\phi(\lambda))}{\predictdens(\phi(\lambda))}\, d\mu_{\Lambda,\omega}(\lambda)\,  d\mu_\mathcal{D}(\omega),
\end{equation}
or as
\begin{equation}\label{eq:density-based-DCI-factored}
    \updatemeas(A) = \int_\mathcal{D} \left(\int_{A\cap \phi^{-1}(\omega)} \initdens(\lambda)\, d\mu_{\Lambda,\omega}(\lambda)\right) \frac{\obsdens(\omega)}{\predictdens(\omega)}\,  d\mu_\mathcal{D}(\omega).
\end{equation}
This later form emphasizes that the DCI solution is conceptually defined as a re-weighting of the initial measure, but only in the directions of $\Lambda$ for which $\phi$ varies.
While~\eqref{eq:updated_density}-\eqref{eq:density-based-DCI-factored} are all presented for Radon-Nikodym derivatives, we generally refer to these as representing the ``density-based'' DCI solution.

We conclude this section with a few remarks on the practical implementation of the density-based DCI solution and verification 
% \tmw{Troy - is this verification or validation?  We changed it at some point.} @Tim: I believe verification is the correct term here because we are verifying the conditions hold to apply the framework. I checked against the resposne an AI gave me to this prompt ``do you verify or validate an assumption holds to apply a theorem'' just to make sure I am not confusing the terms. 
of the predictability assumption.
First, we rewrite~\eqref{eq:updated_density} as 
\begin{equation}\label{eq:updated_density_r}
    \updatedens(\lambda) = \initdens(\lambda) r(\lambda), \ \text{ where } \ r(\lambda) := \frac{\obsdens(\phi(\lambda))}{\predictdens(\phi(\lambda))}.
\end{equation}
For a given set of $N$ independent identically distributed (iid) samples $\{\lambda^{(i)}\}\sim\initdens$, we refer to the corresponding $\{r(\lambda^{(i)})\}$ values as simply the ``$r$-values'' that re-weight this iid set of samples so that the weighted sample set can be interpreted as being drawn from $\updatedens$.
We can, for instance, perform rejection sampling (as in \cite{BJW18a}) on such a sample set to generate an iid set of samples from $\updatedens$.
Alternatively, we can compute a weighted kernel density estimate (KDE) to directly approximate $\updatedens$ (e.g., as shown in~\cite{RHB25}).
Similarly, applying a different QoI map $\tilde{\phi}$ to generate $\{\tilde{\phi}(\lambda^{(i)})\}$ and then computing the weighted KDE with the prior $r$-values will generate an approximation of the push-forward density defined by $\updatedens$ and $\tilde{\phi}$. 
This approach is utilized in Algorithm~\ref{alg:main} (discussed in Section~\ref{sec:iterative_DCI}) to generate the numerical results of Section~\ref{sec:numerics}.
Finally, we note that if the predictability assumption holds, then $\updatedens$ is indeed a density and it follows that
\begin{equation}
     \mathbb{E}_\text{init}(r(\lambda)) = \int_\pspace \updatedens(\lambda) = 1 \ \Rightarrow \ \frac{1}{N}\sum_i r(\lambda^{(i)}) \approx 1.
\end{equation}
This naturally leads to a computational diagnostic where, for a given set of parameter samples and associated $r$-values, we check the validity of the predictability assumption by verifying that
\begin{equation}\label{eq:diagnostic}
    \left\vert \frac{1}{N}\sum_i r(\lambda^{(i)}) - 1 \right\vert < tol_r
\end{equation}
for a specific tolerance $tol_r$ that is often chosen to be $0.1$ to account for finite sampling error in the approximation of the expected $r$-values.
This diagnostic is utilized in Algorithm~\ref{alg:main} and is discussed in more detail in the numerical results of Section~\ref{sec:numerics}.

\section{DCI Solutions Minimize f-Divergences}\label{sec:f-divergence-solns}

It is worth noting that by varying $\initmeas$, we can generate a range of possible solutions, $\updatemeas$, implying that any pullback measure can be recovered by a well chosen initial distribution. 
It may also be necessary to vary $\initmeas$ to ensure that $\initmeas \ll \updatemeas$, which ensures that the space of potential solutions is nonempty. 
This flexibility in the choice of an initial distribution ultimately allows for a more comprehensive exploration of the solution space.
Although varying $\initmeas$ is not the focus of this work, it will be the topic of a future work as mentioned in Section~\ref{sec:conclusions}.

Once $\initmeas$ is chosen, we prove below that the DCI solution is optimal in the sense of minimizing the f-divergences over the space of all probability measures that are absolutely continuous with respect to $\initmeas$ and whose push-forward measure is equal to $\obsmeas$. 
The f-divergences \cite{polyanskiy2022information} are commonly employed in probability and information theory to quantify the difference between probability measures often in the context of determining optimal parameters or hyper-parameters of density models, e.g., see~\cite{Guntuboyina2011, BKM17}. 
A key feature of f-divergences are their flexibility, as they generalize and connect many well-known divergence measures, such as the Kullback-Leibler (KL) divergence and total variation distance \cite{KL1951, RenyiKL2014, BKM17, AH21}.
Utilizing the notation of this work, the f-divergence is defined as:
\[
    D_{f}(P || P_\text{init}) = \int_\pspace f\left(\frac{dP}{dP_\text{init}}(\lambda)\right) dP_\text{init}(\lambda)
\]
where $P \ll P_\text{init}$ and \( f: [0, \infty) \to \mathbb{R} \) is a convex function satisfying \( f(1) = 0 \) and $f(0) = \lim_{t \rightarrow 0^{+}}f(t)$.
When $f(t)=t\log t$, we obtain the KL divergence.

\begin{lemma}\label{lemma:for_bac} 
Let $\obsmeas$ denote a probability measure on $(\dspace, \dborel)$ and $\initmeas$ denote any measure on $(\pspace,\pborel)$ such that $\obsmeas$ is absolutely continuous with respect to $\predictmeas=\initmeas\circ \phi^{-1}$. 
Let $\p = \{P\ll \initmeas\, : \, P\circ \phi^{-1} = \obsmeas\}$, then $\updatemeas\in \p$ as constructed in~\eqref{integration} solves both the forward minimization and backward minimization problems, i.e., 
\[
    \updatemeas = \argmin_{P\in \p} D_{f}(P || \initmeas) \quad \text{Forward minimization},
\]
and, if $\initmeas \ll \updatemeas$, then
\[
    \updatemeas = \argmin_{P\in \tilde{\p}} KL(\initmeas || P) \quad \text{Backward minimization}
\]
where $\tilde{\p}$ is the subspace of $\p$ such that $\initmeas \ll P$ for all $P\in\tilde{\p}$.
\end{lemma}

\begin{proof}
Note that $\updatemeas\in\p$ is immediate from its construction. 
We therefore first prove that $\updatemeas$ is the solution to the forward minimization problem.
First, we establish that $\updatemeas$ is absolutely continuous with respect to $\initmeas$ so that the standard f-divergence can be computed. 
By the construction of $\updatemeas$ via the disintegration theorem and the Radon-Nikodym formula for applying change-of-measure, we have for each $A\in \B_\Lambda$ the following, 
\begin{eqnarray*}
    \updatemeas(A) =\int_A \, d\updatemeas(\lambda) &=& \int_\dspace \initmeasomega(A) \, d\obsmeas(\omega)\\
                &=&\int_{\dspace} \initmeasomega(A) \cfrac{d\obsmeas}{d\predictmeas}(\omega)\, d\predictmeas(\omega) \\
                &=&\int_A  \cfrac{d\obsmeas}{d\predictmeas}(\phi(\lambda)) dP_\text{init}.
\end{eqnarray*}
It follows that for a.e.~$\lambda$, we have
\[
    \cfrac{d\updatemeas}{d\initmeas}(\lambda) = \cfrac{d\obsmeas}{d\predictmeas}(\phi(\lambda)).
\]
In other words, the Radon-Nikodym derivative of $\updatemeas$ with respect to $P_\text{init}$ can be computed in terms of the Radon-Nikodym derivative of $\obsmeas$ with respect to $\predictmeas$, which exists by the assumption of absolute continuity of $\obsmeas$ with respect to $\predictmeas$. 
Then, by the definition of f-divergence for the non-singular case and the disintegration theorem, we have
\begin{align}\label{eq:Aint}
    D_{f}(\updatemeas || \initmeas) &=  \int_{\Lambda} f\left(\cfrac{d\updatemeas}{d\initmeas}(\lambda)\right) d\initmeas(\lambda) \nonumber \\
    &= \int_{\mathcal{D}} \underbrace{\int_{\phi^{-1}(\omega)} f\left(\cfrac{d\updatemeas}{d\initmeas}(\lambda)\right) d\initmeasomega(\lambda)}_{=: a(\omega)} d\predictmeas(\omega) \\ 
    &= \int_{\mathcal{D}} a(\omega) d\predictmeas(\omega) \nonumber.
\end{align}

Of particular note is that since $\updatemeas$ and $P_\text{init}$ admit the same family of disintegrated conditional densities, then for each $\omega$, the Radon-Nikodym derivative $\frac{d\updatemeas}{d\initmeas}$ is a constant for all $\lambda\in\phi^{-1}(\omega)$.
Thus,
\begin{eqnarray*}
    a(\omega) &=& \int_{\phi^{-1}(\omega)} f\left(\cfrac{d\updatemeas}{d\initmeas}(\lambda)\right) d\initmeasomega(\lambda) \\
            &=& \int_{\phi^{-1}(\omega)} f\left(\cfrac{d\obsmeas}{d\predictmeas}(\phi(\lambda))\right) d\initmeasomega(\lambda)  \\
            &=& f\left(\cfrac{d\obsmeas}{d\predictmeas}(\omega)\right) \underbrace{\int_{\phi^{-1}(\omega)} d\initmeasomega(\lambda)}_{=1}\\
            &=& f\left(\cfrac{d\obsmeas}{d\predictmeas}(\omega)\right).
\end{eqnarray*}

Consider any other $P\in\p$ where it similarly follows that
\begin{align}\label{eq:Bint}
  D_{f}(P || \initmeas) =  \int_{\Lambda} f\left(\cfrac{dP}{d\initmeas}\right) d\initmeas(\lambda) &= \int_{\mathcal{D}} \underbrace{\int_{\phi^{-1}(\omega)} f\left(\cfrac{dP}{d\initmeas}\right) d\initmeasomega(\lambda)}_{=: b(\omega)} d\predictmeas(\omega) \\
  &= \int_{\mathcal{D}} b(\omega) d\predictmeas(\omega). \nonumber 
\end{align}
Due to the convexity of $f$,
\begin{equation*}
    b(\omega) = \int_{\phi^{-1}(\omega)} f\left(\cfrac{dP}{d\initmeas}(\lambda)\right) d\initmeasomega(\lambda) \ge f\left(\int_{\phi^{-1}(\omega)} \cfrac{dP}{d\initmeas}(\lambda) d\initmeasomega(\lambda)\right).
\end{equation*}
Note that for any $B \in \mathcal{B}_{\mathcal{D}}$, applying the disintegration theorem and using the Radon-Nikodym derivative of $P$ with respect to $P_\text{init}$ to perform a change of measure on $\phi^{-1}(B)\in\B_\Lambda$ yields
\begin{equation*}
    P(\phi^{-1}(B)) = \int_{\phi^{-1}(B)} \cfrac{dP}{d\initmeas}(\lambda) d\initmeas(\lambda) = \int_{B} \left( \int_{\phi^{-1}(\omega)} \cfrac{dP}{d\initmeas}(\lambda) \, d\initmeasomega(\lambda) \right) \, d\predictmeas(\omega).
\end{equation*}
At the same time, since $P\in\p$, we have
\begin{equation*}
    P(\phi^{-1}(B)) = \obsmeas(B) = \int_B \, d\obsmeas(\omega) = \int_{B} \cfrac{d\obsmeas}{d\predictmeas}(\omega) d\predictmeas(\omega).
\end{equation*}
Combining these results, it follows that for a.e.~$\omega$, 
\begin{equation*}
    \int_{\phi^{-1}(\omega)} \cfrac{dP}{d\initmeas}(\lambda) d\initmeasomega(\lambda) = \cfrac{d\obsmeas}{d\predictmeas}(\omega).
\end{equation*}
In other words, the integration of $\cfrac{dP}{dP_\text{init}}(\lambda)$ with respect to the conditional measure $\initmeasomega$ over $\phi^{-1}(\omega)$ is equal to $\cfrac{d\obsmeas}{d\predictmeas}(\omega)$.
This is different than the relationship discussed above between $\cfrac{d\updatemeas}{dP_\text{init}}(\lambda)$ and $\cfrac{d\obsmeas}{d\predictmeas}(\omega)$ since $\cfrac{dP}{dP_\text{init}}(\lambda)$ can vary across the set $\lambda\in\phi^{-1}(\omega)$.  
More importantly, this establishes the desired inequality of $b(\omega) \geq f\left(\cfrac{d\obsmeas}{d\predictmeas}(\omega)\right)$.
We have thus demonstrated $\omega$-a.e. that $a(\omega) \le b(\omega)$ so  that
\begin{equation*}
 D_{f}(\updatemeas || \initmeas) = \int_{\mathcal{D}} a(\omega) d\predictmeas(\omega)  \le \int_{\mathcal{D}} b(\omega) d\predictmeas(\omega) = D_{f}(P || \initmeas) ,
\end{equation*}
proving the assertion of forward minimization. 

We next show that $\updatemeas$ is the solution to the backward minimization problem where we restrict $P\in\tilde{\p}$.
We begin by manipulating Radon-Nikodym derivatives through the chain rule to obtain
\begin{align*}
KL( \initmeas || P) - KL( \initmeas || \updatemeas) &= \int_{\Lambda} \log\left(\cfrac{d\initmeas}{dP}(\lambda)\right) d\initmeas(\lambda) - \int_{\Lambda} \log\left(\cfrac{d\initmeas}{d\updatemeas}\right)(\lambda) d\initmeas(\lambda)\\
&= \int_{\Lambda} \log\left(\cfrac{d\updatemeas}{dP}(\lambda)\right) \, d\initmeas(\lambda).
\end{align*}
For each $A\in\B_\Lambda$, define the measure $P^{\prime}$ as follows %{\bf to-do: be more precise about the use of RN derivatives below by referring to the assumptions of absolute continuity that need to be made explicit in this lemma that allow for $dP/d\updatemeas$ to have meaning}
\begin{equation}\label{eq:change_of_conditionals}
    P^{\prime}(A) := \int_A \cfrac{dP}{d\updatemeas}(\lambda) \, d\initmeas(\lambda).
\end{equation}
By direct substitution, 
\begin{eqnarray*}
    KL( \initmeas || P) - KL( \initmeas || \updatemeas) &=& \int_{\Lambda} \log\left(\cfrac{d\updatemeas}{dP}(\lambda)\right) d\initmeas(\lambda) \\
    &=& \int_{\Lambda} \log\left(\cfrac{d\initmeas}{dP^{\prime}}(\lambda)\right) d\initmeas(\lambda) = KL(P_\text{init} || P^{\prime}).
\end{eqnarray*}    
If $P^{\prime}$ is a probability measure, then $KL(P_\text{init} || P^{\prime})\geq 0$ by properties of the KL divergence, which will complete the proof. 
Clearly, $\int \cfrac{dP}{d\updatemeas} d\initmeas$ defines a measure, so we only show that $P^{\prime}(\Lambda) = 1$, which follows from 
\begin{align*}
    P^{\prime}(\Lambda) = \int_{\Lambda} \cfrac{dP}{d\updatemeas}(\lambda) d\initmeas(\lambda) &= \int_{\mathcal{D}} \left(\int_{\phi^{-1}(\omega)} \cfrac{dP}{d\updatemeas}(\lambda) \, d\initmeasomega(\lambda)\right) \, d\predictmeas(\omega) \\
    &= \int_{\mathcal{D}} \left(\int_{\phi^{-1}(\omega)} \cfrac{dP_\omega}{d\initmeasomega}(\lambda) \, d\initmeasomega(\lambda) \right)\, d\predictmeas(\omega) \\
    &= \int_{\mathcal{D}} \int_{\phi^{-1}(\omega)} \, dP_\omega(\lambda) \, d\predictmeas(\omega) \\
    &= P(\pspace) = 1.
\end{align*}
Above, the conditionals defining $dP_\omega$ are uniquely determined by the disintegration of $P$ and we utilized the fact that, by construction, $\updatemeas$ has the same conditionals as $\initmeas$. 
\end{proof}

An interesting consequence from the above proof is that the last step implies
\begin{equation*}
P^{\prime}(A) = \int_{\mathcal{D}}P_{\omega}(A) d\predictmeas(\omega), \ \forall \, A\in\B_\Lambda.
\end{equation*}
In other words, the probability measure $P^{\prime}$ has the same push-forward as $P_\text{init}$ but disintegrates into the conditional distributions defined by $P$.
Subsequently,~\eqref{eq:change_of_conditionals} represents a change of conditional measure formula.

The following corollary is an immediate consequence of~\eqref{eq:density-based-DCI}.

\begin{corollary}
If the Radon-Nikodym derivatives $\initdens$, $\predictdens$, and $\updatedens$ all exist and there exists $C> 0$ such that $\predictdens(\omega)\leq C\obsdens(\omega)$, then the conclusions of Lemma~\ref{lemma:for_bac} hold for $\updatemeas$ as constructed in~\eqref{eq:density-based-DCI}.
\end{corollary}

%%%%%%%%%%%%
\section{The Generalized Stochastic Inverse Problem}\label{sec:GSIP}
%%%%%%%%%%%%
We now describe the \textit{generalized stochastic inverse problem (GSIP)} that is the particular focus of this work, which also serves to motivate the iterative DCI approach developed and analyzed in the following section. 

Consider separable, complete metric spaces $\Lambda\subset\mathbb{R}^p$ and $\{\mathcal{D}_{i}\}_{i \le k}$ for a positive integer $k$, where, for each $i\le k$, $\mathcal{D}_i\subset\mathbb{R}^{d_i}$ for some positive integer $d_i$. 
Let $(\Lambda, \B_{\Lambda})$ and $\{(\mathcal{D}_{i}, \B_{\mathcal{D}_{i}})\}_{i \le k}$ be the corresponding measurable spaces with Borel $\sigma$-algebras $\B_{\Lambda}$ and $\B_{\mathcal{D}_{i}}$ for $i\le k$, respectively. 
For each $i$, let $\phi_{i}: \Lambda \rightarrow \mathcal{D}_{i}$ be a measurable mapping. 
Any probability measure $P$ on $(\Lambda, \B_\Lambda)$ induces push-forward probability measures denoted by $P\circ \phi_{i}^{-1}$ on $(\mathcal{D}_i, \B_{\mathcal{D}_i})$ for each $i$. 

The GSIP is then defined as follows: For a given collection of observed probability measures, denoted by $P_{\text{obs},i}$ for each $i\le k$, defined on each $(\mathcal{D}_i, \B_{\mathcal{D}_i})$, we aim to find a \emph{single} probability measure $P$ satisfying the following $k$ push-forward constraints
\begin{equation}\label{eq:k-data-consistent}
    (P\circ \phi_{i}^{-1})(A) = P_{\text{obs},i}(A), \quad \forall A\in\B_{\mathcal{D}_i}, \quad i =1, \ldots, k. 
\end{equation}
It is self-evident that if each $P_{\text{obs},i}$ is the induced push-forward measure of some probability measure defined on $(\Lambda, \B_\Lambda)$, then a solution is guaranteed to exist.
We make such an assumption in this work and refer to this measure as the {\em data-generating distribution} (denoted as $P_\text{DG}$), given its role in inducing the observed probability measures. 
Such an assumption is reasonable in practice when observed measures are associated with observed data, but such an assumption may be violated if observed measures are specified as target measures a priori of any data collection.

We briefly emphasize several important differences between the GSIP and the SIP.
First, for the SIP solved by DCI, only a single (possibly vector-valued) map $\phi$ is considered. 
By contrast, there are multiple (possibly vector-valued) maps considered in the GSIP. 
Moreover, for the GSIP, it is possible that some components of a map $\phi_i$ may also appear in other maps $\phi_j$.
This is related to a second key difference in terms of how uncertainties in output data are assumed to be quantified in the SIP and GSIP.
In the SIP, a single {\em joint} observed measure is assumed to be defined on the full data space defined by all the outputs of the map $\phi$. 
However, the GSIP does not assume that the full joint observed distribution is available on $\mathcal{D}_1\times\cdots\times\mathcal{D}_k$. 
A final important note is that a predictability assumption is sufficient to guarantee the DCI solution exists to the SIP (e.g., that $\obsmeas$ is absolutely continuous with respect to $\predictmeas=\initmeas\circ\phi^{-1}$). 
While an analogous predictability assumption is also required for the solution of the GSIP, a further assumption is needed to guarantee that each of the separate $k$ push-forward constraints of~\eqref{eq:k-data-consistent} are simultaneously satisfied, which is handled in this work by the additional assumption of a data-generating distribution. 

There are some notable similarities between the GSIP and SIP.
First of all, both involve similar forms of push-forward constraints, which at least hints at the possibility that the DCI framework could be used in some manner to solve the GSIP. 
We also note that for either problem, a convex combination of solutions is also a solution.
Thus, these problems either have a unique solution or infinite solutions. 
However, it is worth emphasizing that for either problem the solution may not be unique, which is immediately obvious once one considers the potentially ``ideal'' situation for the SIP where $\phi$ is invertible since this alone is insufficient to guarantee the conditions are satisfied for the well-known change-of-variables formula, which served as the motivation for the DCI framework \cite{BET+14, BJW18a}. 
Ultimately, for both the SIP and GSIP, solutions obtained via DCI-inspired techniques are unique only up to the choice of an initial distribution. 
%\tmw{When would we expect them to give the same solution?  When the push-forward of the data-generating distribution is actually the product of the marginals, i.e., if the QoI are actually independent?  This might be worth noting.} @ Tim, see the next sentence. I don't want to elaborate further how this actually gets into geometric decompositions of the QoI maps. I think this is something worth diving into in more detail in a follow-up paper on OED. 
When the initial distribution is identical for both the SIP and GSIP, we would only expect the solutions to the SIP and GSIP to agree if (1) the predicted distribution for the SIP is defined by the product of the predicted distributions for the GSIP, and (2) the observed distribution for the SIP is similarly defined by the product of the observed distributions for the GSIP. 
%In the case of nonuniqueness, all solutions are equivalent in the sense that they all induce the same collection of push-forward measures denoted by the $P_{\text{obs},i}$. 

%%%%%%%%%%%%
\section{An Iterative DCI Approach}\label{sec:iterative_DCI}
%%%%%%%%%%%%
Motivated by Lemma~\ref{lemma:for_bac}, we propose to find a solution to the GSIP based on an iterative DCI approach to sequentially update an initial distribution that results in a unique solution, up to the choice of $\initmeas$, that simultaneously minimizes f-divergences associated with each constraint.
This is in contrast to the DCI approach for solving the SIP that minimizes a single f-divergence associated with a single push-forward constraint.

For each $1\leq i\leq k$, let $\p_{i}$ denote the set of probability measures that satisfy the $i$th push-forward constraint, i.e., $\p_{i} := \{P\ll \initmeas\, :\,  P\circ\phi_{i}^{-1} = P_{\text{obs},i}\}$. 
For mathematical convenience in the iterative scheme utilized in this work, denote $\p_{mk+1}=\p_1, \p_{mk+2}=\p_2, \ldots$, where $m=0, 1, \ldots$. 
It follows that $\p_n=\p_i$ if $(n-i) \mathrm{~mod~} k =0$ (i.e., if there exists nonnegative integer $m$ such that $n=mk+i$ for $1\leq i\leq k$).
The solution set of the GSIP is then expressed as 
\begin{equation*}
    \p = \bigcap_{i = 1}^k \p_{i} = \bigcap_{n = 1}^\infty \p_{n}.
\end{equation*} 
Since each set $\p_i$ is convex, so is the intersection defining $\p$. 
Moreover, under the assumption that the $P_{\text{obs},i}$ are themselves induced by some data-generating distribution, the above intersection is non-empty as it must contain at least $P_\text{DG}$. 

Finding a solution in $\p$ can be challenging. 
We therefore begin by formulating an optimization problem over a convex set. 
In particular, we consider
\begin{align} \label{overall_min}
    \min_{P \in \p}D_f(P|| P_\text{init}).    
\end{align}
%where $D$ denotes a discrepancy such as an $f$-divergence between probability measures. 
Clearly, all elements in $\p$ can be viewed as a minimizer for some choice of $P_\text{init}$. 
However, we emphasize that $P_\text{init}$ need not belong to $\p$. 
This particular probability measure simply represents the initial probability measure previously mentioned and is a place where expert opinion can weigh in on the final structure of the solution.
The only requirement on $\initmeas$ is that $\obsmeasi \ll \initmeas\circ \phi_i^{-1}$ for each $i$. 

A commonly used approach to tackle an optimization problem with many constraints is to consider one constraint at a time, which leads to a sequence of \textit{subproblems} that we iterate through to compute a solution. 
For the $n$th iteration, we consider $\mathtt{Subproblem}_n$, defined as the minimization problem
\begin{align} \label{sub_min}
    \min_{P \in \p_{n}} D_f(P ||P^{n-1}),
\end{align}
where $P^{n-1}$ is the minimizer of $\mathtt{Subproblem}_{n-1}$.
By Lemma~\ref{lemma:for_bac}, the minimum of~\eqref{sub_min} is given by the DCI updated measure  defined by $P^{n-1}$ using map $\phi_i$ (where $i=n \mathrm{~mod~} k$) and observed measure $\obsmeasi$.
At this point, it is worth noting the use of subscripts and superscripts in this work and the role of DCI in constructing the iterative solutions based on the associated subproblems.
Subscripts are generally used to denote a particular map, space, subproblem, or the role of a particular measure whereas superscripts are used to denote measures that change at each iteration. 
% Thus, in the context of DCI, we start with $P^0=P_\text{init}$, determine the first prediction, denoted $P^1_\text{pred}$, and use density representations of these measures along with that of $P_{\text{obs},1}$ to obtain a solution to $\mathtt{Subproblem}_1$ denoted as $P^1_\text{update}$. 
% In $\mathtt{Subproblem}_2$, $P^1_\text{update}$ plays the role of $P^1$, so in the context of DCI, the first updated measure serves as the initial probability measure for the second iteration, i.e., $P^2\text{init} = P^1_\text{update}$. 
% Once we epoch through the first $k$ subproblems, we arrive at $\matt{Subproblem}_{k+1}$, which again involves $P_{\text{obs},1}$ in the construction of the solution, but now the initial measure $P^{k+1}_\text{init}$ is given by $P^k_\text{update}$. 
In constructing solutions, we utilize the DCI framework at each iteration where the updated measure that solves $\mathtt{Subproblem}_{n-1}$ subsequently serves as the initial measure used to construct a new updated measure that solves $\mathtt{Subproblem}_n$. 
Given the separate roles that a particular initial or updated measure may play, depending on the iteration, it is notationally convenient to drop subscripts from these measures and simply use $P^n$ where its role as either the initial or updated measure is understood based on the relationship of this superscript to the iteration index.
Thus, for a given initial probability measure $P_\text{init}$, we set $P^0=P_\text{init}$ and iteratively solve the subproblems within the DCI framework to yield a sequence of local solutions, $P^1, P^2, \ldots$. 
This is illustrated, abstractly, in Figure~\ref{fig:iterative} for the first $5$ iterations of the proposed iterative minimization process when $k = 3$. 
The illustration suggests that as the iterations progress, each successive element in the sequence is expected to move closer to the intersection $\p = \p_{1} \cap \p_{2} \cap \p_{3}$, and eventually converge to an element in $\p$. 
Such convergence depends on the choice of the discrepancy $D$.
In fact, under mild conditions, we show that such a sequence obtained via the DCI based construction converges to a solution of the optimization problem described in \eqref{overall_min}, namely, a global solution.

\begin{figure}[h!]
    \centering
    \includegraphics[width=0.4\textwidth]{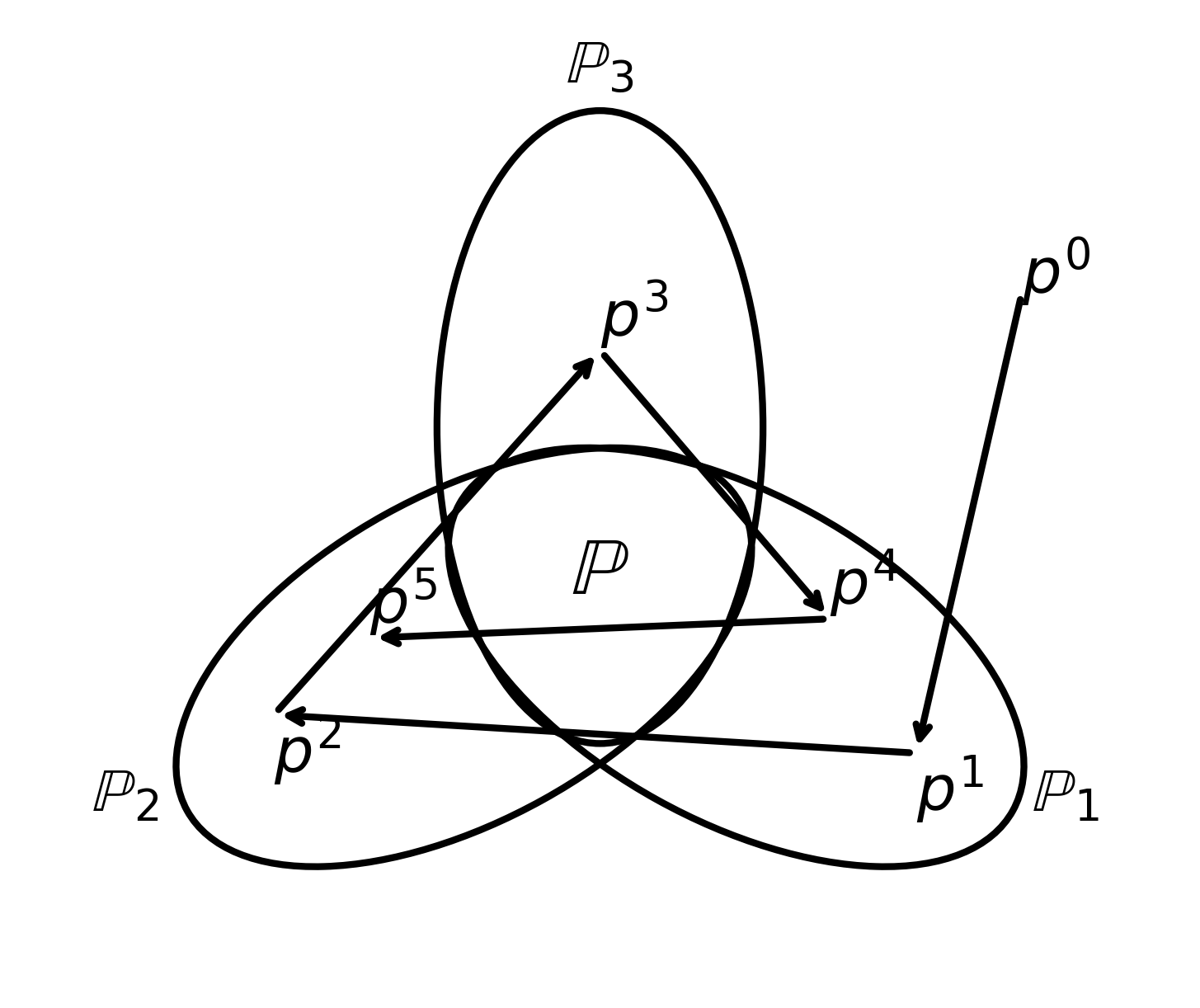}
    \caption{Schematic that conceptualizes the iterative DCI process when $k = 3$ and $\mathbb{P}=\mathbb{P}_1\cap\mathbb{P}_2\cap\mathbb{P}_3$.}
    \label{fig:iterative}
\end{figure}

Note that, due to the non-symmetry of most divergences, the solutions that minimize $D_f(\cdot || P_\text{init})$ and $D_f(P_\text{init}|| \cdot)$ are likely different. 
In practice, we may consider a problem similar to \eqref{overall_min}, defined as
\begin{align} \label{backward}
\min_{P \in \p}D_f(P_\text{init}||P).
\end{align}
Similar to Lemma~\ref{lemma:for_bac}, the problem in \eqref{overall_min} is commonly referred to as the \textit{forward minimization}, whereas \eqref{backward} is known as the \textit{backward minimization}. 
In general, the forward and backward minimization problems require different treatments. 
In this paper, we focus the on forward minimization only.
However, Lemma~\ref{lemma:for_bac} implies that if the KL divergence is utilized, then each updated measure is a solution to both the forward and backward minimization problems for the associated subproblems. 
It then follows that for the special case where $k = 2$ and the KL divergence is utilized, the iterative procedure can be viewed as alternating forward and backward minimization, which assures desirable properties of its accumulation points \cite{csiszar1984, alter_min}.

Before we state the main theoretical results for the iterative approach, we first take advantage of the structure of the DCI solutions to write out explicit forms of the probability measures obtained at each iteration, which sets the stage for both Algorithm~\ref{alg:main} and the proof of Theorem~\ref{thm:main}. 
Recall that $P^0$ not only denotes $\initmeas$, but it also serves as the initial measure in $\mathtt{Subproblem}_1$.
Following the notation of~\eqref{disintegration}, we represent its disintegration for $\mathtt{Subproblem}_1$ as follows,  
\begin{equation*}
    P^{0}(\cdot) = \int_{\mathcal{D}_{1}} P^{0}_{\omega,{1}}(\cdot) dQ^{0}_{1},
\end{equation*}
where $Q^{0}_{1}$ denotes the push-forward of $P^0$ with respect to $\phi_1$ (i.e., $Q^0_1:= P^{0}\circ\phi^{-1}_{1}$) and $\{P^0_{\omega,1}\}_{\omega \in\mathcal{D}_1}$ are the associated conditionals obtained via the disintegration.
The solution to $\mathtt{Subproblem}_1$ is then given by
\begin{equation*}
    P^{1}(\cdot) := \argmin_{P \in \p_{1}}D_f(P||P^{0}) = \int_{\mathcal{D}_{1}} P^{0}_{\omega,{1}}(\cdot) \, dP_{\text{obs}, 1}.
\end{equation*}
Consequently, at step $n$ (with $i=n\mathrm{~mod~}k$), we first disintegrate the solution to $\mathtt{Subproblem}_{n-1}$ with respect to $\phi_i$ to obtain
\begin{equation*}
    P^{n-1}(\cdot) = \int_{\mathcal{D}_{i}} P^{n-1}_{\omega,i}(\cdot) dQ^{n-1}_{i},
\end{equation*}
where $Q^{n-1}_{i} := P^{n-1}\circ\phi^{-1}_{i}$ and $\{P^{n-1}_{\omega,i}\}_{\omega\in\mathcal{D}_i}$ are the conditionals, and we construct the solution to $\mathtt{Subproblem}_n$ as
\begin{equation*}
    P^{n}(\cdot) := \argmin_{P \in \p_{i}}D_f(P || P^{n-1}) = \int_{\mathcal{D}_{i}} P^{n-1}_{\omega,i}(\cdot) dP_{\text{obs},i}.
\end{equation*}

The next result is an immediate consequence of the solution forms to the subproblems.

\begin{lemma} \label{lemma:seq_con}
If $\initmeas \circ \phi_i^{-1} \gg P_{\text{obs},i}$ for all $i\leq k$ and $\p\neq\emptyset$, then $P^{0}=\initmeas \gg P^{1} \gg \dots$. 
Furthermore, after the first $k$-epoch of iterations defined by iterating through all $k$ subproblems one time, all the $P^{n}$ measures are mutually absolutely continuous, i.e., $P^{n}\gg P^m$ and $P^n \ll P^m$ for $n, m\geq k $.
\end{lemma}

Lemma~\ref{lemma:seq_con} implies that there exists a reference measure (e.g., $\initmeas$) that uniformly dominates the entire sequence. 
%For simplicity, we let the reference measure be the initial measure denoted by $P^0$.
% Therefore, we may write the algorithm in terms of Radon-Nikodym derivatives with respect to this particular measure. 
% We use whichever form is better suited for a given scenario. \tmw{Do we?} Nope. Revised.
Therefore, we may write the algorithm in terms of Radon-Nikodym derivatives with respect to this particular measure. 
Note that this is a weaker assumption than assuming the various probability measures admit Radon-Nikodym derivatives with respect to dominating measures on $(\pspace,\pborel)$ and $(\dspace,\dborel)$ that is assumed when applying the iterative density-based DCI approach summarized in Algorithm~\ref{alg:main} below. 
%\tmw{There's a disconnect here.  Algorithm 1 is Iterative DCI.} Fixed above. 

\begin{algorithm}\caption{Iterative DCI}\label{alg:main}
\begin{algorithmic}
    \REQUIRE Predicted QoI samples, initialized $r$-values, observed densities, tolerances and $\max_\text{epoch}.$  
    \FOR {epoch = $1,\ldots, \max_{\text{epoch}}$}
        \FOR {each QoI subspace}
        \STATE {Normalize current $r$-values to have unit sample mean.}
            \STATE {Estimate predicted density on subspace with QoI samples and $r$-values.}
            \STATE {Multiplicatively update the $r$-values using DCI.}
            \IF {diagnostic of $r$-values within tolerance}
        	\STATE {Continue to next subspace.}
            \ELSE
                \STATE {Exit both loops.}
            \ENDIF
        \ENDFOR
        \IF {absolute or relative KL divergences of QoI marginals within tolerances}
        	\STATE {Exit loop.}
            \ELSE
                \STATE {Continue to next epoch.}
        \ENDIF
    \ENDFOR
    \ENSURE Updated $r$-values.
\end{algorithmic}
\end{algorithm}

The algorithm is presented at a high-level to avoid being overly prescriptive about how one may go about estimating densities and applying DCI at each iteration. 
We provide more specific computational details in the numerical results presented in Section~\ref{sec:numerics} including a discussion of the KL divergence tolerances (both absolute and relative) since the algorithm terminates due to these tolerances in all of the numerical results.
The diagnostic of $r$-values mentioned in the algorithm refers to~\eqref{eq:diagnostic} discussed in Section~\ref{sec:DCI_summary}.
Note that by normalizing the $r$-values to have unit sample mean at the start of each iteration through the QoI subspaces ensures that the diagnostic is checking the predictability assumption applies at each specific iteration.
The first numerical example of Section~\ref{sec:numerics} provides more details about this diagnostic in the context of both the iterative algorithm and joint DCI inversion. 
By QoI marginals, we are referring to the push-forward of the updated density at the end of each epoch or the observed density on each QoI subspace but {\em not} the individual components of each QoI subspace for subspaces that are multidimensional.

At a conceptual-level, the algorithm in this work operates much like an extension of the iterative proportional fitting procedure (IPFP) for bivariate densities as analyzed in~\cite{Ruschendorf1995}.
We do note that there does exist some prior work on extending the IPFP method to multiple marginals in the context of the Sinkhorn algorithm and optimal transport, but these require stronger assumptions and do not directly apply to our framework, e.g., see \cite{Marino2020}.
The IPFP is analogous to an alternating projection algorithm designed to work on a product space where iterative projections of marginals are utilized to systematically adjust a multivariate distribution to some specified marginals.
Unfortunately, this cannot be applied directly to the solutions of the subproblems since the marginals of these solutions are unknown on $\pspace$.
Instead, we use the joint push-forward through the QoI maps to connect the DCI updates with marginal projections on a product data space.

To help build some intuition about these projections, suppose $\R^3$ is equipped with an initial joint density $f^0(x,y,z)$.
Its $x$-marginal is $f_1^0(x):=\int_{\R^2}f^0(x,y,z)\,dy\,dz$.
We wish to replace this marginal with a specified density $f_1(x)$.
This is accomplished by defining the new joint density $f^1$ as follows (assuming the support of $f_1$ is a subset of the support of $f_1^0$):
\begin{equation*}
    f^1(x,y,z) = \cfrac{f^0(x,y,z)}{\int_{\R^2} f^0(x,y,z)\,dy\,dz}f_1(x) = f^0(x,y,z)a^1(x),
\end{equation*}
where $a^1(x)=f_1(x)/f_1^0(x)$.
It follows that
\begin{equation*}
    \cfrac{f^1(x,y,z)}{f^0(x,y,z)} = a^1(x).
\end{equation*}
We can then adjust the $y$-marginal of $f^1$ to a target density $f_2(y)$ by multiplying by a factor $b^1(y)$.
This gives $f^2(x,y,z)=f^1(x,y,z)b^1(y)$.
Here $b^1(y)=f_2(y)/f_2^1(y)$ and $f_2^1$ denotes the $y$-marginal of $f^1$.
Similarly, we can adjust the $z$-marginal of $f^2$ to $f_3(z)$ by multiplying by a factor $c^1(z)$.
This gives $f^3(x,y,z)=f^2(x,y,z)c^1(z)$.
Here $c^1(z)=f_3(z)/f_3^2(z)$ and $f_3^2$ denotes the $z$-marginal of $f^2$.
After these three marginal adjustments, we have
\begin{equation*}
    \cfrac{f^3(x,y,z)}{f^0(x,y,z)} = a^1(x)b^1(y)c^1(z).
\end{equation*}
Continuing this process generates evolving multiplicative factors associated with the individual marginal constraints, just as the factors appearing in the joint push-forward representation of iterative DCI.
The product of these factors represents successive adjustments relative to the initial joint measure and does not assume that the QoI components are independent.

For the joint QoI map $\phi=(\phi_1,\ldots,\phi_k)$, write $Q^m:=P^m\circ\phi^{-1}$ and let $Q_i^m$ denote its $i$th marginal on $\mathcal{D}_i\subset\mathbb{R}^{d_i}$.
The iterative DCI updates impose one marginal constraint at a time, and we show below that all prescribed marginals are satisfied asymptotically.
However, convergence of the marginals alone does not establish convergence of the joint push-forward measures.
The next lemma shows that, for any sequence of joint measures dominated by $Q^0$ whose marginals approach the prescribed observations, the asymptotic KL divergence relative to $Q^0$ is bounded below by the minimum KL divergence over all feasible joint measures.
This lower bound, together with the KL Pythagorean identity and the additional assumption in Theorem~\ref{thm:main}, allows us to establish convergence to the I-projection.

\begin{lemma}\label{lemma:updatedTheorem-KL-liminf}
Let $Q^0$ be a probability measure on $\mathcal{D}:=\mathcal{D}_1\times\cdots\times\mathcal{D}_k$ and let $\q:=\{Q:\ Q_i=P_{\text{obs},i},\ 1\leq i\leq k\}$.
Suppose there exists $Q\in\q$ with $KL(Q || Q^0)<\infty$.
Then there exists a unique I-projection $Q^*:=\argmin_{Q\in\q}KL(Q || Q^0)$.
Moreover, for any sequence $\{\widetilde Q^n\}$ with $\widetilde Q^n\ll Q^0$ and $\|\widetilde Q_i^n-P_{\text{obs},i}\|_{\text{TV}}\to0$ for every $1\leq i\leq k$, we have
\begin{equation}\label{eq:updatedTheorem-KL-liminf}
	KL(Q^* || Q^0)\leq\liminf_{n\to\infty}KL(\widetilde Q^n || Q^0).
\end{equation}
\end{lemma}

\begin{proof}
A minimizing sequence in $\q$ has bounded KL divergence relative to $Q^0$, so its densities are uniformly integrable in $L^1(Q^0)$.
Weak compactness in $L^1(Q^0)$ and lower semicontinuity of KL yield a minimizer, since the marginal constraints are preserved under weak $L^1$ convergence.
Strict convexity of KL gives uniqueness.
If the right-hand side of~\eqref{eq:updatedTheorem-KL-liminf} is infinite, the inequality is immediate.
Otherwise, choose a subsequence along which the KL divergences converge to this limit inferior.
The same compactness argument gives a further subsequence converging weakly in $L^1(Q^0)$ to the density of a probability measure $\overline Q$.
Testing against bounded functions of each coordinate and using the assumed marginal convergence shows that $\overline Q\in\q$.
Thus, lower semicontinuity gives
\begin{equation*}
	KL(Q^* || Q^0)\leq KL(\overline Q || Q^0)\leq\liminf_{n\to\infty}KL(\widetilde Q^n || Q^0).
\end{equation*}
\end{proof}

The next theorem establishes conditions under which the iterative sequence converges in KL divergence and total variation to the I-projection of the initial measure onto the set of solutions to the GSIP.
We use a weaker form of the usual predictability assumption since finite-KL feasibility guarantees the absolute continuity needed to define each exact DCI update without requiring an upper bound on the associated Radon-Nikodym derivative.
However, for practical implementations involving numerical density estimation, we generally seek to satisfy the stronger, usual one-sided predictability assumption to control the multiplicative $r$-values used in the algorithm.
Finally, we impose an asymptotic entropy-consistency condition requiring that, along a subsequence of epochs, the KL divergence of the joint push-forward from its initial measure approaches the cumulative KL divergence of the successive marginal updates.
In light of the preceding lemma and the KL Pythagorean identity, this condition ensures that the accumulated adjustments asymptotically account for the minimum KL divergence required to satisfy all the marginal constraints.
We also acknowledge the recent independent work of Shi and Estep \cite{ShiEstep2026}, who identify the issue with Step 2 of the original proof of Theorem 5.1 and establish convergence of the iterative procedure under additional regularity assumptions.
Their analysis combines the KL Pythagorean identity with boundedness and uniform integrability assumptions to establish an asymptotic equality between the KL divergence of the iterates from the initial measure and the cumulative KL divergences of the successive updates, thereby providing useful, explicit sufficient conditions for the asymptotic entropy-consistency condition considered here.
Our approach instead assumes this asymptotic entropy-consistency condition directly and uses a KL lower-bound argument to establish convergence to the I-projection without requiring these additional regularity assumptions.
An advantage of their analysis is that it provides more explicit sufficient conditions that may be verified in particular settings, whereas our asymptotic entropy-consistency condition, although less restrictive, can be difficult to verify directly.

\begin{theorem} \label{thm:main}
Let $P^0=\initmeas$ and suppose there exists a $P^\dagger\in\p$ such that $KL(P^\dagger || P^0)<\infty$.
For the resulting sequence of exact DCI updates, suppose also that the joint push-forward measures $Q^m=P^m\circ\phi^{-1}$ satisfy the asymptotic entropy-consistency condition
\begin{equation}\label{eq:updatedTheorem-entropy-assumption}
	\liminf_{n\to\infty}\left[KL(Q^{kn} || Q^0)-\sum_{m=1}^{kn}KL(Q^m || Q^{m-1})\right]\leq0.
\end{equation}
Then the I-projection of $P^0$ onto $\p$ is uniquely defined by
\begin{equation*}
	P^\infty:=\argmin_{P\in\p}KL(P || P^0),
\end{equation*}
and the sequence $P^n$ converges to $P^\infty$ in KL divergence and total variation.
\end{theorem}

\begin{proof}
{\bf Step 1 (Construct the family of joint push-forward measures):}
Consider the product space $\mathcal{D}:=\mathcal{D}_1\times\cdots\times\mathcal{D}_k$ equipped with its product Borel $\sigma$-algebra, and define $\phi=(\phi_1,\ldots,\phi_k):\Lambda\to\mathcal{D}$.
For each $m\geq0$, let $Q^m:=P^m\circ\phi^{-1}$ and denote its $i$th marginal by $Q_i^m=P^m\circ\phi_i^{-1}$.
By the DCI construction, $Q_i^{i+nk}=P_{\text{obs},i}$ for $n\geq0$ and $1\leq i\leq k$.

{\bf Step 2 (Represent the Radon-Nikodym derivatives):}
Fix $P^\dagger\in\p$ with $KL(P^\dagger || P^0)<\infty$.
We show inductively that the exact DCI updates are well-defined and $KL(P^\dagger || P^m)<\infty$ for every $m$.
Suppose this holds through iteration $m-1$, and let $i=1+(m-1)\bmod k$.
Since $P^\dagger$ satisfies the $i$th push-forward constraint, pushing both measures forward through $\phi_i$ cannot increase their KL divergence, so
\begin{equation*}
	KL(P_{\text{obs},i} || Q_i^{m-1})\leq KL(P^\dagger || P^{m-1})<\infty.
\end{equation*}
Thus $P_{\text{obs},i}\ll Q_i^{m-1}$, and the next DCI update is well-defined, with joint push-forward satisfying
\begin{equation}\label{eq:updatedTheorem-marginal-update}
	\frac{dQ^m}{dQ^{m-1}}(x)=\frac{dP_{\text{obs},i}}{dQ_i^{m-1}}(x_i).
\end{equation}
The KL Pythagorean identity for this single marginal update gives
\begin{equation*}
	KL(P^\dagger || P^{m-1})=KL(P^\dagger || P^m)+KL(P^m || P^{m-1}),
\end{equation*}
so $KL(P^\dagger || P^m)<\infty$ and the induction is complete.
The same argument applies to every feasible measure with finite KL divergence relative to $P^0$, and successive absolute continuity gives $P^m\ll P^0$ and $Q^m\ll Q^0$ for all $m$.
In particular, after $n$ complete $k$-epochs, we have the representation
\begin{equation}\label{eq:updatedTheorem-product}
	\frac{dQ^{kn}}{dQ^0}(x) = \prod_{j=1}^n \left[ \prod_{i=1}^k r_i^j(x_i) \right] = \prod_{i=1}^k q_{\mathcal{D}_i}^n(x_i),
\end{equation}
where
\begin{equation*}
	r_i^j(x_i):=\frac{dP_{\text{obs},i}}{dQ_i^{(j-1)k+i-1}}(x_i), \qquad q_{\mathcal{D}_i}^n(x_i):=\prod_{j=1}^n r_i^j(x_i), \qquad q_{\mathcal{D}_i}^0=1.
\end{equation*}
The functions $r_i^j$ are the individual multiplicative factors used to update the numerical $r$-values as in Algorithm~\ref{alg:main}, while $q_{\mathcal{D}_i}^n$ accumulates these adjustments over the first $n$ epochs, analogous to the factors $a^1$, $b^1$, and $c^1$ in the preceding discussion used to build intuition prior to Lemma~\ref{lemma:updatedTheorem-KL-liminf}.

{\bf Step 3 (Establish convergence of the joint push-forward measures):}
Let $\q:=\{Q:\ Q_i=P_{\text{obs},i},\ 1\leq i\leq k\}$.
The assumed finite-KL feasible measure $P^\dagger\in\p$ induces $Q:=P^\dagger\circ\phi^{-1}\in\q$ with $KL(Q || Q^0)\leq KL(P^\dagger || P^0)<\infty$ where we again use the fact that pushing two probability measures forward through the same measurable map cannot increase their KL divergence.
Let $Q^\infty:=\argmin_{Q\in\q}KL(Q || Q^0)$ denote the unique I-projection guaranteed by Lemma~\ref{lemma:updatedTheorem-KL-liminf}.
Each update in~\eqref{eq:updatedTheorem-marginal-update} is an I-projection onto one marginal constraint by Lemma~\ref{lemma:for_bac}.
For every $Q\in\q$ with $KL(Q || Q^0)<\infty$, the KL Pythagorean identity~\cite{Csiszar1975} therefore gives
\begin{equation}\label{eq:updatedTheorem-pythagorean}
	KL(Q || Q^0)=KL(Q || Q^N)+\sum_{m=1}^N KL(Q^m || Q^{m-1}).
\end{equation}
It follows that $\sum_{m=1}^\infty KL(Q^m||Q^{m-1})<\infty$ and hence, by Pinsker's inequality, $\|Q^m-Q^{m-1}\|_{\text{TV}}\to0$.
Each marginal is fitted exactly once per epoch, so the triangle inequality gives
\begin{equation}\label{eq:updatedTheorem-marginal-convergence}
	\|Q_i^{kn}-P_{\text{obs},i}\|_{\text{TV}} \leq \sum_{m=kn-k+i+1}^{kn}\|Q^m-Q^{m-1}\|_{\text{TV}}\longrightarrow0.
\end{equation}
It follows that the prescribed marginals associated with the push-forward of the iterate obtained at the end of each epoch converge without the additional entropy assumption.

Let $S_n:=\sum_{m=1}^{kn}KL(Q^m||Q^{m-1})$ and $H_n:=KL(Q^{kn}||Q^0)$.
Here, $S_n$ measures the cumulative KL divergence associated with the successive marginal updates over the first $n$ epochs, whereas $H_n$ measures the KL divergence of the resulting joint QoI measure from the joint push-forward of the initial measure.
By construction, $S_n\uparrow S_\infty<\infty$ and applying both Lemma~\ref{lemma:updatedTheorem-KL-liminf} and~\eqref{eq:updatedTheorem-pythagorean} gives
\begin{equation}\label{eq:updatedTheorem-entropy-sandwich}
	S_\infty \leq KL(Q^\infty || Q^0)\leq\liminf_{n\to\infty}H_n.
\end{equation}
Since $S_n\to S_\infty$, the entropy assumption given by~\eqref{eq:updatedTheorem-entropy-assumption} implies $\liminf_{n\to\infty}H_n\leq S_\infty$, and hence equality holds throughout~\eqref{eq:updatedTheorem-entropy-sandwich}.
Moreover, the partial sums of the KL divergences in~\eqref{eq:updatedTheorem-pythagorean} increase to $S_\infty$ over all iterations, so substitution of $Q=Q^\infty$ into~\eqref{eq:updatedTheorem-pythagorean} and rearranging terms gives, for every $N$,
\begin{equation*}
	KL(Q^\infty || Q^N)=KL(Q^\infty || Q^0)-\sum_{m=1}^N KL(Q^m || Q^{m-1})\longrightarrow0.
\end{equation*}

{\bf Step 4 (Apply disintegration):}
Note that $P^0$ admits a disintegration through $\phi$ such that $P^0(\cdot) = \int_\mathcal{D} P^0_{\omega}(\cdot) dQ^0(\omega)$. 
By the nature of the procedure and construction of $Q^m$, $P^0_{\omega}$ stays unchanged so $P^m(\cdot) = \int_\mathcal{D} P^0_{\omega}(\cdot) dQ^m(\omega)$.
Define the feasible parameter measure
\begin{equation*}
	P^\infty(\cdot):=\int_{\mathcal{D}}P_\omega^0(\cdot)\,dQ^\infty(\omega)\in\p.
\end{equation*}
Since $P^m$ and $P^\infty$ have the same conditional measures, the KL chain rule under disintegration gives
\begin{equation*}
	KL(P^\infty || P^m)=KL(Q^\infty || Q^m)\longrightarrow0.
\end{equation*}
For any $P\in\p$, the same chain rule gives
\begin{equation*}
	KL(P || P^0)\geq KL(P\circ\phi^{-1} || Q^0)\geq KL(Q^\infty || Q^0)=KL(P^\infty || P^0).
\end{equation*}
Thus, $P^\infty$ is the I-projection of $P^0$ onto $\p$, and uniqueness follows from strict convexity of KL.
Convergence in total variation follows from Pinsker's inequality.
\end{proof}

\begin{rem}
As discussed above, the usual one-sided predictability assumption remains important for the practical implementation of DCI, even though it is not required for the theoretical convergence result.
At iteration $m$, the bound $dP_{\text{obs},i}/dQ_i^{m-1}\leq C_m$ for the marginal being updated controls the corresponding multiplicative factors used to update the numerical $r$-values.
The additional entropy-consistency assumption is used only after the prescribed marginals have been shown to converge, and it requires a one-sided inequality along a subsequence rather than uniform control of the individual multiplicative factors.
In particular, it does not require $Q^0$ to be equivalent to the product of the observed marginal measures and allows for singular joint push-forwards.
The assumption is sufficient for KL convergence, and hence total-variation convergence, but is not claimed to hold for every numerical example.
The use of marginal adjustments on a common support is also central to the classical IPFP analysis~\cite{Ruschendorf1995}.
\end{rem}

The limit is clearly dependent upon the choice of $P^{0}=\initmeas$. 
As previously mentioned, this is typically specified by the user and reflects prior knowledge or beliefs about the unknown distribution, which is conceptually similar to the role of the prior distribution in Bayesian inference. 
However, in the absence of a specified initial distribution and assuming $\Lambda$ is a precompact subset of a finite-dimensional Euclidean space, we recommend using a uniform distribution as a default option based on the maximum entropy principle \cite{Guiasu1985}.
The following corollary is then a simple result of the limit being an I-projection. 

\begin{corollary}
If $P^{0}$ is the uniform measure and the conditions of Theorem~\ref{thm:main} hold, then $P^{\infty}$ has the maximal entropy compared to any other $P \in \p$.    
\end{corollary}

\section{Numerical Examples}\label{sec:numerics}

We present two numerical examples to illustrate various aspects of the iterative approach. 
The first example involves the use of linear QoI maps on a low-dimensional parameter space.
This example is intended to help build intuition regarding the similarities and differences of the iterative solution to the solution obtained from using either the joint QoI distribution or the product of the observed QoI marginals on the product QoI space. 
It also serves to highlight the key feature that the iterative solution exists even in scenarios where the product of the observed QoI marginals fails to satisfy the predictability assumption. 
Such scenarios can arise, for instance, when the observational data for the components of a vector-valued QoI map are obtained asynchronously. 

The second example involves a high-dimensional parameter space for a partial differential equation model with various low- and high-dimensional QoI maps.
This example illustrates the computational utility of the iterative approach in cases where approximating the joint observational density is complicated due to a variety of factors including the dimension of the space, limited observational sample size, and potentially complex and nonlinear relationships between the various components. 

\subsection{Linear QoI Maps}\label{sec:example1}

\subsubsection{Numerical Setup}
We consider a 2-dimensional parameter space given by $\Lambda=[0,1]^2\subset\mathbb{R}^2$. 
% We first consider linear QoI maps. epochs are terminated once diagnostic (i.e., $\mathbb{E}_\text{init}(r)$) is more than 10\% from target value of unity. 
The initial distribution is taken to be uniform on $\pspace$.
The data-generating distribution is defined by assuming independent Beta distributions on both parameters. Specifically, we assume a $Beta(3,9)$ distribution on $\lambda_1$ and a $Beta(8,2)$ distribution on $\lambda_2$. 
We initially consider the linear QoI maps given by
\begin{align}
    Q_1(\lambda) &= 2\lambda_1 + \lambda_2, \label{eq:linear_Q1}\\
    Q_2(\lambda) &= 2.5\lambda_1 + 0.5\lambda_2. \label{eq:linear_Q2}
\end{align}

Computationally, we generate $\expnumber{1}{3}$ samples from both the initial and data-generating distributions.
The predicted and observed QoI samples are obtained by evaluating the QoI maps on the samples drawn from the initial and data-generating distributions, respectively.
These sample sets are shown in Figure~\ref{fig:ex1_param_and_QoI_samples}.
All densities are approximated with a standard Gaussian KDE with Scott's rule used to determine the bandwidth parameter~\cite{TS92}.

% \begin{figure}
%     \centering
%     \includegraphics[width=0.5\linewidth]{figures/ex1_linear_QoI_predicted_vs_observed.png}
%     \caption{The QoI samples from the initial and DG distributions using the QoI maps from~\eqref{eq:linear_Q1} and~\eqref{eq:linear_Q2}.}
%     \label{fig:ex1_linear_QoI_samples}
% \end{figure}
\begin{figure}
    \centering
    \includegraphics[width=0.25\linewidth]{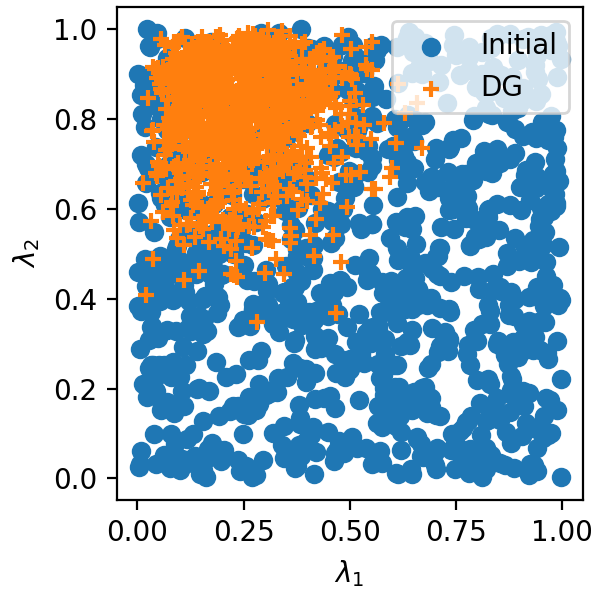}
    \includegraphics[width=0.25\linewidth]{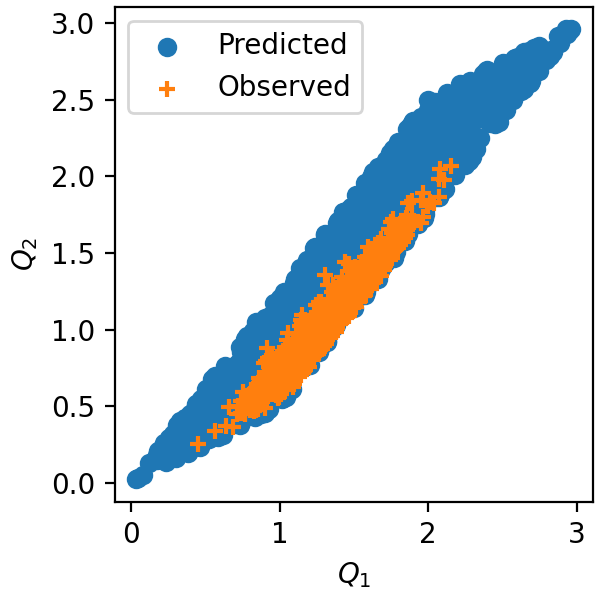}
    \caption{(Left) Uniform initial samples (blue circles) and data-generating samples (orange pluses) on $\Lambda=[0,1]^2\subset\mathbb{R}^2$. (Right) The predicted (blue circles) and observed (orange pluses) samples obtained by evaluating the QoI maps (eqs.~\eqref{eq:linear_Q1} and~\eqref{eq:linear_Q2}) on the initial and data-generating samples.}
    \label{fig:ex1_param_and_QoI_samples}
\end{figure}

For the algorithm, the stopping criteria were defined as follows.
The maximum of number of epochs was set to 100, which was determined by trial-and-error to ensure that the algorithm stopped due to other criteria.
The tolerance was set to 0.1 for the numerical diagnostic, i.e., if the sample average of the $r$-values deviates by more than 10\% from the theoretically expected value of unity, then the algorithm terminates.
This value of 0.1 was chosen based on the prior experience of the authors regarding typical variation seen in the diagnostic for DCI problems where the predictability assumption is satisfied. 
The tolerance for the KL divergence of all QoI marginals associated with the observed distribution and push-forward of the updated distribution after each epoch was set to $\expnumber{1}{-7}$.
In other words, this tolerance is reached at epoch $k$ only if every QoI marginal associated with the observed distribution and the push-forward of the updated distribution is less than $\expnumber{1}{-7}$.
This value was chosen based on some numerical experiments so that it was smaller than typically observed KL divergence values between KDE estimates obtained from different sample sets taken from the same distribution. 
Finally, the relative tolerance for the difference in such KL divergences obtained between successive epochs was set to $\expnumber{1}{-2}$. 
In other words, this tolerance is reached at epoch $k+1$ only if the KL divergences of every QoI marginal associated with the observed and updated distributions are within $1\%$ of those computed at epoch $k$.

\subsubsection{Iterative Results}

The relative tolerance stopping criteria was triggered at the 46th epoch where the push-forward of the updated distribution at this epoch produced marginal QoI densities that both had KL divergences from the associated observed marginals that were within 1\% from the values computed from the 45th epoch. 
We illustrate some of the results in Figures~\ref{fig:ex1_linear_epoch1_results}-\ref{fig:ex1_linear_epoch46_results}.
For reference, the top-right plot in each of these figures shows the KDE of the data-generating samples shown as the orange pluses in the left plot of Figure~\ref{fig:ex1_param_and_QoI_samples}.
The top-left and top-middle plots in each of these figures show the weighted KDE computed on the initial samples shown as the blue circles in Figure~\ref{fig:ex1_param_and_QoI_samples} with weights determined at the first (left) and second (middle) iteration of each epoch.
The bottom two plots in each of these figures show the KDE estimates of the marginal densities on $Q_1$ (left) and $Q_2$ (right) associated with the observed distribution and push-forward of the updated distribution at the end of each epoch.

\begin{figure}
    \centering
    \includegraphics[width=0.75\linewidth]{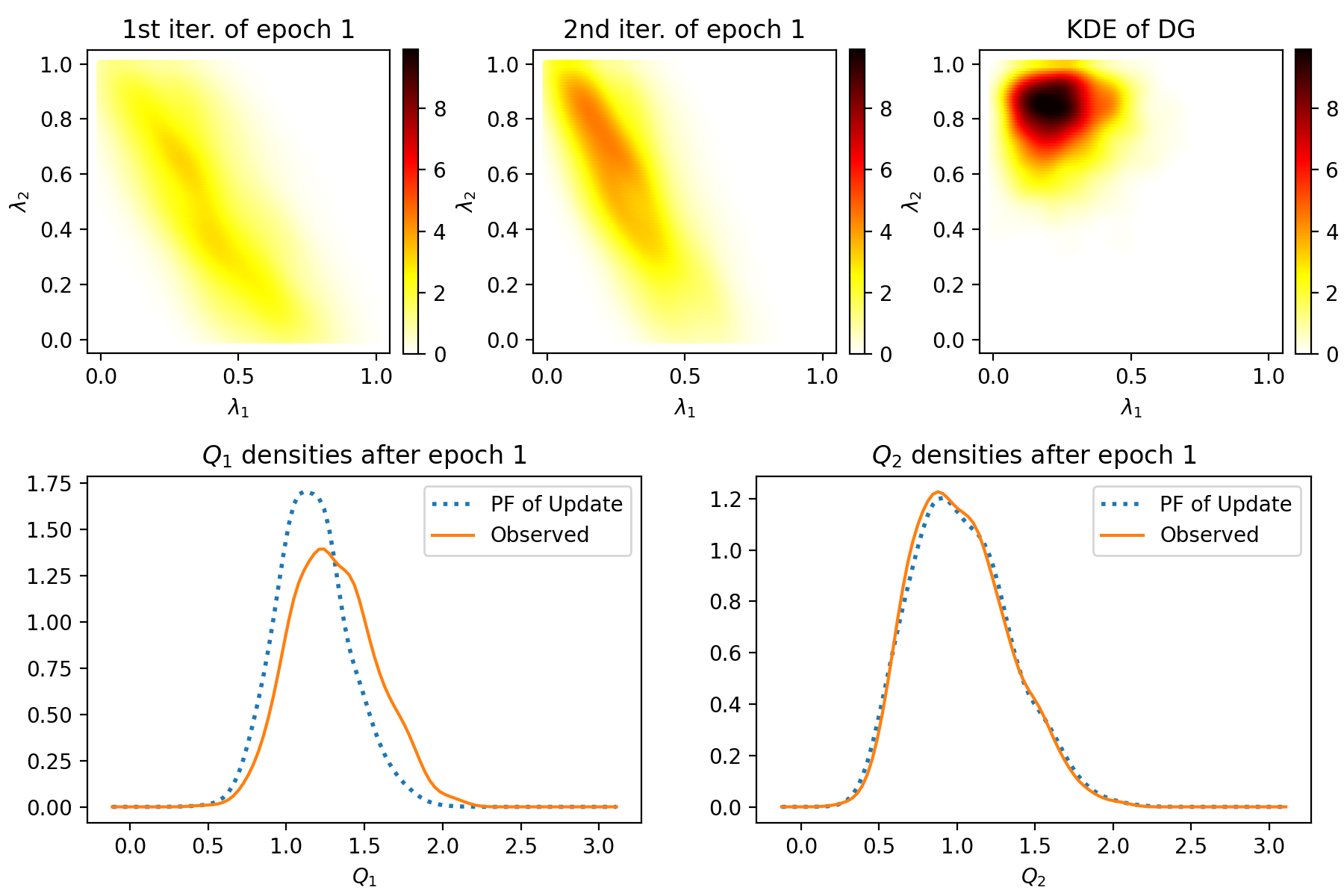}
    \caption{The KDE estimates of densities associated with Figure~\ref{fig:ex1_param_and_QoI_samples} after the first epoch of iterations. (Top Row): The updated density obtained after using the $Q_1$ map in the first iteration of this epoch (left), after using the $Q_2$ map in the second iteration of this epoch (middle), and the data-generating density estimate (right). (Bottom Row): The marginal QoI densities for $Q_1$ (left) and $Q_2$ (right) associated with the push-forward of the updated density (blue dashed curves) and observed data (orange solid curves).}
    \label{fig:ex1_linear_epoch1_results}
\end{figure}

The top-left and top-middle plots of Figure~\ref{fig:ex1_linear_epoch1_results} show a significant change both between the first two iterations of the first epoch as well as in relation to the uniform initial distribution.
The bottom two plots illustrate a few things worth noting.
First, at the end of the first epoch, the marginals associated with the push-forward of the iterative solution produces a better estimate of the $Q_2$ observed KDE than of the $Q_1$ observed KDE.
This is expected due to the alternating projection that occurs in the algorithm as well as the specific ordering of the maps. 
The reverse would be observed if we swapped the order of the maps although there would be little discernible difference in the final result obtained after the termination of the algorithm.
Second, the discrepancies between both of the marginals associated with the push-forward of the iterative updated distribution and the observed distribution are significant enough that more epochs are clearly necessary before any termination criteria are reached.

\begin{figure}
    \centering
    \includegraphics[width=0.75\linewidth]{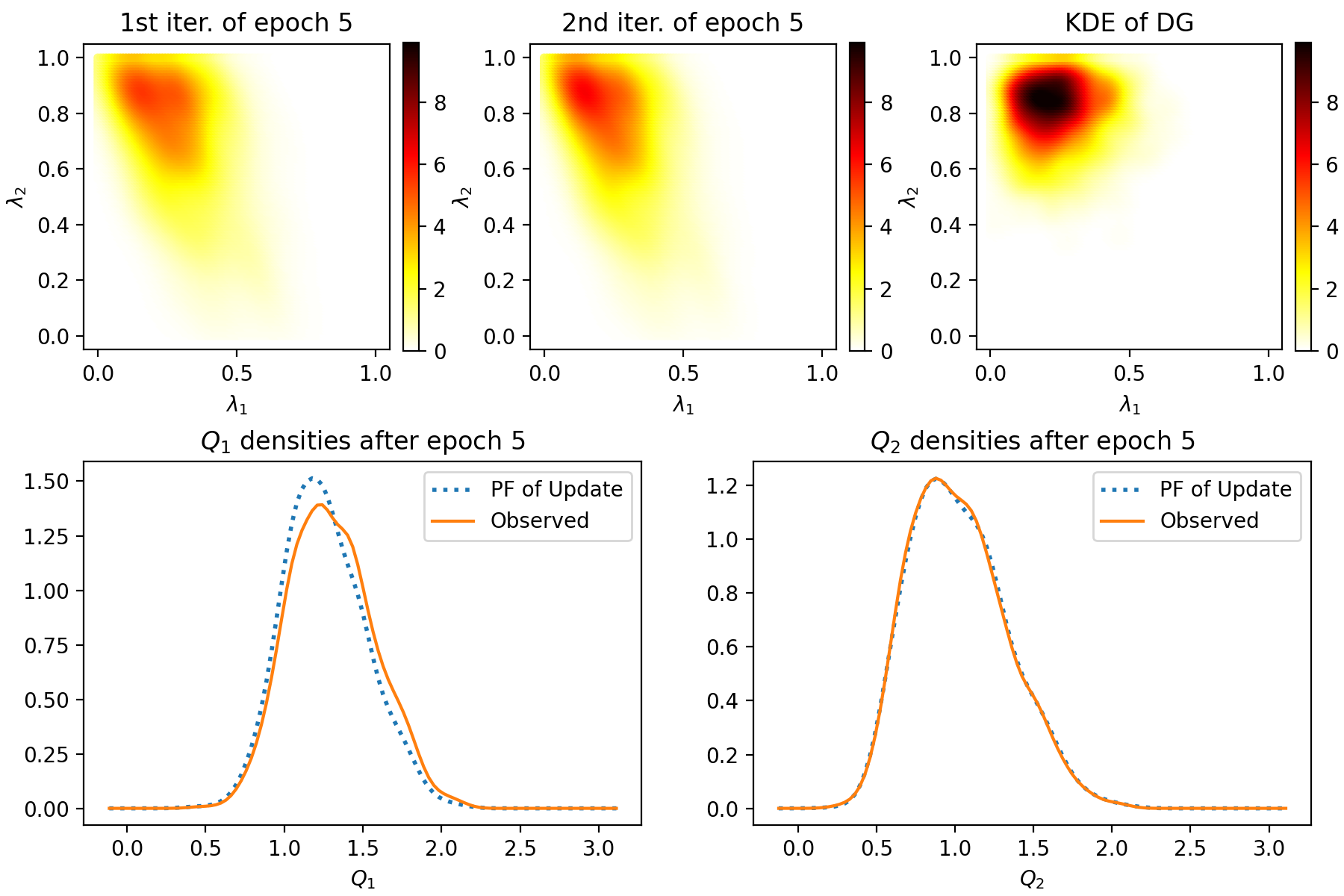}
    \caption{The KDE estimates of densities associated with Figure~\ref{fig:ex1_param_and_QoI_samples} after the fifth epoch of iterations. (Top Row): The updated density obtained after using the $Q_1$ map in the first iteration of this epoch (left), after using the $Q_2$ map in the second iteration of this epoch (middle), and the data-generating density estimate (right). (Bottom Row): The marginal QoI densities for $Q_1$ (left) and $Q_2$ (right) associated with the push-forward of the updated density (blue dashed curves) and observed data (orange solid curves).}
    \label{fig:ex1_linear_epoch5_results}
\end{figure}

The top-left and top-middle plots of Figure~\ref{fig:ex1_linear_epoch5_results} show that after five epochs, the difference in the solution between iterations within the epoch are not as large as with earlier epochs.
Moreover, more of the probability mass is becoming concentrated around the data-generating distribution samples as evidenced by comparing these plots with the corresponding plots of Figure~\ref{fig:ex1_linear_epoch1_results} and the KDE of the data-generating distribution shown as the top right plot.
While the bottom-right plot indicates solid agreement between the $Q_2$-marginals of the push-forward of the iterative updated distribution and the observed distribution after five epochs, the bottom-left plot indicates that convergence of the alternating projection is still not yet achieved across both densities.

\begin{figure}
    \centering
    \includegraphics[width=0.75\linewidth]{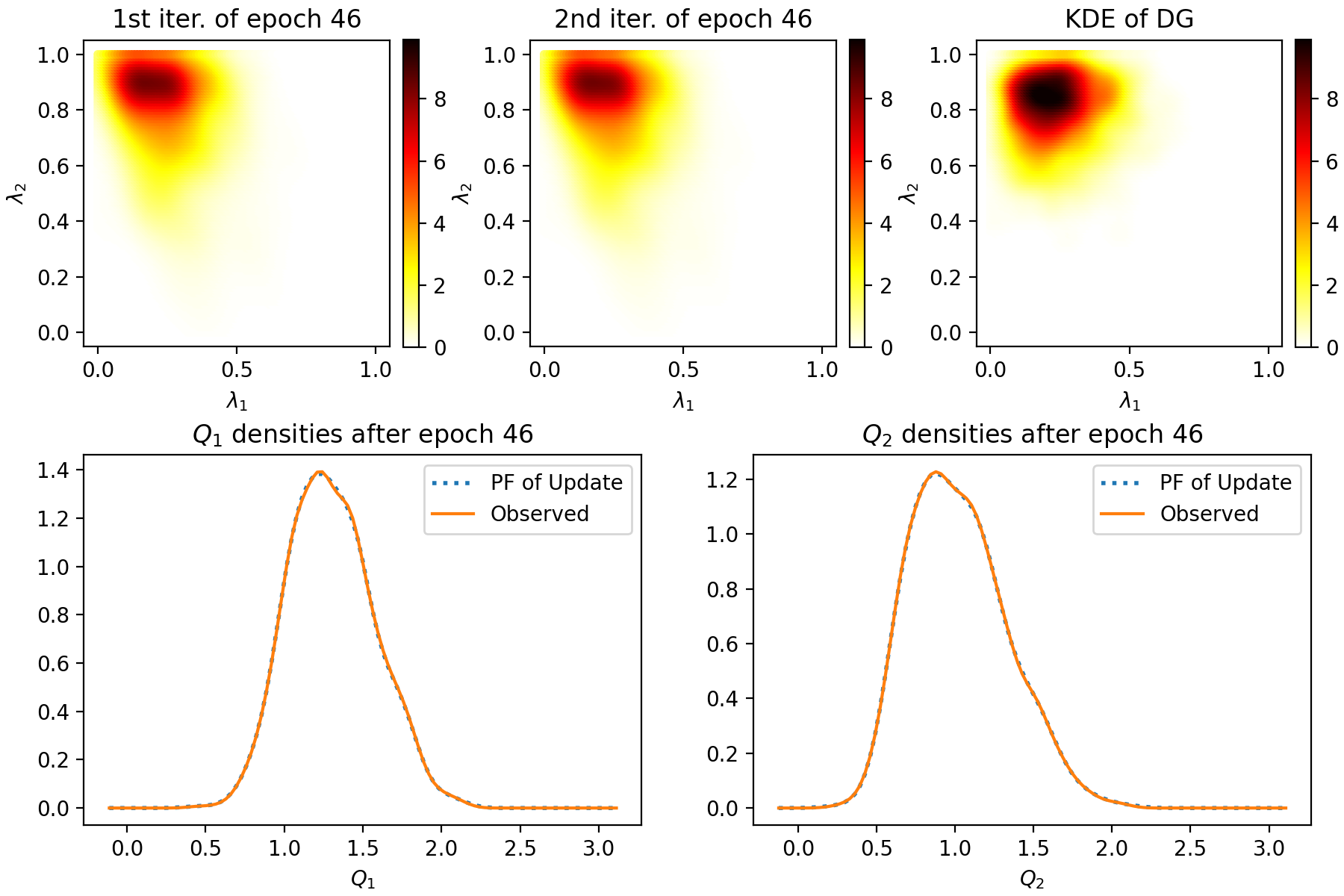}
    \caption{The KDE estimates of densities associated with Figure~\ref{fig:ex1_param_and_QoI_samples} after the final epoch (46th) of iterations. (Top Row): The updated density obtained after using the $Q_1$ map in the first iteration of this epoch (left), after using the $Q_2$ map in the second iteration of this epoch (middle), and the data-generating density estimate (right). (Bottom Row): The marginal QoI densities for $Q_1$ (left) and $Q_2$ (right) associated with the push-forward of the updated density (blue dashed curves) and observed data (orange solid curves).}
    \label{fig:ex1_linear_epoch46_results}
\end{figure}

The results at the 46th epoch where the algorithm terminated are shown in Figure~\ref{fig:ex1_linear_epoch46_results}.
While the termination criteria indicates these results are nearly identical to those obtained at the 45th epoch, the results are practically indistinguishable from any of those obtained after 20 epochs.
The interested reader can verify this with the supplementary material provided.
This provides numerical evidence that the tolerances utilized for identifying ``numerical convergence'' are sufficiently small for this example.
Consequently, it is difficult to identify any significant differences between results obtained at individual iterations of this epoch in the parameter space or in the marginals of the push-forward of the updated distribution at the end of the epoch.

\begin{figure}
    \centering
    \includegraphics[width=0.5\linewidth]{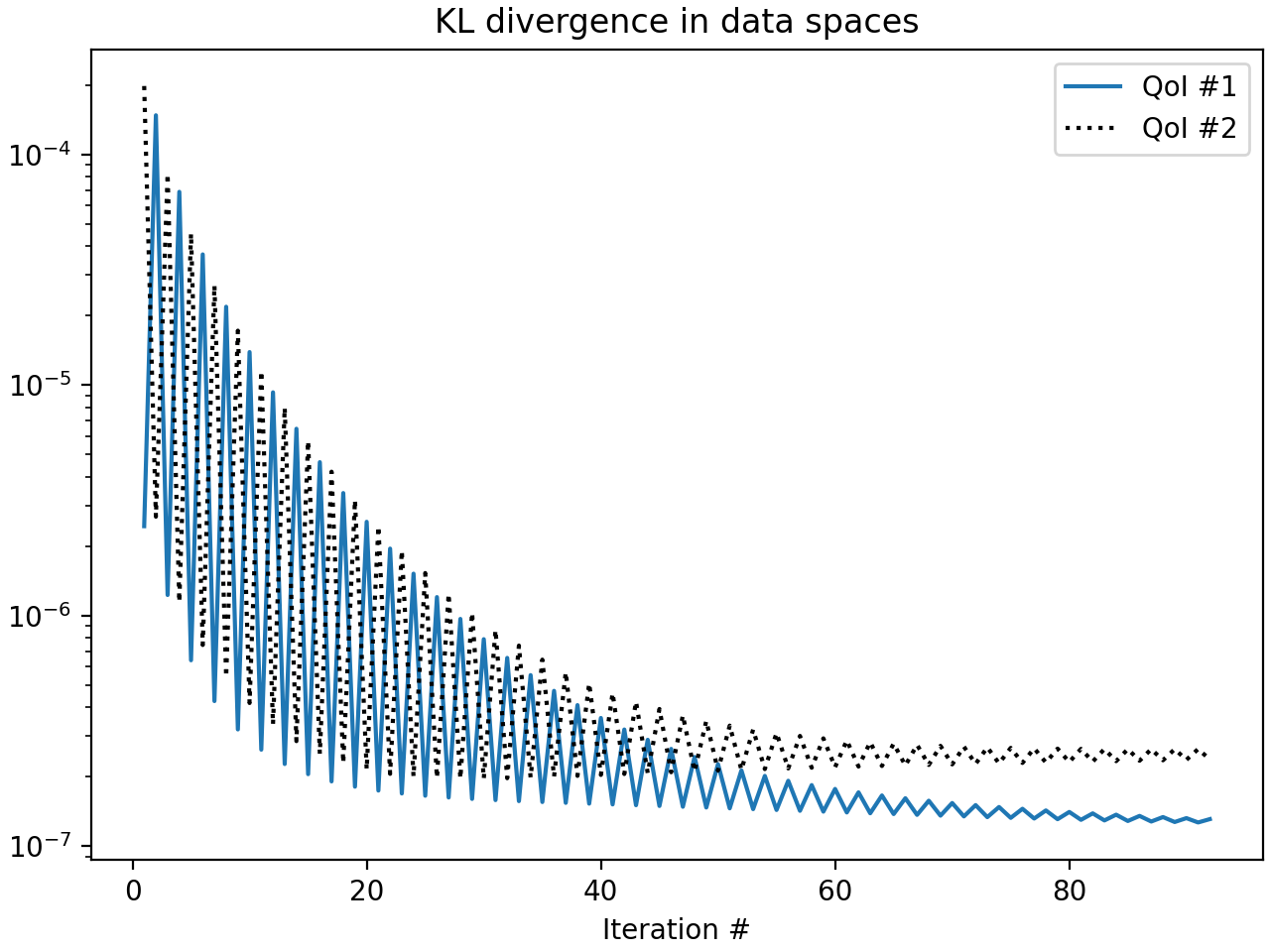}
    \caption{The KL divergences between the QoI marginals associated with the push-forward of the updated distribution at each iteration and the observed distribution. The results for the $Q_1$ marginal are shown as the blue solid curve and for the $Q_2$ marginal are shown as the black dotted curve.}
    \label{fig:ex1_linear_KL_data_iterations}
\end{figure}

Numerical evidence of convergence is further supported by Figure~\ref{fig:ex1_linear_KL_data_iterations}, which shows the rapid decrease in the KL divergence between the QoI marginals associated with the observed distribution and the push-forward of the updated distribution at each iteration.
Here, we show results for each iteration (of which there are 92 total) instead of each epoch (of which there are 46) to highlight the impact of alternating between the $Q_1$ and $Q_2$ distributions at each iteration within an epoch that produce oscillations in the decay trend. 
The magnitude of the oscillations is observed to quickly decay before essentially becoming flat with minor oscillations smaller than $\expnumber{1}{-6}$ in magnitude after 30-40 iterations (i.e., after 15-20 epochs). 

\subsubsection{Joint DCI Results}

\begin{figure}
    \centering
    \includegraphics[width=0.75\linewidth]{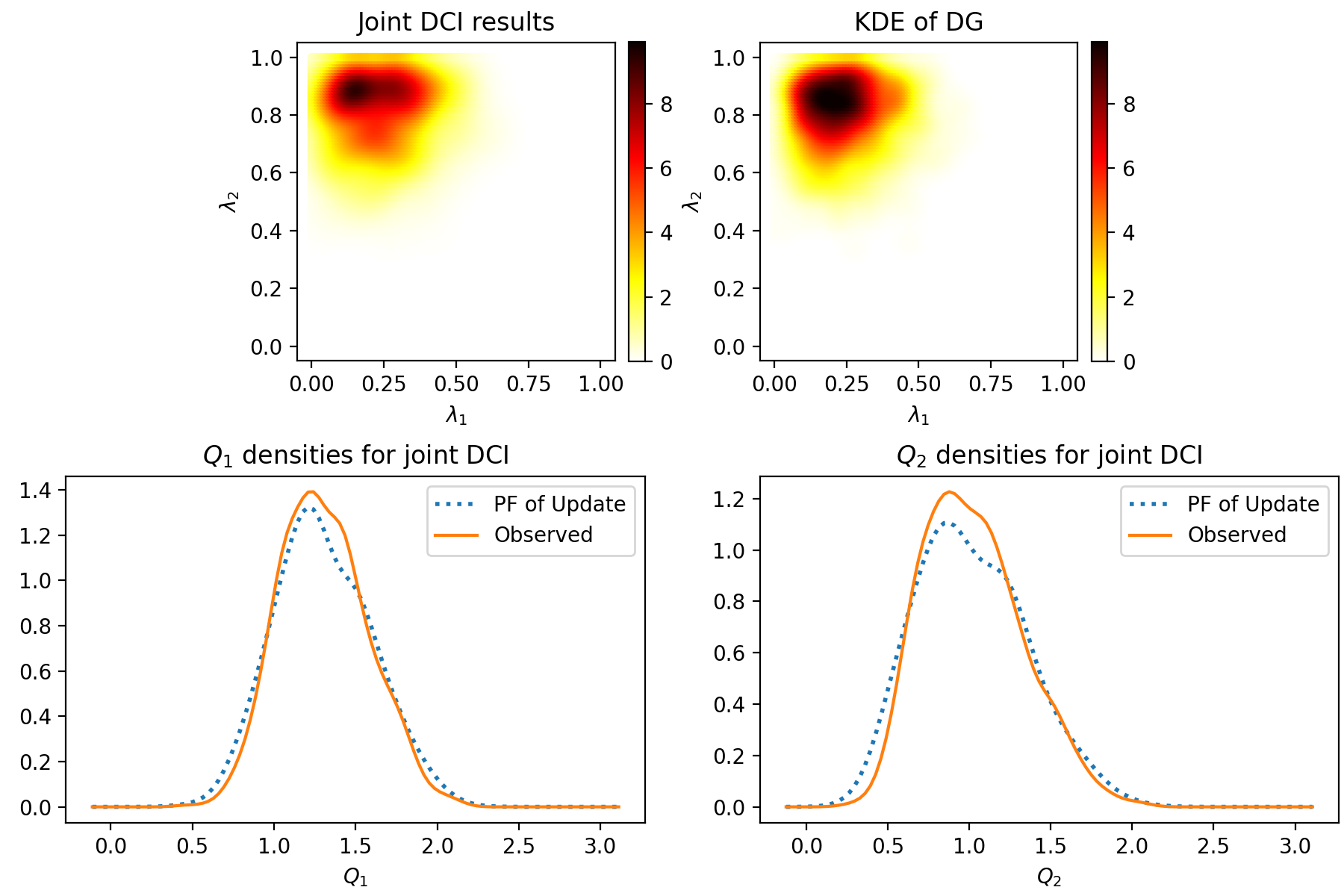}
    \caption{(Top Row:) The updated density obtained by using both the joint QoI data shown in Figure~\ref{fig:ex1_param_and_QoI_samples} simultaneously (left) and the data-generating density estimate (right). (Bottom Row): The marginal QoI densities for $Q_1$ (left) and $Q_2$ (right) associated with the push-forward of the updated density (blue dashed curves) and observed data (orange solid curves).}
    \label{fig:ex1_linear_joint_results}
\end{figure}

We now compare the above results obtained for the iterative example to those obtained from the joint inversion.
Figure~\ref{fig:ex1_linear_joint_results} summarizes the results obtained from the inversion of the observed distribution on the joint QoI space.
Note that this assumes that it is is possible to observe, simultaneously, the $Q_1$ and $Q_2$ data, which is not assumed in the iterative approach.
It is readily apparent from the QoI marginal plots that the joint approach suffers from some density approximation errors that the iterative approach was able to resolve over its epochs.
The errors in density approximation methods are often written as functions of sample size and dimension. 
However, in the authors' experience, such errors are exacerbated by several other factors.
One factor is the ``skewness'' of the joint data space qualitatively described as ``acute angles'' forming along boundaries of the data space.
Note that some skewness is observed in Figure~\ref{fig:ex1_param_and_QoI_samples}.
Another factor is boundary effects that can arise when the observed distribution has probability mass concentrated near the tail-end of a predicted distribution or in an area that is sparsely sampled such as near a boundary of the data space.
We again note that this is observed in Figure~\ref{fig:ex1_param_and_QoI_samples}.
Taken together, the diagnostic we obtain for the joint DCI inversion is approximately $1.138$, which is more than 10\% from the theoretical value of unity.
In the authors' experience, when the predictability assumption is theoretically satisfied by the exact densities, numerical approximations that produce diagnostics within 10\% of the target value of unity are expected due to approximation errors from finite sampling, and, in many cases, reasonable approximations of the updated density are still made with diagnostics that deviate up to $15-20\%$ from the target value of unity since the deviation is attributed to finite sample error and not a violation of the predictability assumption.
The interested reader can verify with the supplemental material that this can be reduced to $1.068$ by increasing the sample size of both the initial and data-generating distributions from $\expnumber{1}{3}$ to $\expnumber{1}{4}$.
Although, it is also worth noting that such an increase in sample sizes results in the iterative method converging at the 12th epoch (instead of the 46th epoch) due to every QoI marginal associated with the push-forward of the updated distribution being within the absolute KL tolerance of the associated observed QoI marginal.

\begin{figure}
    \centering
    \includegraphics[width=0.75\linewidth]{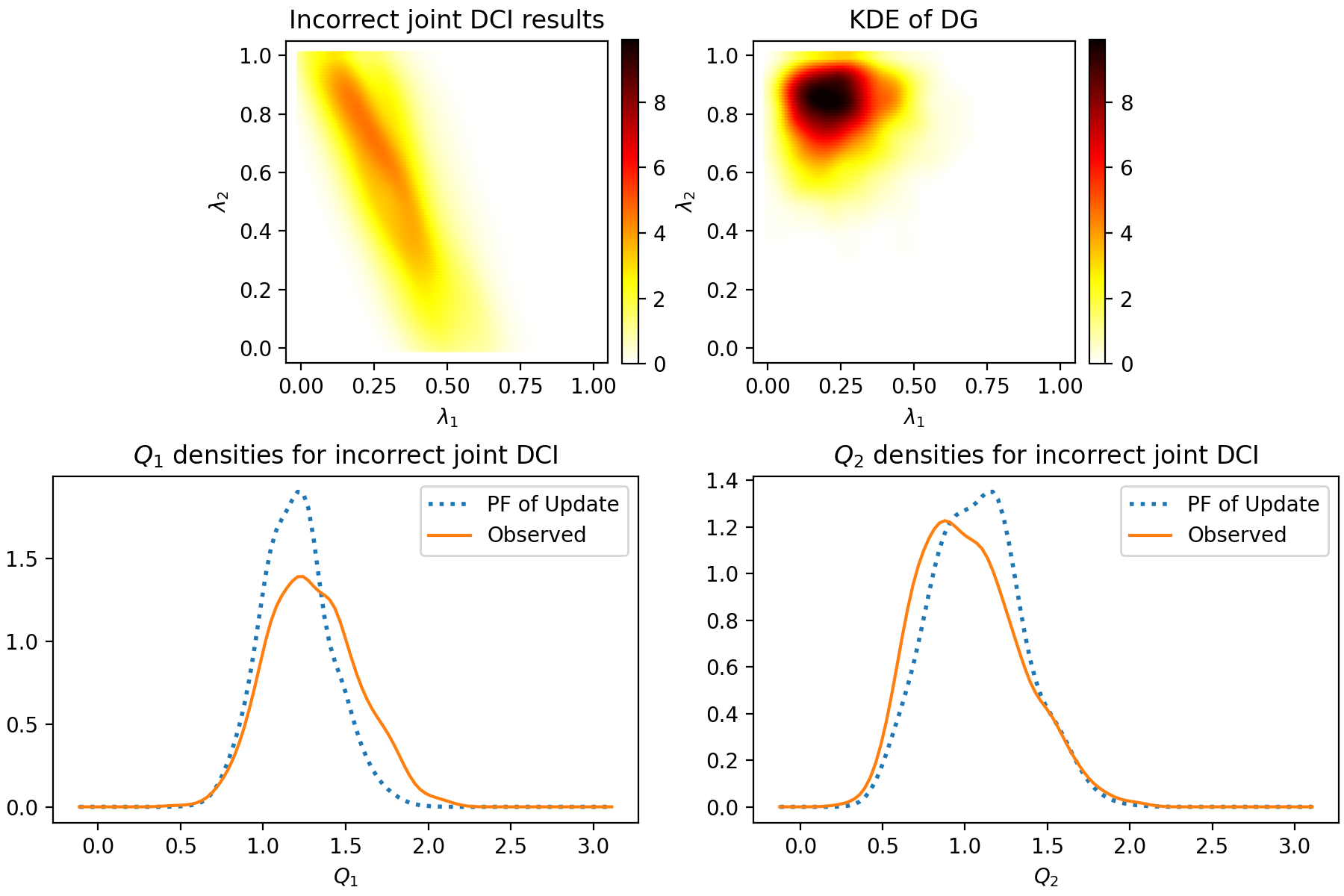}
    \caption{(Top Row:) The updated density obtained by incorrectly defining the observed QoI distribution in terms of the product of the marginals estimated from the QoI data shown in Figure~\ref{fig:ex1_param_and_QoI_samples} (left) and the data-generating density estimate (right). (Bottom Row): The marginal QoI densities for $Q_1$ (left) and $Q_2$ (right) associated with the push-forward of the updated density (blue dashed curves) and observed data (orange solid curves).
    This represents an incorrect approach to DCI on joint QoI data that results in a violation of the predictability assumption.}
    \label{fig:ex1_linear_wrong_joint_results}
\end{figure}

It is worth emphasizing that approximating the observed density as the product of its marginals, which ignores the clear relationship between $Q_1$ and $Q_2$ shown in Figure~\ref{fig:ex1_param_and_QoI_samples}, leads to a violation of the predictability assumption, and produces a diagnostic of $0.572$, which is not improved in any significant way by increasing sample size. 
In the interest of completeness, the results for this incorrect joint inversion are shown in Figure~\ref{fig:ex1_linear_wrong_joint_results}.
This ``wrong'' example illustrates an important, but perhaps subtle, point about Step 2 in the proof of Theorem~\ref{thm:main}.
Specifically, while a joint product space is utilized to represent the joint Radon-Nikodym derivative as a product of individual Radon-Nikodym derivatives for each QoI, this is with respect to the initial push-forward on the product space (i.e., $Q^0$) {\em not} the Lebesgue measure.
In other words, the structure of the data space is implicit within this theoretical construct.
While the iterative approach does not need to explicitly construct the product of these Radon-Nikodym derivatives, this ``wrong'' example is demonstrating that the iterative inversion obtained by this ``product of Radon-Nikodym derivatives'' is, in fact, distinct from constructing the product of marginal densities and performing a joint inversion.

% KL Divergence of updated density from the joint inversion to the DG density (both approximated with KDEs and computed on samples from the initial): 2.404e-05.

% KL Divergence of iteratively obtained updated density to the DG density (both approximated with KDEs): 1.012e-04

% KL divergences of push-forwards of updated densities from joint obserserved: 2.404e-05 (joint) and 1.929e-04 (iterative)

% KL divergences of marginals of push-forwards of updated densities from observed: 1.012e-05 ($Q_1$, joint), 1.683e-05 ($Q_2$, joint), and 1.283e-07 ($Q_1$, iterative), and 2.629e-07 ($Q_2$, iterative).

\subsubsection{Increasing the QoI Dimension}

To conclude this example, suppose we add the following third QoI to the vector-valued map, 
\begin{align}
    Q_3(\lambda) &= -\lambda_1 + \lambda_2. \label{eq:linear_Q3}
\end{align}
The result is a joint data space defined as a 2d-planar object embedded within $\mathbb{R}^3$ as shown in Figure~\ref{fig:ex1_linear_3_QoI_samples}.

\begin{figure}[htbp]
    \centering
    \includegraphics[width=0.5\linewidth]{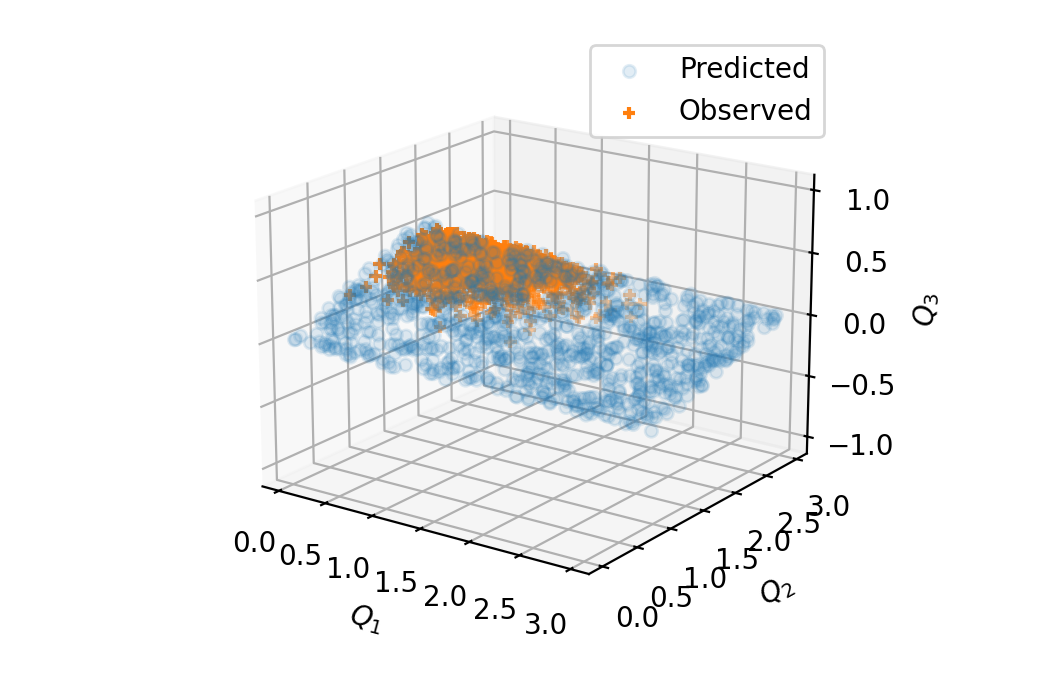}
    \caption{The joint data space for three linear QoI maps (eqs.~\eqref{eq:linear_Q1}-\eqref{eq:linear_Q3}) on a 2-dimensional parameter space results in a 2-dimensional planar object embedded in $\mathbb{R}^3$. The predicted (blue circles) and observed (orange pluses) samples associated with the initial samples shown in Figure~\ref{fig:ex1_param_and_QoI_samples} are shown here.}
    \label{fig:ex1_linear_3_QoI_samples}
\end{figure}

The joint DCI inversion now requires either more specialized methods for density estimation or pre-processing of the data before we can compute the updated density since traditional methods based on KDE fail due to the singular nature of the sample covariance matrix. 
The interested reader is encouraged to test this with the supplemental material where a linear algebra error occurs when a KDE is applied to the QoI data that correctly identifies the ``data appears to lie in a lower-dimensional subspace of the space in which it is expressed'' and the user is directed to ``Consider performing principal component analysis / dimensionality reduction'' to rectify the issue. 
While outside the scope of this work, there is active research on utilizing a machine-learning enhanced framework to construct low-dimensional QoI for DCI, e.g., see \cite{MSB+22, RHB25}. 

\begin{figure}
    \centering
    \includegraphics[width=0.75\linewidth]{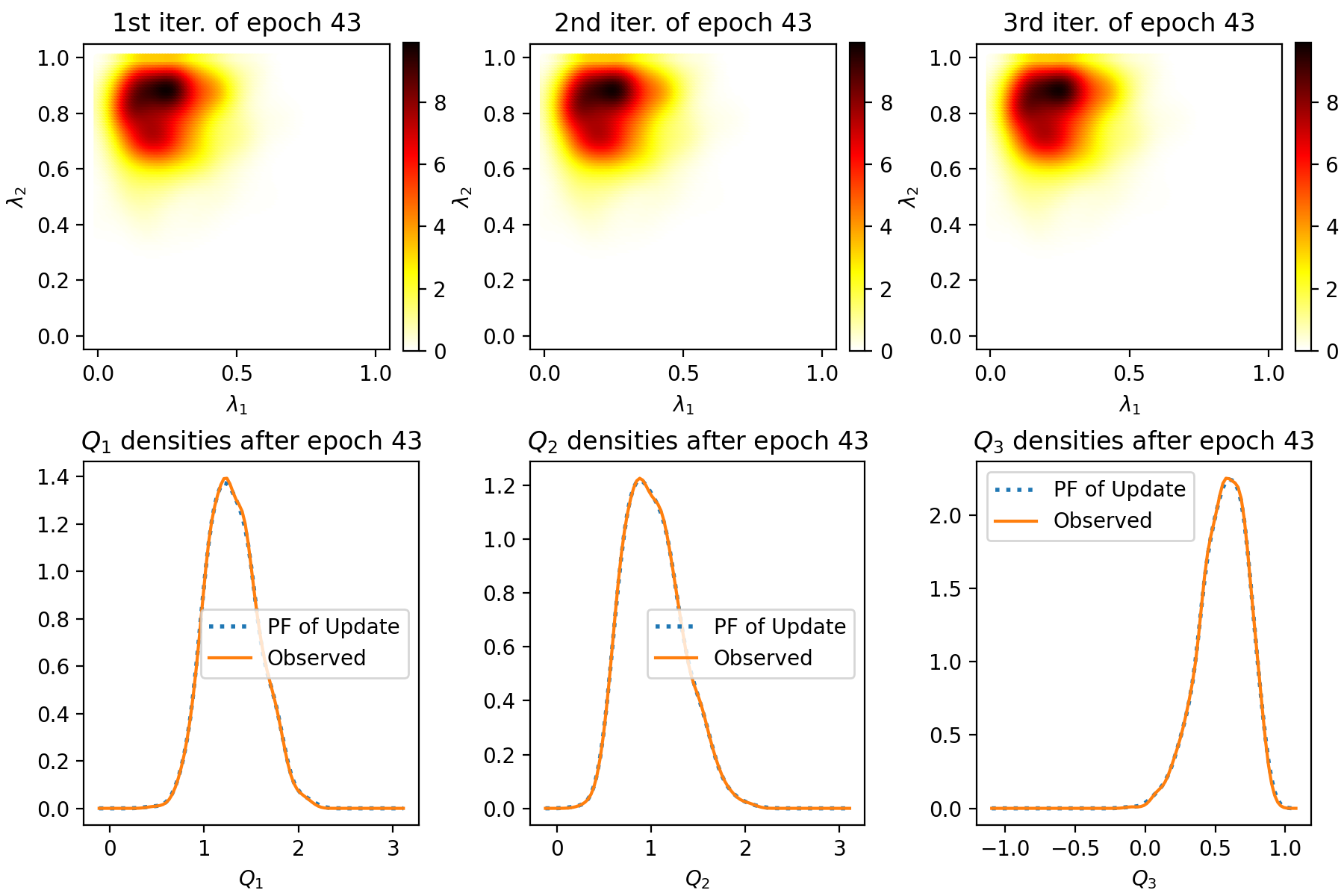}
    \caption{The KDE estimates of densities associated with parameter samples from Figure~\ref{fig:ex1_param_and_QoI_samples} and QoI samples from Figure~\ref{fig:ex1_linear_3_QoI_samples} after the final epoch (43rd) of iterations. (Top Row): The updated density obtained after using the $Q_1$ map in the first iteration of this epoch (left), after using the $Q_2$ map in the second iteration of this epoch (middle), and the $Q_3$ map in the third iteration of this epoch (right). (Bottom Row): The marginal QoI densities for $Q_1$ (left), $Q_2$ (middle), and $Q_3$ (right) associated with the push-forward of the updated density (blue dashed curves) and observed data (orange solid curves).}
    \label{fig:ex1_linear_3QoI_epoch43_results}
\end{figure}

The iterative method requires no such pre-processing of the data or the use of any specialized method for learning a density on a lower-dimensional manifold. 
In this case, the iterative algorithm terminated at the 43rd epoch due to the relative KL tolerance being reached.
The results of the final epoch are shown in Figure~\ref{fig:ex1_linear_3QoI_epoch43_results}. 
Since each epoch now includes three iterations, each iteration of the last epoch is shown in the top row of plots while the KDE of the data-generating distribution is omitted. 
However, comparing these to the prior top-row of plots involving two QoI (whether the iterative or joint results), we observe similar, if not better, qualitative agreement with the structure of the KDE of the data-generating distribution. 
This indicates that the inclusion of additional QoI that are ``geometrically redundant'' does not deteriorate the iterative DCI solution in this example whereas the joint solution is not even computable unless more sophisticated methods are utilized as mentioned above.

\subsection{A Higher-dimensional PDE-based Example}\label{sec:example2}

\subsubsection{Numerical Setup}

We adopt an example from~\cite{BJW18a} used to illustrate joint DCI solutions in high-dimensional spaces.
Specifically, we consider the following single-phase incompressible flow model:
\begin{equation}\label{eq:porous}
\begin{cases}
    -\nabla \cdot (K(\lambda) \nabla p) = 0, & \mathbf{x}=(x_1,x_2)\in\Omega = (0,1)^2,\\
    p = 1, & x_1=0, \\
    p = 0, & x_1=1, \\
    K\nabla p \cdot \mathbf{n} = 0, & x_2=0 \text{ and } x_2=1.
    \end{cases}
\end{equation}
Here, $p$ is the pressure field and $K$ is the permeability field which we assume is a scalar field given by a Karhunen-Lo\`eve expansion of the log transformation, $Y = \log{K}$, with
\begin{equation*}
    Y(\lambda) = \overline{Y} + \sum_{i=1}^\infty \lambda_i\sqrt{\eta_i}f_i(x_1,x_2),
\end{equation*}
where $\overline{Y}$ is the mean field and $\lambda_i$ are mutually uncorrelated random variables with zero mean and unit variance \cite{ganis2008stochastic,wheeler2011multiscale}.
Various QoI for this problem are considered from the evaluation of the pressure and/or norms of the velocity field, given by $-K\nabla p$, at various points within the physical domain $\Omega$ while the parameters are given by the $\lambda$ inputs to the permeability field.
We refer to $\lambda_i$ as the KL modes, which should not be confused with the KL divergence.
The eigenvalues, $\eta_i$, and eigenfunctions, $f_i$, are computed numerically using the following covariance function,
\begin{equation*}
    C_Y({\mathbf x},{\mathbf x}^\prime) = \sigma_Y^2 \exp \left[ -\frac{(x_1-{x}^\prime_1)^2}{2\zeta_1} - \frac{(x_2-{x}^\prime_2)^2}{2\zeta_2}\right],
\end{equation*}
where $\sigma_Y$ and $\zeta_i$ denote the variance and correlation length in the $i$th spatial direction, respectively.
For a given correlation length, the expansion is typically truncated once a sufficient fraction of the energy in the eigenvalues is retained~\cite{zhang2004efficient,ganis2008stochastic}. 
However, we set the correlation length to $0.1$ in each spatial direction and truncate the expansion at 100 terms for the sake of demonstration since these were observed to produce significant variability in both the permeability fields and ranges of potential QoI values.  

To approximate solutions to the PDE in~\eqref{eq:porous} we use the open source finite element software MrHyDE~\cite{mrhyde_user,mrhyde2023github} to implement a mixed finite element discretization with a physics-compatible discretization using the lowest-order Raviart-Thomas elements to approximate the velocity field, defined as $\mathbf{u} = -K \nabla p$, and a piecewise constant approximation of the pressure on a ($100\times 100$) spatial grid of $\Omega$.
In Figure~\ref{fig:ex2_perms}, we plot two random realizations of the permeability field generated using the truncated KL expansion.
\begin{figure}
    \centering
    \includegraphics[width=0.8\linewidth]{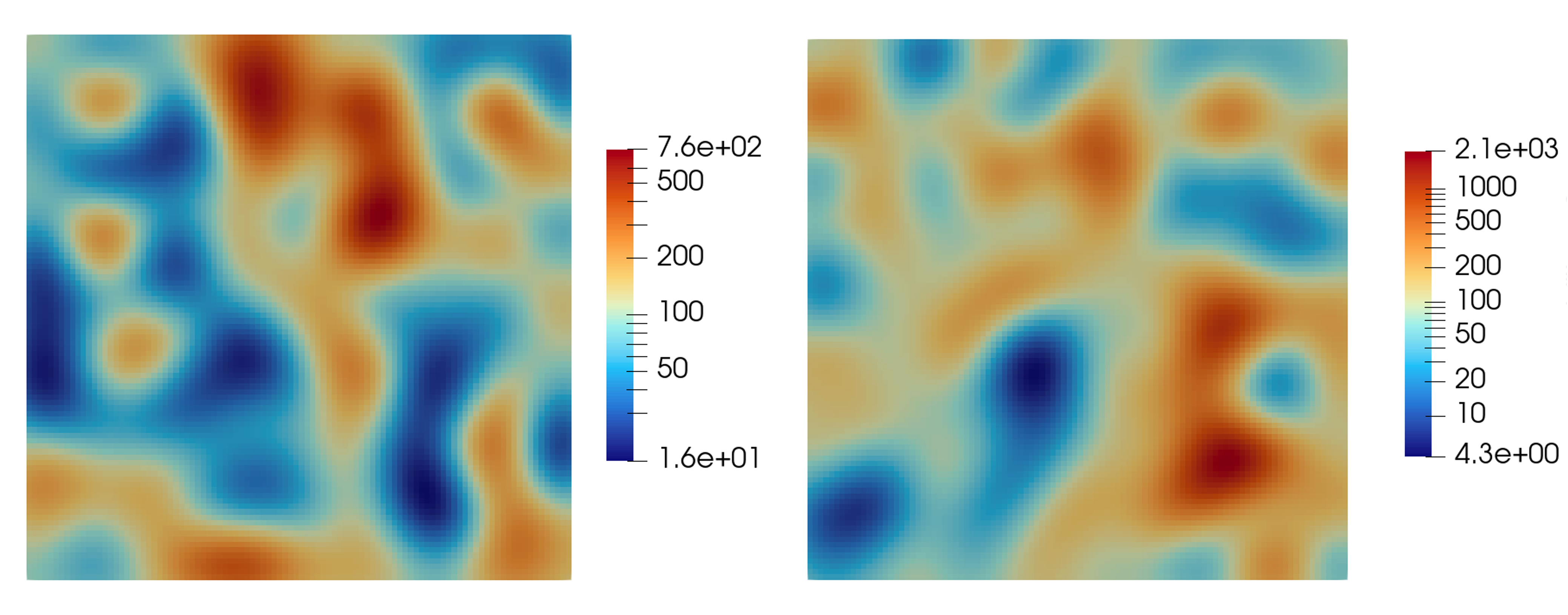}
    \caption{Two random realizations of the permeability field generated using the truncated KL expansion.}
    \label{fig:ex2_perms}
\end{figure}
Correspondingly, in Figures~\ref{fig:ex2_pressures} and ~\ref{fig:ex2_magvel} we plot the pressure fields and the magnitude of the velocity fields for each of these random realizations.
\begin{figure}
    \centering
    \includegraphics[width=0.8\linewidth]{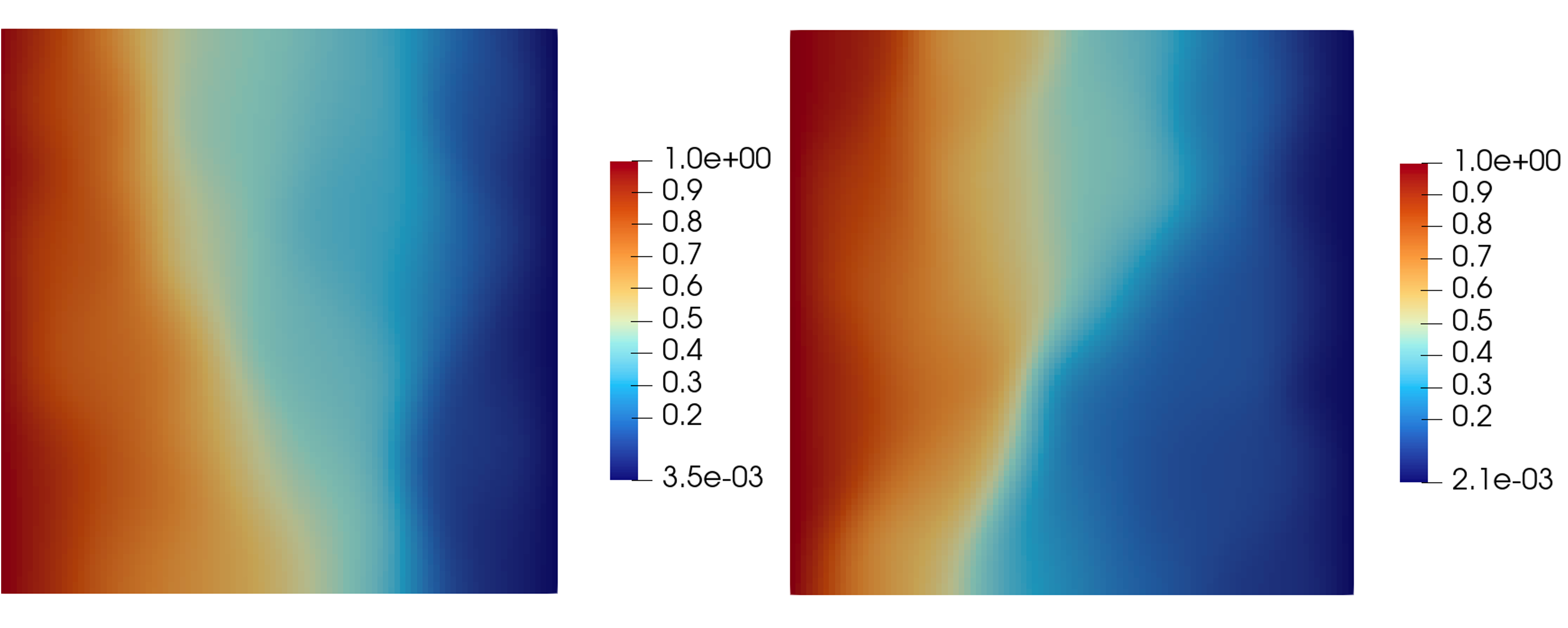}
    \caption{The approximated pressure fields corresponding to the two random realizations of the permeability field generated using the truncated KL expansion.}
    \label{fig:ex2_pressures}
\end{figure}
\begin{figure}
    \centering
    \includegraphics[width=0.8\linewidth]{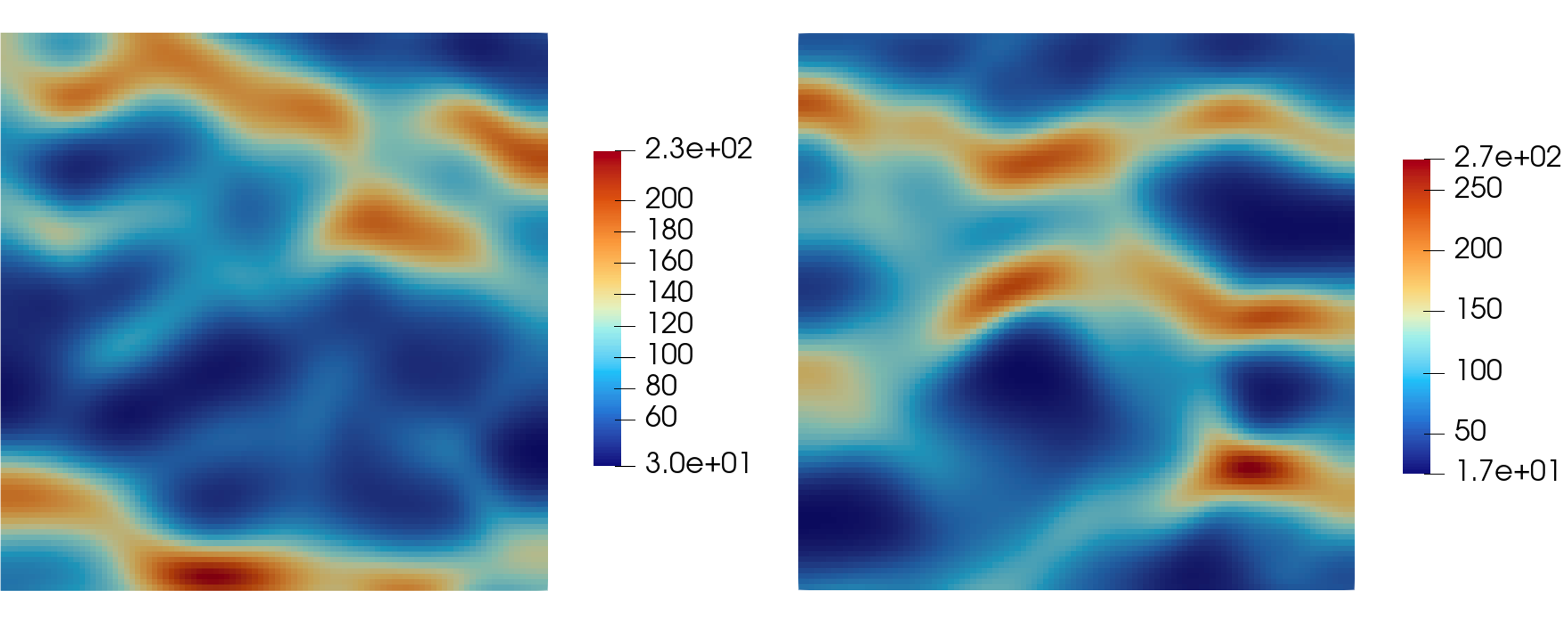}
    \caption{The magnitude of the approximated velocity fields corresponding to the two random realizations of the permeability field generated using the truncated KL expansion.}
    \label{fig:ex2_magvel}
\end{figure}
To define our QoI, we generate 1000 random sensor locations throughout $\Omega$ and measure both the pressure and velocity at these locations.  
Only a small subset of these were actually used for the numerical results in this paper.

We then generated $\expnumber{1}{5}$ samples from both the initial and data-generating distributions, which we define below. 
First, we utilize the assumptions of the Karhunen-Lo\`eve expansion to define the initial distribution as a multivariate standard normal density $\initdens \sim N({\mathbf 0},{\mathbf I})$ where ${\mathbf I}$ is the standard identity matrix.
The data-generating distribution is defined as a multivariate normal density $N({{\lambda}}_{\DG}, {\Sigma}_{\DG})$ where ${\Sigma}_\DG = 0.25 \mathbf{I}$ and the 100-dimensional vector ${\lambda}_\DG$ is given by
\begin{equation*}
    {\lambda}_{\DG,i} = \begin{cases}
        0.3, & i=1, 3, \ldots, 15, \\
        -0.3, & i=2, 4, \ldots, 16, \\
        0, & i>16. 
    \end{cases}
    % \ \text{ and } \ 
    % {\alpha}_{i} = \begin{cases}
    %     0.5, & i\leq 16, \\
    %     1, & i>16. 
    % \end{cases}
\end{equation*}

As discussed in~\cite{BJW18a}, an advantage of the DCI method is that we work directly with functions of the form $r(Q(\lambda))$ and the QoI often define data spaces that are lower-dimensional than the parameter space.
In that work, a single QoI is utilized defining a 1-dimensional data space for which the DCI solution is computed by applying a standard Gaussian KDE to approximate the predicted and observed densities.
However, such KDE methods become impractical as the dimension of the data space increases.
Generally, KDE methods are considered most effective and reliable for low-dimensional data with a common rule of thumb to restrict application of KDE methods to problems of dimension no greater than four or five.
For instance, in the \verb|R| statistical package \verb|ks|, grid-based KDE computation is only performed automatically for dimensions up to four, requiring specific evaluation points for higher dimensions.
When the standard multivariate Gaussian is the target density, a study presented in \cite{NC16} demonstrated that to achieve a KDE with accuracy comparable to 50 samples in one dimension, more than $\expnumber{1}{6}$ samples are required in ten dimensions. 

\begin{figure}
    \centering
    \includegraphics[width=0.5\linewidth]{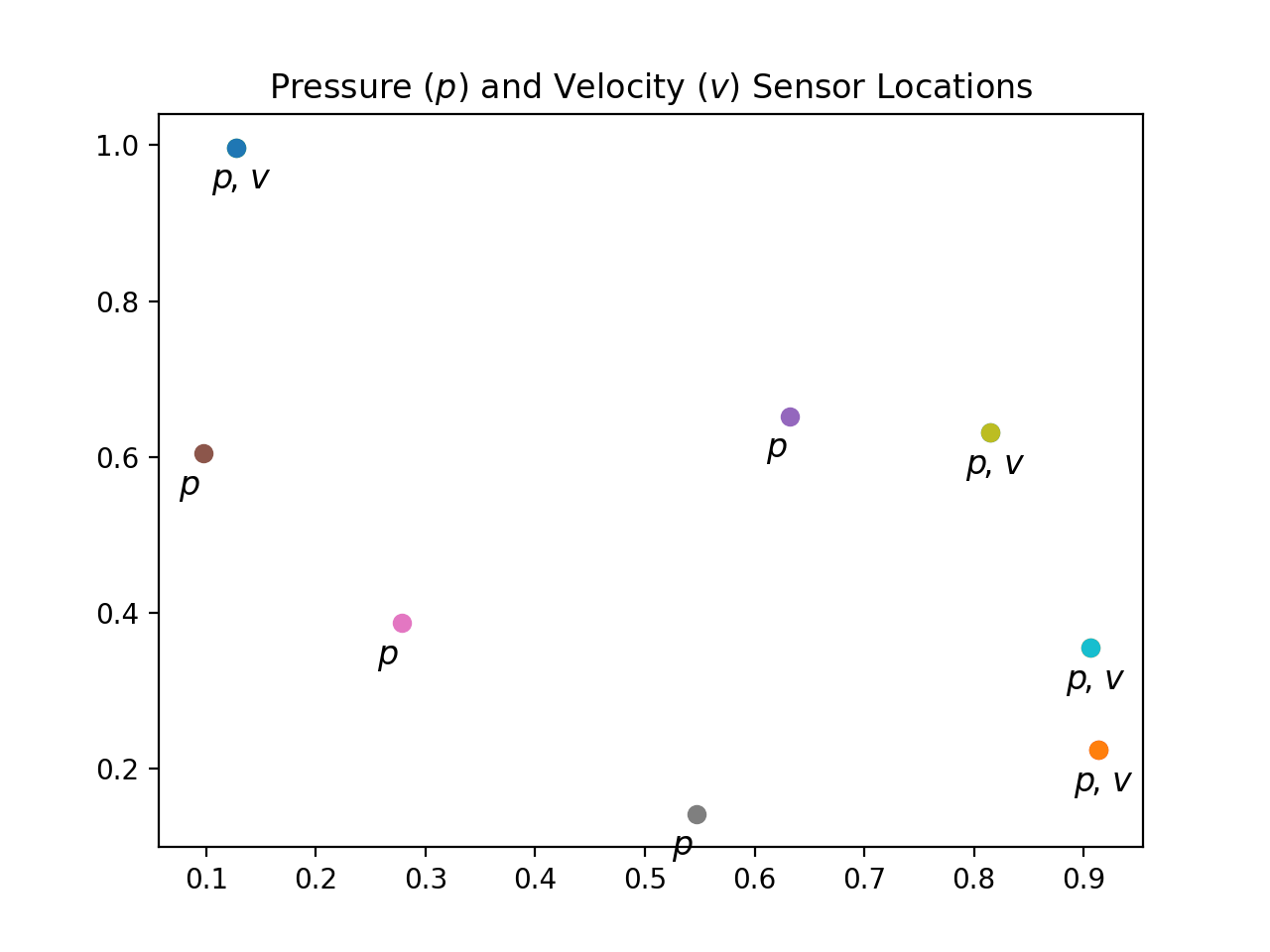}
    \caption{The first eight randomly generated locations for sensors are used to generate a 12-dimensional QoI map. The first four measurement locations produce both pressure and velocity data and the other four measurement locations produce only pressure data. Sensor locations are illustrated by the dots and the type of measurement data produced is shown as either a $p$ for pressure data only or as $p, v$ if both pressure and velocity data are observed.}
    \label{fig:ex2_sensors}
\end{figure}

In the results shown in this work, we demonstrate both the utility and flexibility of the iterative DCI solution by considering two different GSIPs defined by a 12-dimensional QoI map that contains the two types of measurements (pressure and velocity) obtained from randomly placed sensors within the domain.
Specifically, the first four sensors are assumed to measure both the pressure and the norm of the velocity field, which define eight of the QoI, and the last four sensors are assumed to measure only the pressure, which defines the other four QoI.
The locations of these eight sensors are shown in Figure~\ref{fig:ex2_sensors}. 
In the supplementary material, we include the scripts to run MrHyDE and regenerate the full datasets.
The datasets included in the supplementary material allow the user to define various QoI maps of different dimension from 1000 randomly generated sensor locations for which either pressure or velocity (or both) measurements can be chosen.
The eight sensor locations chosen to construct the 12-dimensional QoI map in this example are simply the first eight randomly generated locations of this dataset. 
For the first GSIP considered, we iterate through all 12 components of this map.
For the second GSIP considered, we assume that the pressure and velocity data are obtained simultaneously.
In that case, the four separate 2-dimensional joint observed and predicted densities are estimated and utilized in the iterative algorithm defining a total of eight QoI subspaces that are iterated over in the algorithm. 
By utilizing the iterative DCI approach in both of these GSIPs, we avoid the complications inherent in approximating non-parametric distributions in high-dimensions with KDEs. 
As in the previous example, all densities are approximated with a standard Gaussian KDE with Scott's rule used to determine the bandwidth parameter~\cite{TS92}.
Since there are significantly more push-forward constraints to simultaneously satisfy in this example compared to the previous example, we relax the absolute KL divergence tolerance to $\expnumber{1}{-6}$ in this example while leaving the remaining stopping criteria for the algorithm identical to those utilized in the previous example.
Additionally, since the scales of the pressure and velocity data are significantly different, we apply a standard scaling to all QoI data before any KDEs are applied so that all visualizations of the QoI densities are at similar scales. 

It is worth noting that by approximating so many densities, including several bivariate densities in the second GSIP, we risk a numerical violation of the assumption that $P_{\text{obs},i}=P\circ\phi_i^{-1}$ for each $i$ in Theorem~\ref{thm:main}.
In this case, we do not expect to numerically converge to a single probability measure.
Instead, we expect that the algorithm will locate a ``limit cycle'' of probability measures where, after a sufficient number of epochs, each additional epoch will result in iterations through the same probability measures that satisfy the specific push-forward constraint for an iteration.
In such a case, the algorithm should terminate due to the relative KL divergence stopping criteria where marginal QoI densities obtained at the end of an epoch are not sufficiently updated from those obtained at the conclusion of the prior epoch.
While outside the scope of this work, a future work will focus on the numerical analysis and convergence of the computational algorithm. 
In this work, we limit the presentation to simply pointing out the presence of numerical errors and their impact on results to keep the focus on the application of the iterative approach. 

\subsubsection{Iterative Results}

\begin{figure}
    \centering
    \includegraphics[width=0.7\linewidth]{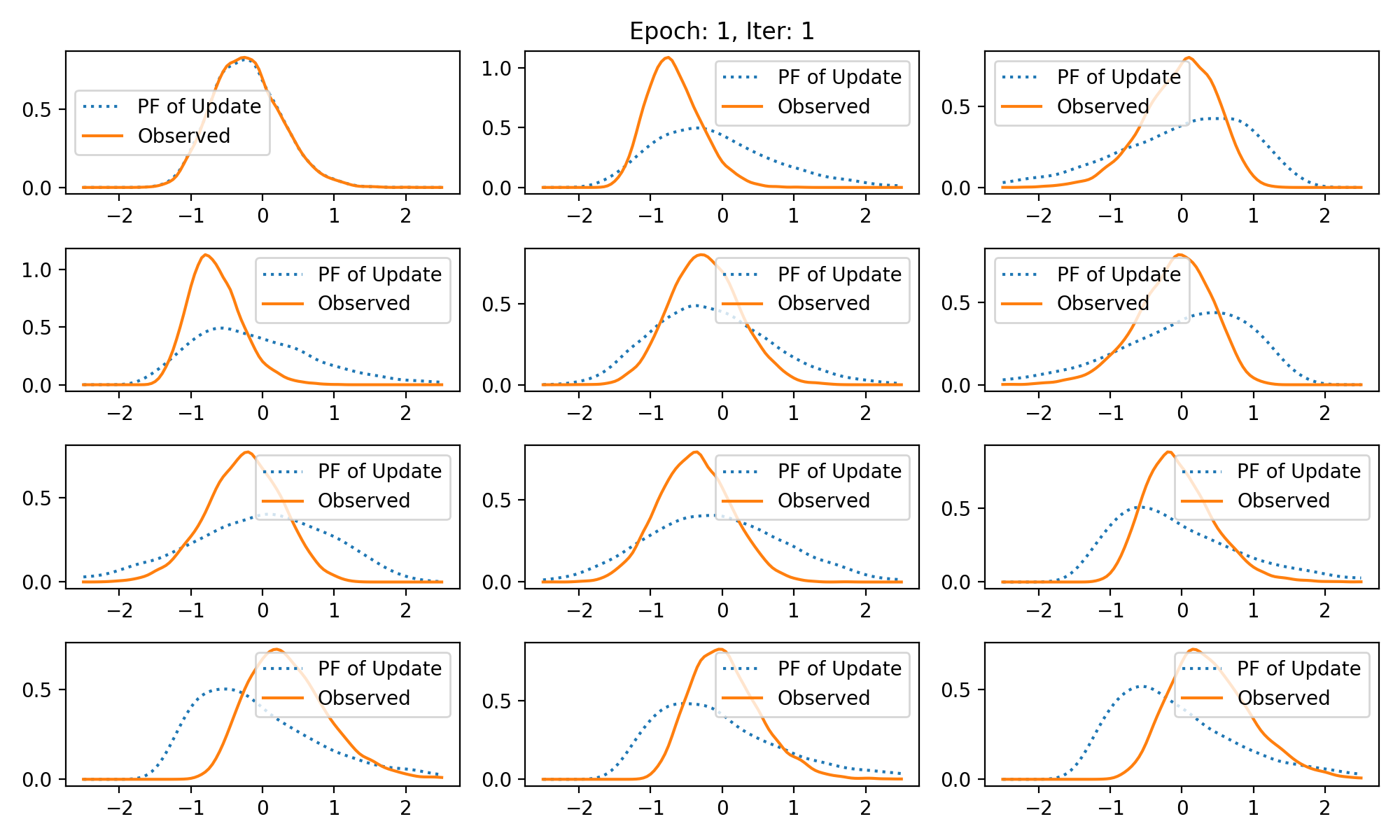}
    \caption{KDE estimates of the marginal QoI densities for $Q_1$ (top left) through $Q_{12}$ (bottom right) ordered from left-to-right in each row starting at the top. 
    The marginals associated with the push-forward measures of the updated distribution for the first iteration of the first epoch are shown as blue dotted curves and for the observed distribution are shown as orange solid curves.
    Except for the two $Q_1$ marginals (top left), the two marginals associated with each of the other QoI all have significant disagreement since none of the associated push-forward constraints have been utilized.}
    \label{fig:ex2_pfs_epoch_1_iter_1}
\end{figure}

\begin{figure}
    \centering
    \includegraphics[width=0.7\linewidth]{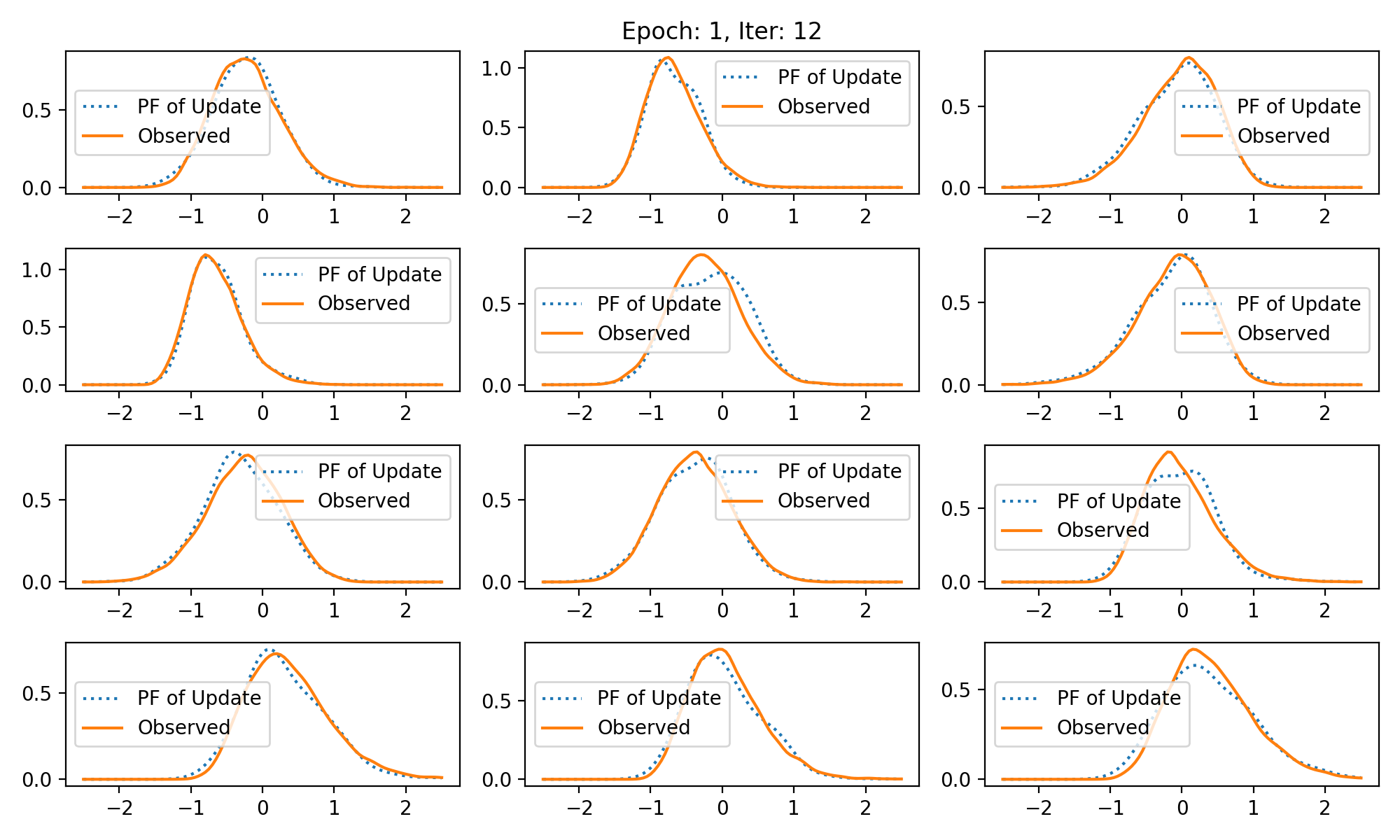}
    \caption{KDE estimates of the marginal QoI densities for $Q_1$ (top left) through $Q_{12}$ (bottom right) ordered from left-to-right in each row starting at the top. 
    The marginals associated with the push-forward measures of the updated distribution for the 12th (and final) iteration of the first epoch are shown as blue dotted curves and for the observed distribution are shown as orange solid curves.
    All push-forward constraints have now been utilized resulting in the two marginals for each of the QoI appearing similar.
    }
    \label{fig:ex2_pfs_epoch_1_iter_12}
\end{figure}

Figures~\ref{fig:ex2_pfs_epoch_1_iter_1} and~\ref{fig:ex2_pfs_epoch_1_iter_12}
summarize the QoI marginals associated with the first and last iterations, respectively, of the first epoch.
In the plots of these figures, the blue dotted curves are the marginals associated with the push-forward of the updated distributions while the orange curves are the marginals associated with the observed distribution.
In Figure~\ref{fig:ex2_pfs_epoch_1_iter_1}, the two marginals associated with $Q_1$ (top left plot) lay almost perfectly on top of each other while significant visual discrepancies are evident between the two marginal densities associated with all of the remaining QoI.
This is expected since the only push-forward constraint utilized at this point is the one associated with $Q_1$. 
%In other words, updating the initial density using only this first component of the QoI map (which corresponds to the pressure measured at approximately $(0.815, 0.631)\in\Omega$) fails to cause the other marginal QoI densities defined by the push-forward of the update and the associated QoI component map. 
However, once all of the push-forward constraints are utilized by iteration \#12, the two marginals for all of the QoI appear qualitatively similar as seen in Figure~\ref{fig:ex2_pfs_epoch_1_iter_12}.
Despite the qualitative similarities, there are some clear numerical discrepancies between several of these marginal densities. 
For instance, the marginals associated with $Q_5$ (shown in the middle plot of the second row) and with $Q_9$ (shown in the right plot of the third row) have the most visually evident discrepancies.
Further epochs are clearly required before numerical convergence, as measured by the absolute KL tolerance of $\expnumber{1}{-6}$, is reached.

\begin{figure}
    \centering
    \includegraphics[width=0.32\linewidth]{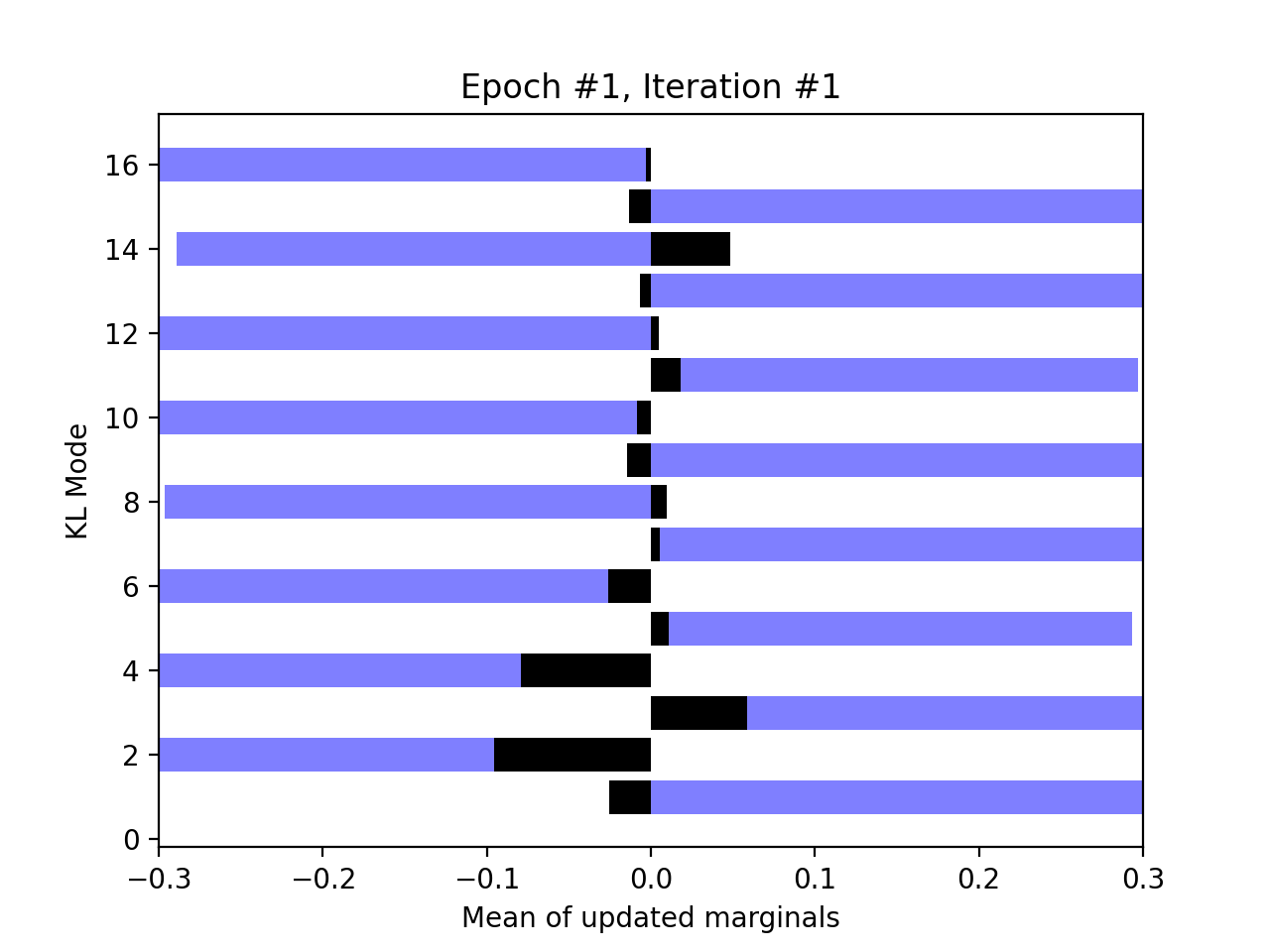}
    \includegraphics[width=0.32\linewidth]{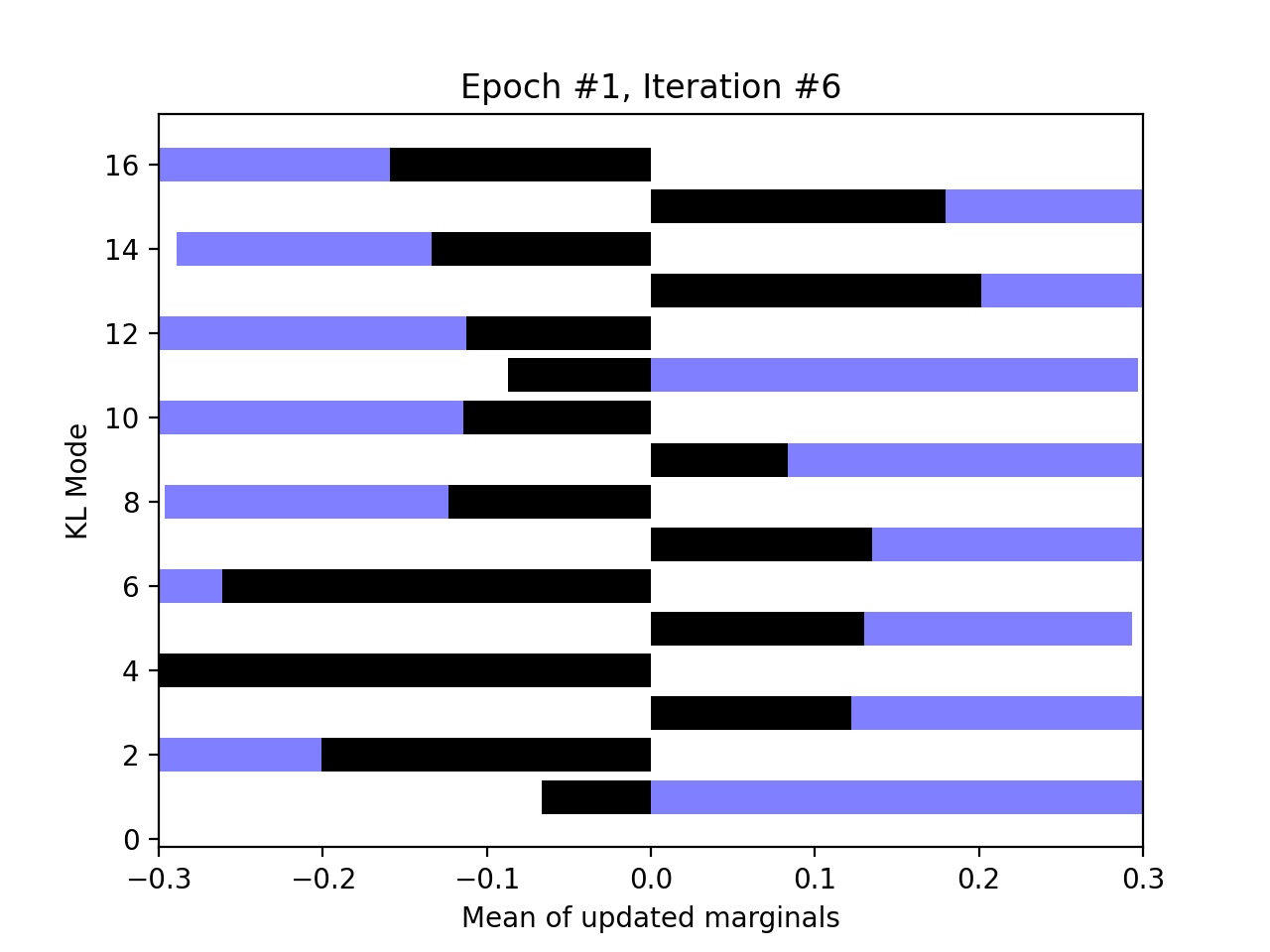}
    \includegraphics[width=0.32\linewidth]{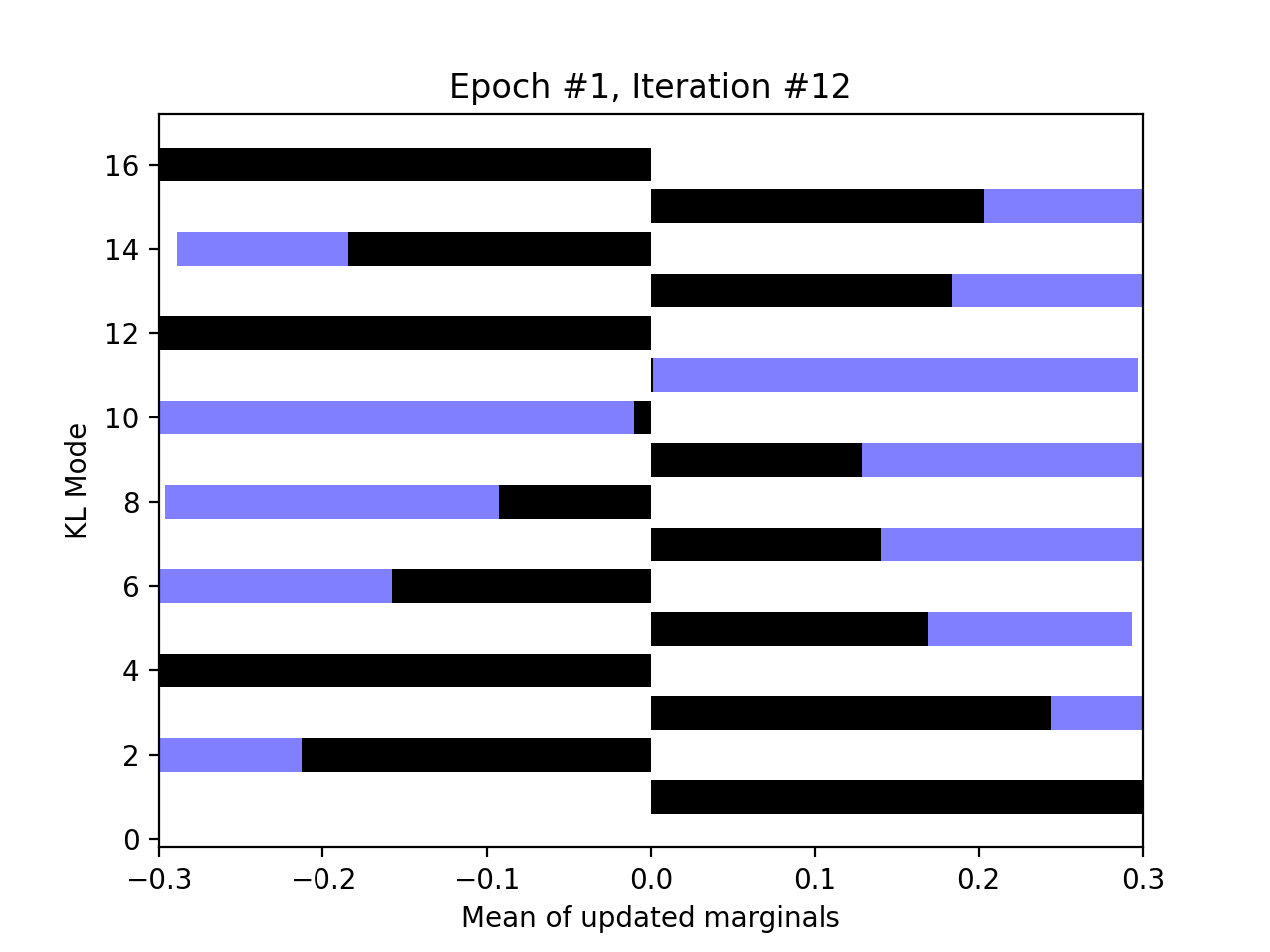}
    \caption{Sample means of the first 16 KL modes for the data-generating distribution (blue bars) and updated distributions (black bars). These plots show the variation in these results across the first epoch with the first iteration (left plot), sixth iteration (middle plot), and 12th and final iteration (right plot).}
    \label{fig:ex2_means_16KL_epoch_1_iters_1_6_and_12}
\end{figure}

Figure~\ref{fig:ex2_means_16KL_epoch_1_iters_1_6_and_12} shows three plots that provide insight into the behavior of the iterative updates on the parameter space obtained throughout the first epoch.
In each of these plots, we show the mean of the updated marginals for the first 16 (out of 100) KL modes.
We focus the plots on just those KL modes because those are the ones for which the sample means of the data-generating distribution were chosen to deviate away from zero, which allows some fixed basis for comparison. 
As previously mentioned, the data-generating means were chosen to alternate between $\pm 0.3$ for the first 16 KL modes, and the solid blue bars in each plot show the sample mean values, which approximate these theoretical values.
The black bars in each plot show the means of the KL modes associated with the updated distribution obtained at iteration \#1 (left plot), \#6 (middle plot), and \#12 (right plot). 
Recall that the initial means are all zero, and the first iteration (left plot) results in an updated distribution that produces means for only a few of these KL modes that appear to deviate strongly from zero and in the directions of the data-generating means (e.g., KL modes 2, 3, and 4) while most of the others either do not appear to deviate strongly away from zero or have values that appear to deviate in a direction opposite to that of the data-generating mean (e.g., KL mode \#14). 
By iteration \#6 (middle plot), there are clear deviations away from zero for more of the updated means of KL modes.
Furthermore, the emerging pattern mostly appears to mimic the alternating structure seen in the data-generating distribution with only a few exceptions (e.g., KL modes \#1 and \#11). 
At iteration \#12 (right plot), which is the final iteration of the epoch, the alternating pattern appears to be mostly present with only KL modes \#10 and \#11 not appearing to be strongly biased in the directions of the data-generating means. 

Before we show the results where numerical convergence was achieved, a few remarks are in order.
First, there are significant variations observed in the iterative updates throughout this epoch as evidenced by the varying mean values of the first 16 KL modes shown in Figure~\ref{fig:ex2_means_16KL_epoch_1_iters_1_6_and_12}.
Second, we do not expect to recover the data-generating distribution here since we (1) are not using a QoI map that defines a bijection, (2) are iterating through components of the QoI map, and (3) related to both (1) and (2), we expect an infinite number of probability distributions on the 100-dimensional parameter space that satisfy the 12 push-forward constraints. 
However, we found that most QoI maps chosen that varied from as few as eight to as many as 20 components resulted in final updated distributions that produced means of these KL modes that mostly matched the alternating pattern of the data-generating means for these first 16 KL modes. 
We expect this is due to the PDE solution being most sensitive to perturbations in the lower KL modes than the higher KL modes, which results in most of the potential measurements exhibiting similar sensitivities. 
While outside the scope of this work, a future work will examine systematic greedy approaches to optimal experimental design (OED) for constructing QoI maps, component-by-component, for the iterative solution of high-dimensional DCI problems.

\begin{figure}
    \centering
    \includegraphics[width=0.7\linewidth]{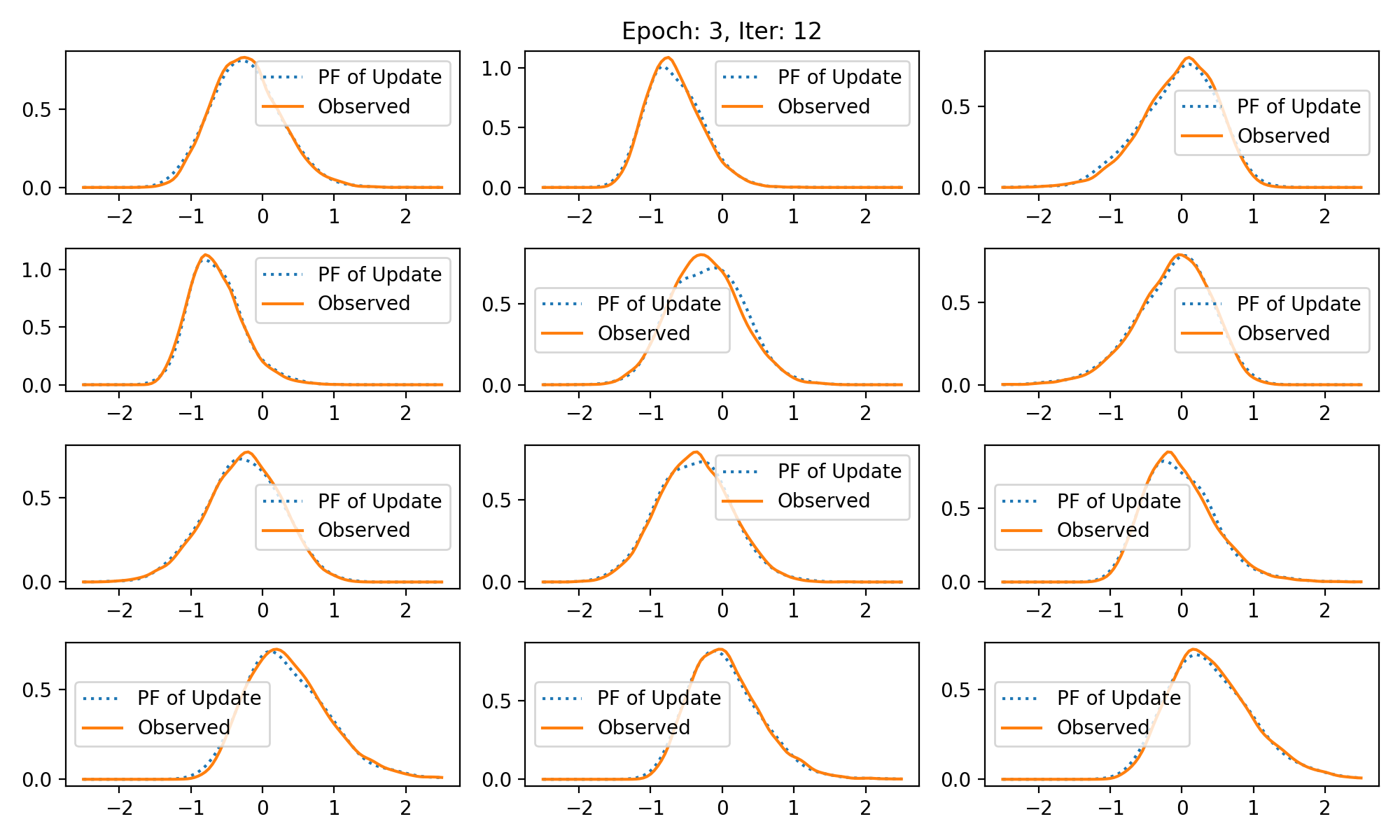}
    \caption{KDE estimates of the marginal QoI densities for $Q_1$ (top left) through $Q_{12}$ (bottom right) ordered from left-to-right in each row starting at the top. 
    The marginals associated with the push-forward measures of the updated distribution for the 12th (and final) iteration of the third (and final) epoch are shown as blue dotted curves and for the observed distribution are shown as orange solid curves.
    All pairs of QoI marginals are within the relaxed KL divergence tolerance of $\expnumber{1}{-6}$.}
    \label{fig:ex2_pfs_epoch_3_iter_12}
\end{figure}

The algorithm terminated at the third epoch due to the absolute KL divergence stopping criteria.
In other words, the KL divergence of every QoI marginal associated with the push-forward of the updated distribution at the end of this epoch was within $\expnumber{1}{-6}$ of the associated observed marginal. 
This is visually confirmed by the array of QoI marginal plots shown in Figure~\ref{fig:ex2_pfs_epoch_3_iter_12} that is arranged in the same order as Figures~\ref{fig:ex2_pfs_epoch_1_iter_1} and~\ref{fig:ex2_pfs_epoch_1_iter_12}.

\begin{figure}
    \centering
    \includegraphics[width=0.49\linewidth]{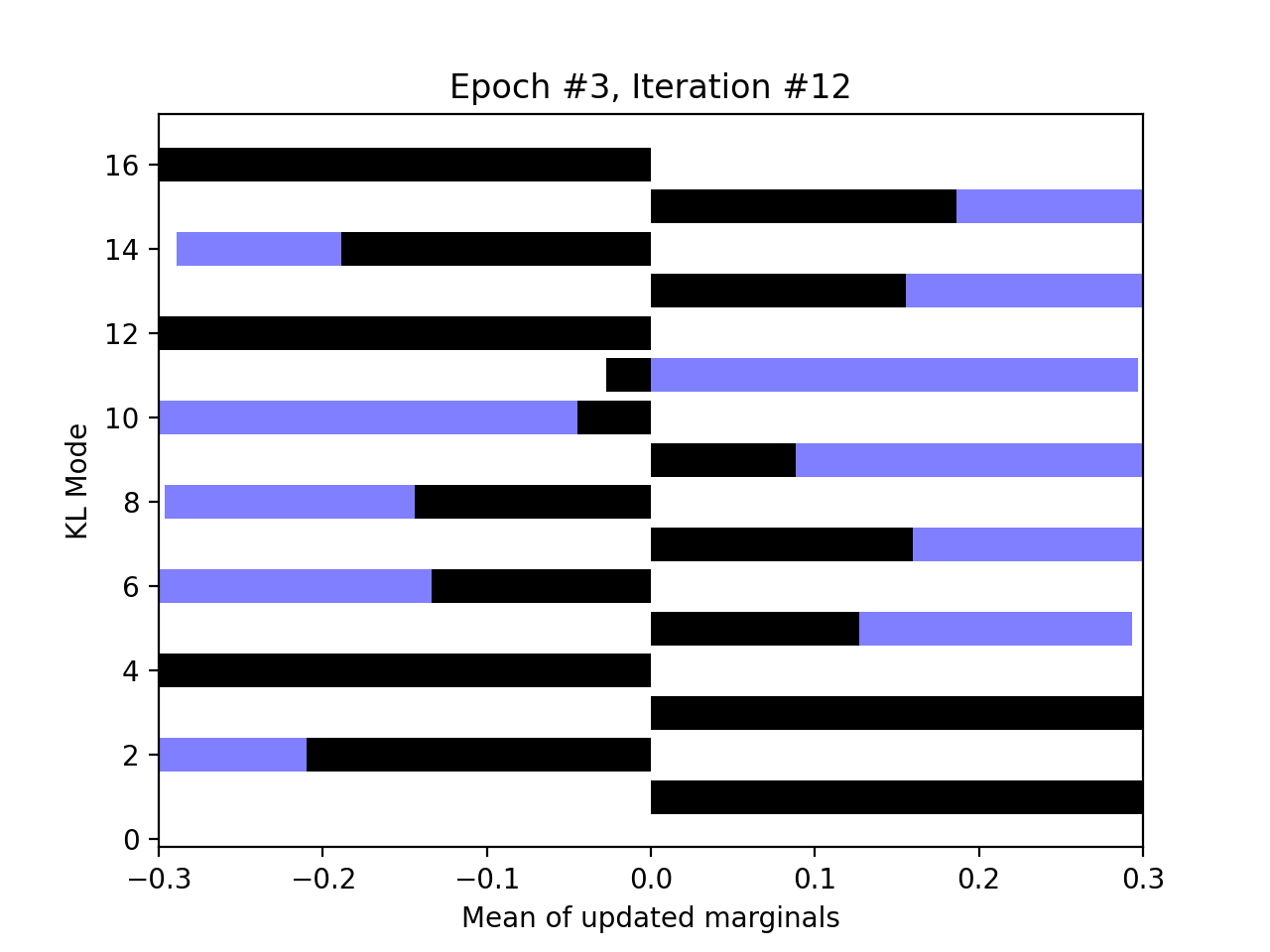}
    \includegraphics[width=0.49\linewidth]{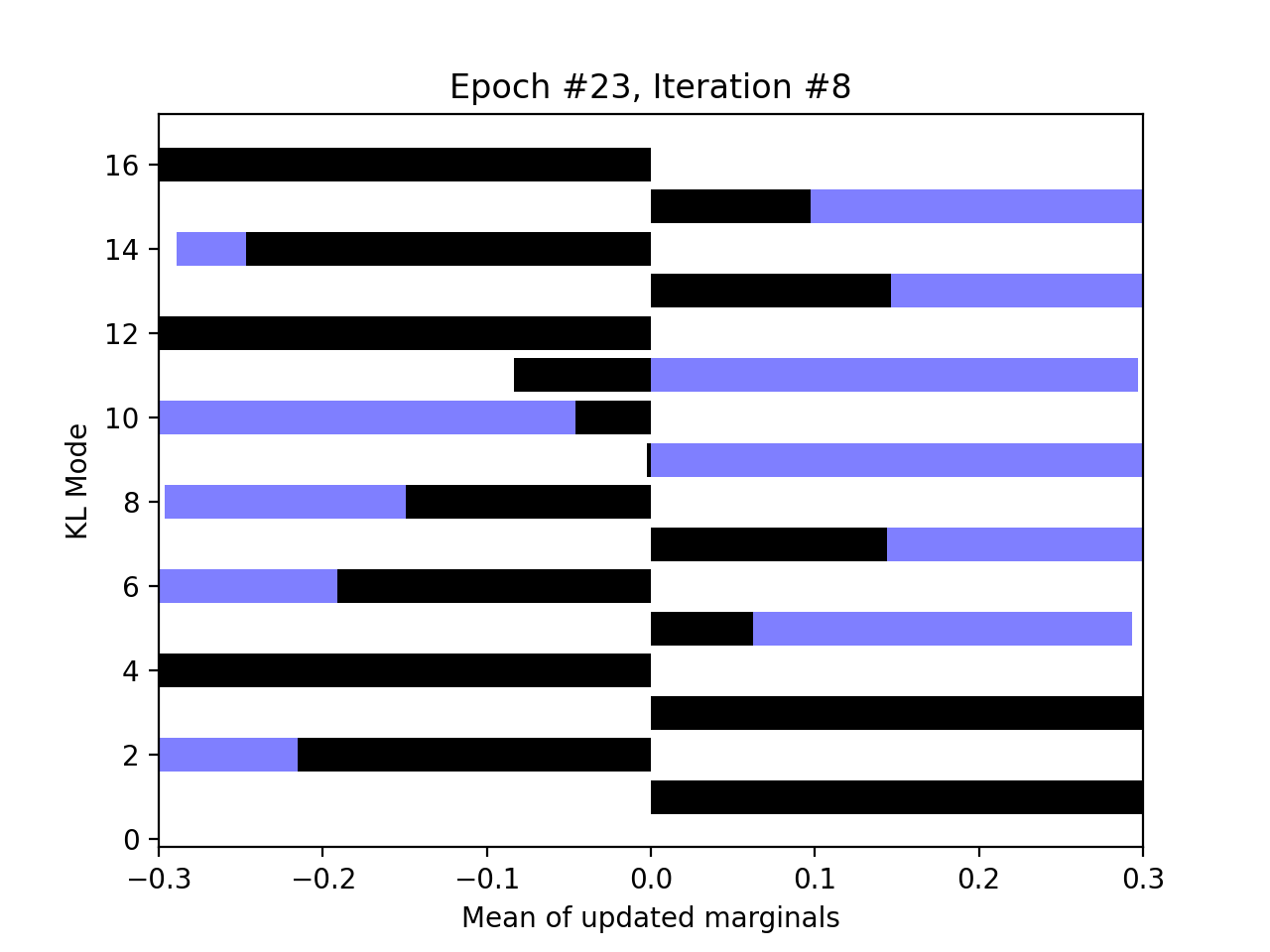}
    \caption{Sample means of the first 16 KL modes for the data-generating distribution (blue bars) and updated distributions (black bars). These plots show the variation in final epoch results when iterating across all 12 QoI (left plot) or across eight QoI subspaces (right plot) defined by four two-dimensional (pressure and velocity) and four one-dimensional (pressure) measurements as illustrated in Figure~\ref{fig:ex2_sensors}.}
    \label{fig:ex2_means_16KL_comparison}
\end{figure}

In the left plot of Figure~\ref{fig:ex2_means_16KL_comparison}, we show the means of the first 16 KL modes for the final updated distribution achieved at the final iteration of epoch \#3 versus the sample means of the data-generating distribution. 
As with the final iteration of the first epoch, the general alternating pattern is again seen with a clear exception for KL mode \#11. 
We again emphasize that the updated distribution is defined on a 100-dimensional space that is not guaranteed to reconstruct the data-generating distribution as previously mentioned, but, as seen in Figure~\ref{fig:ex2_pfs_epoch_3_iter_12}, the 12 push-forward constraints are simultaneously satisfied. 
Unlike the iterated updated distributions of the first epoch, the variation observed across the iterations of this final epoch was minimal, so those results are omitted in the interest of space.
However, the interested reader may use the supplementary material to explore this further and also vary the stopping criteria.

% we show two plots that provide insight into the means of the KL modes computed from the updates within this final epoch (left plot) as well as how the means of these updated KL modes vary as a function of iteration number where we note that there are $27\times 12 = 324$ total iterations.
% In the left plot, we observed only minor variations in the values obtained at each iteration, suggesting that each successive update is not significantly different from the prior ones in this epoch.
% This is further confirmed in the right plot where we observe most trend lines for the KL modes are relatively flat with relatively small oscillations.
% While all push-forward constraints are satisfied (within the specified tolerance), these small oscillations are indicative of the numerical approximation issues previously mentioned that can lead to the updated distributions iterating through a ``limit cycle'' of probability distributions. 
% The interested reader is encouraged to use the supplementary material to verify that reducing the absolute stopping criteria for the KL divergence from $\expnumber{1}{-7}$ to $\expnumber{1}{-10}$ will result in the relative stopping criteria terminating the algorithm at epoch 30 while producing results that are virtually indistinguishable from those seen here. 

\subsubsection{Iterative Result with Joint Pressure \& Velocity Data}

We conclude this numerical example by demonstrating the flexibility of the algorithm where we now assume that the four sensors that produce both pressure and velocity data do so simultaneously so that the joint data at these locations can be utilized in the iterations.
In other words, each epoch now contains eight iterations where four of the iterations involve the use of a two-dimensional joint QoI space. 

\begin{figure}
    \centering
    \includegraphics[width=0.49\linewidth]{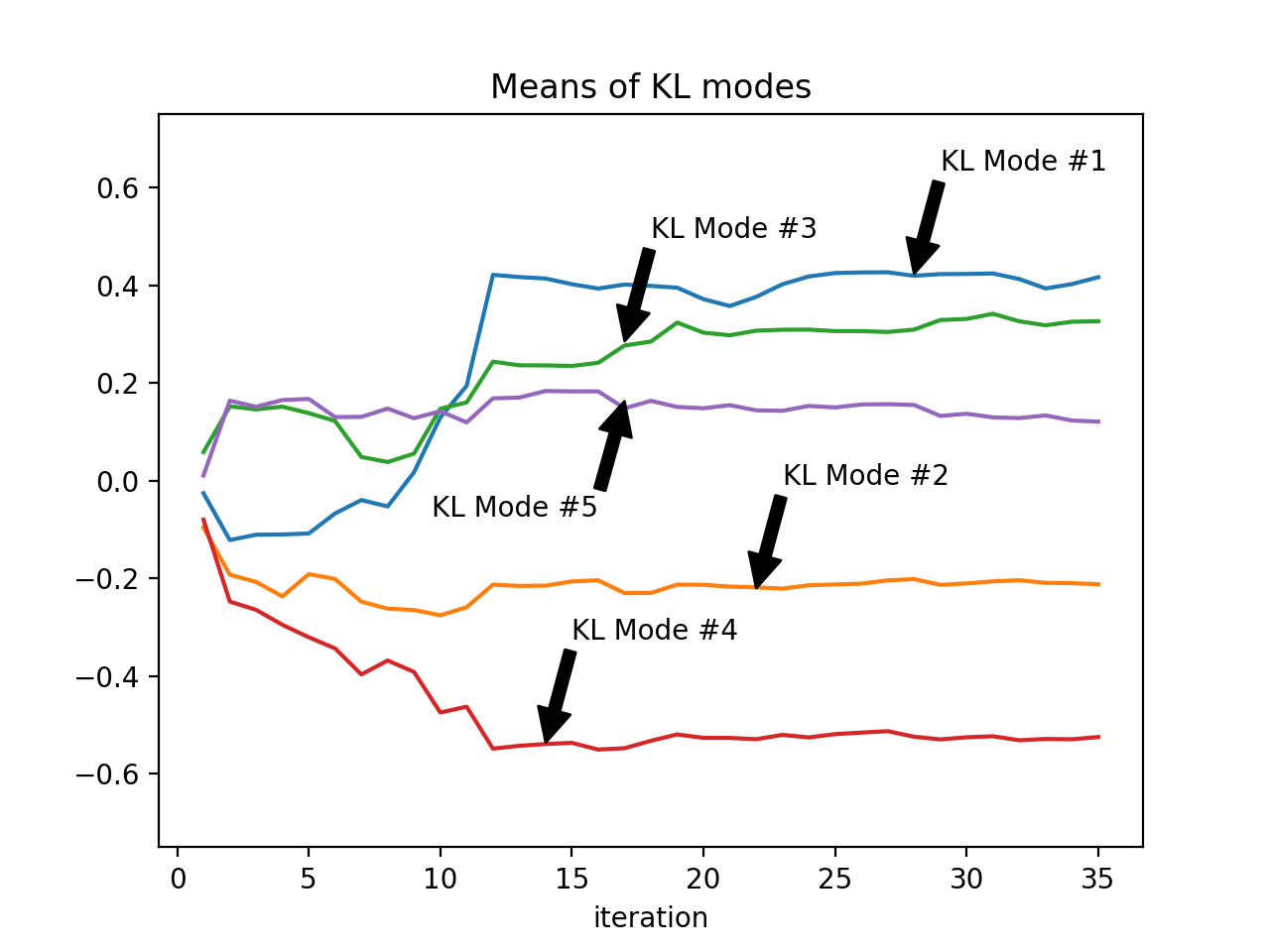}
    \includegraphics[width=0.49\linewidth]{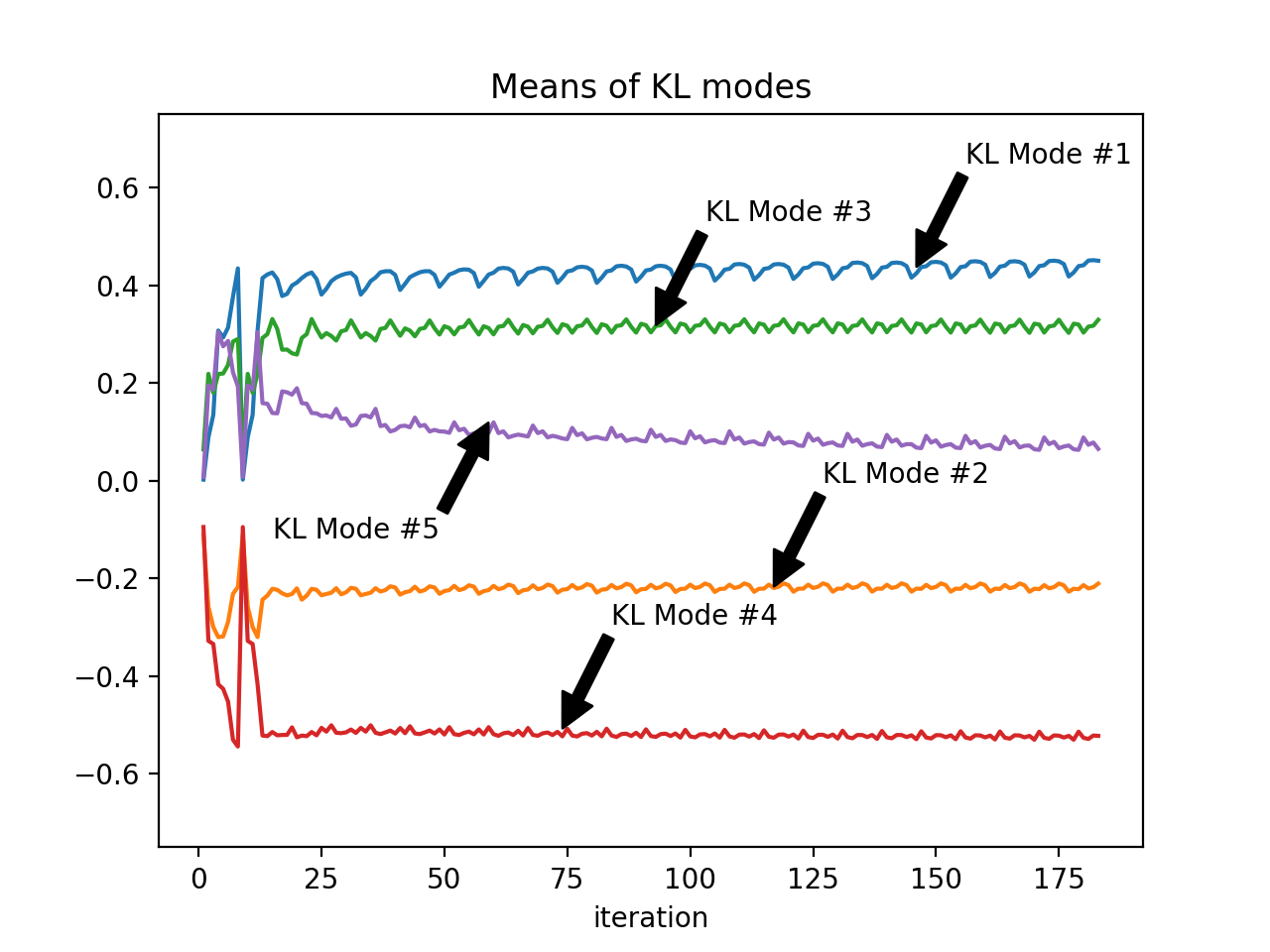}
    \caption{The trends of the updated means of the first five KL modes as a function of iteration when iterating across all 12 QoI (left plot) or across eight QoI subspaces (right plot) defined by four two-dimensional (pressure and velocity) and four one-dimensional (pressure) measurements as illustrated in Figure~\ref{fig:ex2_sensors}.}
    \label{fig:ex2_means_5K_iteration_comparison}
\end{figure}

The algorithm now terminates at epoch \#23 due to the relative KL divergence stopping criteria.
In other words, the QoI marginals associated with the push-forward of the updated distribution at the end of epoch \#23 is not significantly different than that obtained at the end of epoch \#22.
The right plot of Figure~\ref{fig:ex2_means_16KL_comparison} summarizes the final updated means of the first 16 KL modes compared to the sample means of the data-generating distribution.
Comparing these to those in the left plot of this same figure, we see many similarities with a few noticeable differences in the 6 of the 16 means of the KL modes \#5, \#6, \#9, \#11, \#14, and \#15. 
This illustrates the impact of incorporating joint structure in some of the QoI data not present in the prior GSIP. 

We further compare the results of these two GSIPs in Figure~\ref{fig:ex2_means_5K_iteration_comparison} where the updated means of the first five KL modes are shown as a function of iteration.
The left plot shows the results for the prior GSIP where there are a total of $3\times 12 = 36$ iterations and the right plot shows the results where there are eight QoI subspaces (four two-dimensional and four one-dimensional) and a total of $23\times 8=184$ iterations.
We first observe that the right plot shows there are more significant transitory variations in the means of the KL modes in the initial iterations of the current GSIP compared to the prior GSIP. 
After about two epochs, both GSIPs produce updated distributions with relatively stable means of KL modes. 
The left plot specifically shows minor variation in the updated means after iteration 24 (two epochs) whereas the right plot shows curves that seem to converge to a relatively horizontal but oscillatory signal indicating a ``limit cycle'' of the updated distributions is occurring within the epochs. 
This limit cycle behavior is in fact expected due to the fact that the algorithm terminated due to the the relative KL criteria. 
The interested reader can make the oscillations larger by reducing the number of samples used to approximate the densities, which increases the error in the density approximations and generally makes it more difficult for all push-forward constraints to be simultaneously satisfied. 

% \begin{figure}
%     \centering
%     \includegraphics[width=0.49\linewidth]{figures/ex2_means_16KL_modes_last_epoch_joint_pressure_velocity.png}
%     \includegraphics[width=0.49\linewidth]{figures/ex2_means_16KL_modes_vs_iteration_joint_pressure_velocity.png}
%     \caption{These results are for eight QoI spaces defining the 12 total QoI where pressure and velocity data spaces are jointly inverted for each sensor location that provides both types of data as shown in Figure~\ref{fig:ex2_sensors}.
%     (Left) A zoomed in look at the first 16 KL modes for which the means of the data-generating distribution have alternating  values (sample means shown in solid blue). The final iteration of the final epoch is black while the earlier iterations are in pinkish/red colors that are also more transparent. Note the every iteration in this final epoch exhibits the same alternating pattern for the mean of each KL mode that is present in the data-generating distribution.
%     (Right) The means of the first 16 KL modes as a function of the $17\times 12 = 204$ iterations, with final values shown in the plot to the left.}
%     \label{fig:ex2_joint_pressure_velocity}
% \end{figure}

\section{Conclusions and Future Work}\label{sec:conclusions}

The challenge of solving the GSIP, where a parameter measure must satisfy multiple push-forward constraints simultaneously, served as the primary motivation for this work. 
We introduced an iterative DCI framework designed to solve the GSIP through a sequence of local optimization steps. 
This approach leveraged the rigorous measure-theoretic foundations of DCI by systematically allowing for the sequential incorporation of information from disparate data sources associated with various subsets of the QoI.

The theoretical analysis established that the fundamental DCI solution for the standard SIP is an optimal information updated distribution that minimizes the $f$-divergence over the space of all measure-consistent pullback solutions. 
This crucial result enabled the construction of a convergent iterative algorithm for the GSIP, which we proved converges in total variation to a unique information projection of the initial distribution onto the intersection of the solution sets. 
Numerical results across low-dimensional linear systems and high-dimensional PDE-based examples demonstrated the convergence and practical robustness of the iterative DCI method, particularly highlighting its advantage in scenarios where the explicit approximation of high-dimensional joint observational densities is computationally infeasible or not possible due to asynchronous observations of QoI data.

The iterative DCI methodology provides a powerful, modular tool for updating uncertainty quantifications as new datasets become available. 
Future work will focus on analyzing the following topics rooted in this approach:

\begin{itemize}
    \item \textbf{Optimization of Initial Distribution:} The final converged solution ($P^\infty$) is dependent upon the choice of the initial measure ($P^0 = P_{\text{init}}$). 
    Future work will explore strategies for optimizing the choice of $P^0$ to target specific features of the data-generating distribution. 
    This includes analyzing the convergence properties when using a sequence of varying initial distributions that are mutually absolutely continuous after the first epoch.
    This also includes analyzing what conditions on $\initmeas$ are necessary to guarantee $\initmeas\ll\updatemeas$ to conclude optimality of the backward minimization problem. 
    \item \textbf{Numerical Analysis and Computational Efficiency:} While we demonstrated numerical convergence, a dedicated focus on the quantitative convergence analysis of the computational algorithm is needed, particularly addressing the effect of finite sampling errors and approximation errors in density estimation that lead to potential ``limit cycle'' behavior in the iterative sequence.
    \item \textbf{Integration with Learned QoI and Optimal Experimental Design (OED):} The iterative framework is well-suited for dynamically incorporating QoI maps learned from high-dimensional or noisy datasets, such as those derived through machine learning methods (e.g., kPCA in the LUQ framework \cite{MSB+22, RHB25}). 
    Future research will integrate the iterative DCI solution with OED criteria designed for DCI to construct optimal sequential experimental design strategies that systematically select the most informative next constraint ($\phi_i$) to incorporate into the GSIP. 
\end{itemize}

%%%%%%%%%%%%%%%%%%%%%%%%%%%%%%%%%%%%%%%%%%%%%%%%%%%
\section{Data Availability and Supplementary Material}\label{sec:supplementary}
%%%%%%%%%%%%%%%%%%%%%%%%%%%%%%%%%%%%%%%%%%%%%%%%%%%
%{\bf To-Do: Create a github repo and a zenodo link to the notebooks and data sources utilized in this work.}

The public repo \verb|https://github.com/CU-Denver-UQ/Iterative-DCI| contains the code utilized to generate the results for the two numerical examples.
Executing either the Jupyter notebook or the Python script associated with the first example will both generate all the data as well as the figures shown in Section~\ref{sec:example1}. 
The data for the second example shown in Section~\ref{sec:example2} must be generated prior to executing either the associated Jupyter notebook or script. 
To generate this data, install MrHyDE~\cite{mrhyde_user,mrhyde2023github} and use the yaml files as well as the dat file provided in the GitHub repo.

%%%%%%%%%%%%%%%%%%%%%%%%%%%%%%%%%%%%%%%%%%%%%%%%%%%
\section{Acknowledgments}\label{sec:acknowledgments}
%%%%%%%%%%%%%%%%%%%%%%%%%%%%%%%%%%%%%%%%%%%%%%%%%%%
T.~Butler's work is supported by the National Science Foundation under Grant No.~DMS-2208460.
T.~Butler's work was also supported by NSF IR/D program while working at National Science Foundation. 
However, any opinion, finding, and conclusions or recommendations expressed in this material are those of the author and do not necessarily reflect the views of the National Science Foundation.

This paper describes objective technical results and analysis. Any subjective views or opinions that might be expressed in the paper do not necessarily represent the views of the U.S. Department of Energy or the United States Government.
Sandia National Laboratories is a multimission laboratory managed and operated by National Technology and
Engineering Solutions of Sandia, LLC., a wholly owned subsidiary of Honeywell International, Inc.,
for the U.S. Department of Energy’s National Nuclear Security Administration under contract DE-NA-0003525.  This material is based upon work supported by the U.S. Department of Energy, Office of Science, Office of Advanced Scientific Computing Research, under contract 25-025287.

\bibliographystyle{acm}
\bibliography{references} 

@misc{ShiEstep2026,
  author = {Shi, Haiyi and Estep, Don},
  title = {An Iterative Procedure for Statistical Calibration},
  year = {2026},
  eprint = {2611.33561},
  archivePrefix = {arXiv},
  primaryClass = {stat.ME},
  url = {https://arxiv.org/abs/2611.33561}
}

@article{BSW+25,
	author  = {T. Butler and R. Spence and T. Wildey and T.Y. Yen},
	title   = {Stability and Convergence of Solutions to Stochastic Inverse Problems Using Approximate Probability Densities},
    journal = {International Journal for Uncertainty Quantification},
	issn    = {2152-5080},
	year    = {2025},
	volume  = {15},
	number  = {4},
	pages   = {21--51}
}

@article{pyMC3,
author={Salvatier, J. and Wiecki, T.V. and Fonnesbeck, C.},
year={2016},
title={{Probabilistic programming in Python using PyMC3}},
journal={PeerJ Computer Science},
volume={2},
pages={e55},
isbn={2376-5992}
}

@Article{Wikle1998,
author="Wikle, C.K.
and Berliner, L.M.
and Cressie, N.",
title={{Hierarchical {B}ayesian space-time models}},
journal="Environmental and Ecological Statistics",
year="1998",
month="Jun",
day="01",
volume="5",
number="2",
pages="117--154",
issn="1573-3009",
doi="10.1023/A:1009662704779",
url="https://doi.org/10.1023/A:1009662704779"
}

@article{ganis2008stochastic,
title={Stochastic collocation and mixed finite elements for flow in porous media},
author={Ganis, B. and Klie, H. and Wheeler, M.F. and Wildey, T. and Yotov, I. and Zhang, D.},
journal={Computer methods in applied mechanics and engineering},
volume={197},
number={43},
pages={3547--3559},
year={2008},
publisher={Elsevier}
}

@article{wheeler2011multiscale,
title={A multiscale preconditioner for stochastic mortar mixed finite elements},
author={Wheeler, M.F. and Wildey, T. and Yotov, I.},
journal={Computer Methods in Applied Mechanics and Engineering},
volume={200},
number={9},
pages={1251--1262},
year={2011},
publisher={Elsevier}
}

@article{zhang2004efficient,
title={An efficient, high-order perturbation approach for flow in random porous media via {K}arhunen-{L}o\'eve and polynomial expansions},
author={Zhang, D. and Lu, Z.},
journal={Journal of Computational Physics},
volume={194},
number={2},
pages={773--794},
year={2004},
publisher={Elsevier}
}

@article{NC16,
title = {Evading the curse of dimensionality in nonparametric density estimation with simplified vine copulas},
journal = {Journal of Multivariate Analysis},
volume = {151},
pages = {69--89},
year = {2016},
issn = {0047-259X},
doi = {https://doi.org/10.1016/j.jmva.2016.07.003},
url = {https://www.sciencedirect.com/science/article/pii/S0047259X16300471},
author = {T.~Nagler and C.~Czado}
}

@article{TS92,
 ISSN = {00905364},
 URL = {http://www.jstor.org/stable/2242011},
 author = {G.R.~Terrell and D.W.~Scott},
 journal = {The Annals of Statistics},
 number = {3},
 pages = {1236--1265},
 publisher = {Institute of Mathematical Statistics},
 title = {Variable Kernel Density Estimation},
 volume = {20},
 year = {1992},
 doi = {10.1214/aos/1176348768},
}

@article{AH21,
  author  = {Agrawal, R. and Horel, T.},
  title   = {Optimal bounds between $f$-divergences and integral probability metrics},
  journal = {J. Mach. Learn. Res.},
  volume  = {22},
  number  = {1},
  pages   = {1--59},
  year    = {2021}
}

@ARTICLE{Guntuboyina2011,
  author={Guntuboyina, A.},
  journal={IEEE Transactions on Information Theory}, 
  title={Lower Bounds for the Minimax Risk Using $f$-Divergences, and Applications}, 
  year={2011},
  volume={57},
  number={4},
  pages={2386-2399},
  doi={10.1109/TIT.2011.2110791}
}

@ARTICLE{BKM17,
  title     = "Variational Inference: A Review for Statisticians",
  author    = "Blei, D.~M. and Kucukelbir, A. and McAuliffe, J.D.",
  journal = "Journal of the American Statistical Association",
  publisher = "Taylor \& Francis",
  volume    =  112,
  number    =  518,
  pages     = "859--877",
  month     =  apr,
  year      =  2017
}

@article{BWY20,
  title={Data-consistent inversion for stochastic input-to-output maps},
  author={Butler, T. and Wildey, T. and Yen, T.Y.},
  journal={Inverse Problems},
  volume={36},
  number={8},
  pages={085015},
  year={2020},
  publisher={IOP Publishing}
}

@article{BH20,
  title = {What Do We Hear from a Drum? {{A}} Data-Consistent Approach to Quantifying Irreducible Uncertainty on Model Inputs by Extracting Information from Correlated Model Output Data},
  shorttitle = {What Do We Hear from a Drum?},
  author = {Butler, T. and Hakula, H.},
  year = {2020},
  month = oct,
  journal = {Computer Methods in Applied Mechanics and Engineering},
  volume = {370},
  pages = {113228},
  issn = {0045-7825},
  doi = {10.1016/j.cma.2020.113228},
  langid = {english},
}

@ARTICLE{FBB25,
  title     = "Calibration of parameter perturbations for ensemble prediction using data-consistent inversion",
  author    = "Fleury, A. and Bouttier, F. and Bergot, T.",
  journal   = "Monthly Weather Review",
  publisher = "American Meteorological Society",
  volume    =  "153",
  number    =  "4",
  pages     = "655--672",
  month     =  "April",
  year      =  2025
}

@article{BGW2020,
	title = {Data-{Consistent} {Solutions} to {Stochastic} {Inverse} {Problems} {Using} a {Probabilistic} {Multi}-{Fidelity} {Method} {Based} on {Conditional} {Densities}},
	volume = {10},
	issn = {2152-5080, 2152-5099},
	url = {https://www.dl.begellhouse.com/journals/52034eb04b657aea,4c71a7d5794d3085,763087fa0bbe44dd.html},
	doi = {10.1615/Int.J.UncertaintyQuantification.2020030092},
	language = {English},
	number = {5},
	urldate = {2023-11-21},
	journal = {International Journal for Uncertainty Quantification},
	author = {Bruder, L. and Gee, M.W. and Wildey, T.},
	year = {2020}
}

@article{tran2021solving,
  title={Solving stochastic inverse problems for property--structure linkages using data-consistent inversion and machine learning},
  author={Tran, A. and Wildey, T.},
  journal={JOM},
  volume={73},
  number={1},
  pages={72--89},
  year={2021},
  publisher={Springer}
}

@article{RPK+23,
author = {Rumbell, T.  and Parikh, J.  and Kozloski, J.  and Gurev, V.},
title = {Novel and flexible parameter estimation methods for data-consistent inversion in mechanistic modelling},
journal = {Royal Society Open Science},
volume = {10},
number = {11},
pages = {230668},
year = {2023},
doi = {10.1098/rsos.230668}
}

@INPROCEEDINGS{ZM2023,
  author={Zhang, Y. and Mikelsons, L.},
  booktitle={2023 IEEE/ASME International Conference on Advanced Intelligent Mechatronics (AIM)}, 
  title={{Solving Stochastic Inverse Problems with Stochastic BayesFlow}}, 
  year={2023},
  volume={},
  number={},
  pages={966-972},
  doi={10.1109/AIM46323.2023.10196190}
}

@article{PR2000,
author = {D. Poole and A.E. Raftery},
title = {{Inference for Deterministic Simulation Models: The {B}ayesian Melding Approach}},
journal = {Journal of the American Statistical Association},
volume = {95},
number = {452},
pages = {1244-1255},
year = {2000},
publisher = {Taylor & Francis},
doi = {10.1080/01621459.2000.10474324}
}

@article{PMS+14,
author = {Petra, N. and Martin, J. and Stadler, G. and Ghattas, O.},
title = {{A Computational Framework for Infinite-Dimensional {B}ayesian Inverse Problems, Part {II}: Stochastic Newton MCMC with Application to Ice Sheet Flow Inverse Problems}},
journal = {SIAM Journal on Scientific Computing},
volume = {36},
number = {4},
pages = {A1525-A1555},
year = {2014},
doi = {10.1137/130934805}
}

@article{Burger_2014,
	doi = {10.1088/0266-5611/30/11/114004},
	url = {https://doi.org/10.1088%2F0266-5611%2F30%2F11%2F114004},
	year = 2014,
	month = {Oct},
	publisher = {{IOP} Publishing},
	volume = {30},
	number = {11},
	pages = {114004},
	author = {M. Burger and F. Lucka},
	title = {{Maximum a posteriori estimates in linear inverse problems with log-concave priors are proper Bayes estimators}},
	journal = {Inverse Problems}
}

@article{BMP+1994,
author="Berger, J.O.
and Moreno, E.
and Pericchi, L.R.
and Bayarri, M.J.
and Bernardo, J.M.
and Cano, J.A.
and De la Horra, J.
and Mart{\'i}n, J.
and R{\'i}os-Ins{\'u}a, D.
and Betr{\`o}, B.
and Dasgupta, A.
and Gustafson, P.
and Wasserman, L.
and Kadane, J.B.
and Srinivasan, C.
and Lavine, M.
and O'Hagan, A.
and Polasek, W.
and Robert, C.P.
and Goutis, C.
and Ruggeri, F.
and Salinetti, G.
and Sivaganesan, S.",
title="An overview of robust {B}ayesian analysis",
journal="Test",
year="1994",
volume="3",
number="1",
pages="5--124",
issn="1863-8260",
doi="10.1007/BF02562676",
url="http://dx.doi.org/10.1007/BF02562676"
}

@article {KO2001,
author = {Kennedy, Marc C. and O'Hagan, Anthony},
title = {{B}ayesian calibration of computer models},
journal = {Journal of the Royal Statistical Society: Series B (Statistical Methodology)},
volume = {63},
number = {3},
publisher = {Blackwell Publishers Ltd.},
issn = {1467-9868},
url = {http://dx.doi.org/10.1111/1467-9868.00294},
doi = {10.1111/1467-9868.00294},
pages = {425--464},
year = {2001},
}

@article{Fitzpatrick1991,
  author={B. Fitzpatrick},
  title={{B}ayesian analysis in inverse problems},
  journal={Inverse Problems},
  volume={7},
  number={5},
  pages={675},
  url={http://stacks.iop.org/0266-5611/7/i=5/a=003},
  year={1991}
}

@article{CKS14,
  author={D. Calvetti and J. Kaipio and E. Somersalo},
  title={Inverse problems in the {B}ayesian framework},
  journal={Inverse Problems},
  volume={30},
  number={11},
  pages={110301},
  url={http://stacks.iop.org/0266-5611/30/i=11/a=110301},
  year={2014}
}

@book{Gelman2013,
	AUTHOR	= {A.~Gelman and J.~B.~Carlin and H.~S.~Stern and D.~B.~Dunson and A.~Vehtari and D.~B.~Rubin},
	TITLE		= {{{B}ayesian Data Analysis, Third Edition}},
	publisher = {Chapman and Hall/CRC},
	YEAR		= {2013}
}

@article{BBE24,
   author = "Bingham, D. and Butler, T. and Estep, D.",
   title = "Inverse Problems for Physics-Based Process Models", 
   journal= "Annual Review of Statistics and Its Application",
   year = "2024",
   volume = "11",
   pages = "461-482",
   publisher = "Annual Reviews",
   issn = "2326-831X"
}

@article{PCT+23,
title = {Parameter estimation with maximal updated densities},
journal = {Computer Methods in Applied Mechanics and Engineering},
volume = {407},
pages = {115906},
year = {2023},
issn = {0045-7825},
doi = {https://doi.org/10.1016/j.cma.2023.115906},
url = {https://www.sciencedirect.com/science/article/pii/S0045782523000294},
author = {M. Pilosov and C. del-Castillo-Negrete and T.Y. Yen and T. Butler and C. Dawson}
}

@article{RHB25,
title = {A machine-learning enabled framework for quantifying uncertainties in parameters of computational models},
journal = {Computers \& Mathematics with Applications},
volume = {182},
pages = {184-212},
year = {2025},
issn = {0898-1221},
doi = {https://doi.org/10.1016/j.camwa.2025.01.030},
url = {https://www.sciencedirect.com/science/article/pii/S0898122125000367},
author = {T. Roper and H. Hakula and T. Butler}
}

@article{MSB+22,
	author = {S.A. Mattis and K.R. Steffen and T. Butler and C.N. Dawson and D. Estep},
	doi = {https://doi.org/10.1016/j.cma.2021.114230},
	issn = {0045-7825},
	journal = {Computer Methods in Applied Mechanics and Engineering},
	pages = {114230},
	title = {Learning Quantities of Interest from dynamical systems for observation-consistent inversion},
	url = {https://www.sciencedirect.com/science/article/pii/S0045782521005582},
	volume = {388},
	year = {2022}}

@article{BJW18b,
author = {Butler, T. and Jakeman, J. and Wildey, T.},
title = {Convergence of Probability Densities Using Approximate Models for Forward and Inverse Problems in Uncertainty Quantification},
journal = {SIAM Journal on Scientific Computing},
volume = {40},
number = {5},
pages = {A3523-A3548},
year = {2018},
doi = {10.1137/18M1181675},

URL = { 
        https://doi.org/10.1137/18M1181675
    
},
eprint = { 
        https://doi.org/10.1137/18M1181675
    
}
}

@article{BJW18a,
author = {T. Butler and J. Jakeman and T. Wildey},
title = {Combining Push-Forward Measures and {B}ayes' Rule to Construct Consistent Solutions to Stochastic Inverse Problems},
journal = {SIAM Journal on Scientific Computing},
volume = {40},
number = {2},
pages = {A984-A1011},
year = {2018},
doi = {10.1137/16M1087229},

URL = { 
        https://doi.org/10.1137/16M1087229
    
},
eprint = { 
        https://doi.org/10.1137/16M1087229
    
}

}

@ARTICLE{APS+16,
author = {A. Alexandrian and N. Petra and G. Stadler and O. Ghattas},
title = {{A Fast and Scalable Method for A-Optimal Design of Experiments for Infinite-dimensional {B}ayesian Nonlinear Inverse Problems}},
journal = {SIAM Journal on Scientific Computing},
volume = {38},
issue = {1},
pages = {A243--A272},
year = {2016},
doi = {10.1137/140992564},
}

@article{CDS10,
author = "S.L. Cotter and M. Dashti and A.M. Stuart",
title = {{Approximation of {B}ayesian Inverse Problems}},
journal = "SIAM Journal of Numerical Analysis",
volume = "48",
year = "2010",
pages = "322-345",
}

@article{BET+14,
	Author = {Butler, T. and Estep, D. and Tavener, S. and Dawson, C. and Westerink, J.J.},
	Journal = {SIAM Journal on Uncertainty Quantification},
	Pages = {1--27},
	Title = {{A Measure-Theoretic Computational Method For Inverse Sensitivity Problems III: Multiple Quantities of Interest}},
	Year = {2014}}

@article{KL1951,
  author = "S. Kullback and R.A. Leibler",
  title  = "On Information and Sufficiency",
  journal= "The Annals of Mathematical Statistics",
  volume = "22",
  year   = "1951",
  pages  = "79--86"
}

@article{RenyiKL2014,
  author = "T. van Erven and P. Harremoes",
  title  = {{Renyi Divergence and Kullback-Leibler Divergence}},
  journal= "IEEE Transactions on Information Theory",
  volume = "60",
  year   = "2014",
  pages  = "3797--3820"
}

@misc{mrhyde2023github,
      title = "{MrHyDE} library",
      author = {Wildey, T.},
      url = "https://github.com/sandialabs/MrHyDE",
      howpublished = {\url{https://github.com/sandialabs/MrHyDE}},
      note = {Accessed: 09/25/2024}
}

@misc{mrhyde_user,
	Author = {Wildey, T.},
	Month = {May},
        Note = {SAND2024-06292}, 
	Organization = {Sandia National Labs},
	Title = {{User/Reference Guide for MrHyDE - A framework for solving Multi-resolution Hybridized Differential Equations -- Version 1.0}},
	Year = 2024,
	}

@article{Csiszar1975,
author = {I. Csiszar},
title = {{$I$-Divergence Geometry of Probability Distributions and Minimization Problems}},
volume = {3},
journal = {The Annals of Probability},
number = {1},
publisher = {Institute of Mathematical Statistics},
pages = {146 -- 158},
year = {1975},
doi = {10.1214/aop/1176996454},
URL = {https://doi.org/10.1214/aop/1176996454}
}

@article{Ruschendorf1995,
author = {L. Ruschendorf},
title = {{Convergence of the Iterative Proportional Fitting Procedure}},
volume = {23},
journal = {The Annals of Statistics},
number = {4},
publisher = {Institute of Mathematical Statistics},
pages = {1160 -- 1174},
year = {1995},
doi = {10.1214/aos/1176324703},
URL = {https://doi.org/10.1214/aos/1176324703}
}

@article{alter_min,
author = {Gunawardana, A. and Byrne, W.},
title = {Convergence Theorems for Generalized Alternating Minimization Procedures},
year = {2005},
issue_date = {12/1/2005},
publisher = {JMLR.org},
volume = {6},
issn = {1532-4435},
journal = {J. Mach. Learn. Res.},
month = dec,
pages = {2049–2073},
numpages = {25}
}

@article{csiszar1984,
  title={Information geometry and alternating minimization procedures},
  author={Csisz\'ar, I. and Tusn\'ady, G.},
  journal={Statistics and Decisions},
  year={1984},
  volume={Supplemental Issue Number 1},
  pages={205-237}
}

@article{changandpollard1997,
author = {Chang, J.T. and Pollard, D.},
title = {Conditioning as disintegration},
journal = {Statistica Neerlandica},
volume = {51},
number = {3},
pages = {287-317},
doi = {https://doi.org/10.1111/1467-9574.00056},
url = {https://onlinelibrary.wiley.com/doi/abs/10.1111/1467-9574.00056},
eprint = {https://onlinelibrary.wiley.com/doi/pdf/10.1111/1467-9574.00056},
year = {1997}
}

@article{Guiasu1985,
  author    = {S. Guiasu and A. Shenitzer},
  title     = {The principle of maximum entropy},
  journal   = {The Mathematical Intelligencer},
  year      = {1985},
  volume    = {7},
  number    = {1},
  pages     = {42--48},
  doi       = {10.1007/BF03023004},
  url       = {https://doi.org/10.1007/BF03023004},
  issn      = {0343-6993}
}

@book{polyanskiy2022information,
  title={Information Theory: From Coding to Learning},
  author={Polyanskiy, Y. and Wu, Y.},
  note={Draft of October 20, 2022},
  publisher={Cambridge University Press},
  year={2022},
}

@article{Marino2020,
  title = {{An Optimal Transport Approach for the Schr\"{o}dinger Bridge Problem and Convergence of Sinkhorn Algorithm}},
  volume = {85},
  ISSN = {1573-7691},
  url = {http://dx.doi.org/10.1007/s10915-020-01325-7},
  DOI = {10.1007/s10915-020-01325-7},
  number = {2},
  journal = {Journal of Scientific Computing},
  publisher = {Springer Science and Business Media LLC},
  author = {Marino,  S.D. and Gerolin,  A.},
  year = {2020},
  month = oct 
}
\end{document}